%% file: main.tex
\documentclass[11pt]{book}

\usepackage[b5paper,left=2.25cm, right=2.25cm, top=3cm, bottom=2cm]{geometry}
\usepackage{emptypage}

\usepackage[utf8]{inputenc}
\usepackage[english]{babel}

\usepackage{amsmath,amsfonts,amssymb,amscd,amsthm,yhmath}

\usepackage{enumerate}

\usepackage{verbatim}

\usepackage[dvipsnames]{xcolor}

\usepackage{pb-diagram,pb-xy}
\usepackage[cmtip,all]{xy}
\usepackage{graphicx}

\usepackage[colorlinks,linkcolor=blue,citecolor=blue,urlcolor=blue]{hyperref}

\usepackage{fancyhdr}
\usepackage{float}
\usepackage{hhline}
\usepackage{framed}

\theoremstyle{plain}
\newtheorem{theo}{Theorem}[chapter]
\newtheorem{coro}[theo]{Corollary}
\newtheorem{lemm}[theo]{Lemma}
\newtheorem{prop}[theo]{Proposition}
\newtheorem{epiprop}{Proposition}[]  
\newtheorem{epilemm}[epiprop]{Lemma} 
\newtheorem*{theointro}{Theorem}

\theoremstyle{definition}
\newtheorem{defi}[theo]{Definition}
\newtheorem{exam}[theo]{Example}
\newtheorem{rema}[theo]{Remark}

\newcommand{\ZZ}{\mathbb{Z}}
\newcommand{\QQ}{\mathbb{Q}}
\newcommand{\RR}{\mathbb{R}}
\newcommand{\CC}{\mathbb{C}}
\newcommand{\CP}{\mathbb{P}}

\newcommand{\CA}{\mathcal{A}}
\newcommand{\BB}{\mathcal{B}}
\newcommand{\CH}{\mathcal{C}}
\newcommand{\DC}{\mathcal{D}}
\newcommand{\EE}{\mathcal{E}}

\newcommand{\GG}{\mathcal{G}}

\newcommand{\LL}{\mathcal{L}}
\newcommand{\MM}{\mathcal{M}}
\newcommand{\PP}{\mathcal{P}}
\newcommand{\OO}{\mathcal{O}}

\newcommand{\WW}{\mathcal{W}}

\newcommand{\al}{\alpha}
\newcommand{\be}{\beta}
\newcommand{\ga}{\gamma}
\newcommand{\de}{\delta}
\newcommand{\ep}{\varepsilon}
\newcommand{\vp}{\varphi}
\newcommand{\ka}{\kappa}
\newcommand{\io}{\iota}
\newcommand{\lm}{\lambda}
\newcommand{\sig}{\sigma}
\newcommand{\te}{\theta}
\newcommand{\om}{\omega}
\newcommand{\ze}{\zeta}

\newcommand{\Ga}{\Gamma}
\newcommand{\Om}{\Omega}
\newcommand{\De}{\Delta}

\newcommand{\Lm}{\Lambda}
\newcommand{\Sig}{\Sigma}

\newcommand{\ot}{\otimes}
\newcommand{\ov}{\overline}
\newcommand{\STS}{S^1\times S^1}
\newcommand{\MTM}{M\times M}
\newcommand{\MTN}{M\times N}
\newcommand{\NTN}{N\times N}
\newcommand{\GTG}{G\times G}

\newcommand{\sres}{\mathcal{S}_{(h,\ell,P,\eta,\gamma)}}

\newcommand{\sq}{/\kern-.7ex/}  
\newcommand{\ssq}{/\kern-.4ex/} 

\newcommand{\im}{\mbox{im}\,}

\newcommand{\bb}{{\bf b}}
\newcommand{\na}{\nabla}
\newcommand{\mg}{\mathfrak{g}}

\newcommand{\la}{\langle}
\newcommand{\ra}{\rangle}

\begin{document}

\frontmatter 


\input{./titlepage}


\newpage\thispagestyle{empty}\null 


\newpage
\thispagestyle{empty}
\null\vfil
\begin{center}{\Large A les estrelles fugaces del cel de Siurana}\end{center}
\vfil\null
\newpage\null\thispagestyle{empty} 


\newpage
\addcontentsline{toc}{chapter}{Acknowledgements}
\pagestyle{plain}

\begin{center}{\LARGE{\textit{Agraïments/Acknowledgements/Agradecimentos}} \par}\end{center}
\input{./acknow}

\tableofcontents

\listoffigures
\addcontentsline{toc}{chapter}{List of Figures}
\newpage\null\thispagestyle{empty} 


\newpage
\thispagestyle{empty}

\null\vfill 

\begin{flushright}


\textit{``O tempo, ainda que os relógios queiram convencer-nos\\ do contrário, não é o mesmo para toda a gente."}

José Saramago, \textit{Todos os nomes}
\end{flushright}

\vfill\vfill\vfill\vfill\vfill\vfill\null 

\newpage\null\thispagestyle{empty} 


\pagestyle{fancy} 

\input{./Preamble}

\input{./Introduction}

\mainmatter 

\pagestyle{headings}
\renewcommand{\chaptermark}[1]{ \markboth{\thechapter.\ #1}{} }
\renewcommand{\sectionmark}[1]{ \markright{\thesection.\ #1}{} }

\input{./Toolbox}

\input{./Standard}
\input{./Multivalued}

\input{./Biinvariant}

\pagestyle{fancy}

\input{./Epilogue}


\end{document}

%% file: titlepage.tex
\begin{titlepage}
\begin{center}

\rule{\linewidth}{0.5mm} \\[0.4cm] 
{\huge \bfseries Construction of right inverses \\ of the Kirwan map}\\[0.4cm] 
\rule{\linewidth}{0.5mm} \\[1.5cm] 

\begin{minipage}{0.4\textwidth}
\begin{flushleft} \large
\emph{Autor:}\\
Andratx Bellmunt Giralt\\
\href{mailto:andratx.bellmunt@gmail.com}{andratx.bellmunt@gmail.com}
\end{flushleft}
\end{minipage}
\begin{minipage}{0.4\textwidth}
\begin{flushright} \large
\emph{Director:} \\
Dr. Ignasi Mundet i Riera\\
\hfill
\end{flushright}
\end{minipage}\\[2cm]

\vfill

\textit{Tesi presentada per a optar \\ al grau de Doctor en Matemàtiques}\\[0.3cm]
Novembre 2015

\vfill\vfill

Departament d'Àlgebra i Geometria\\[0.5cm]
Facultat de Matemàtiques\\[0.5cm]

\hspace{0.7cm}\includegraphics[scale=0.4]{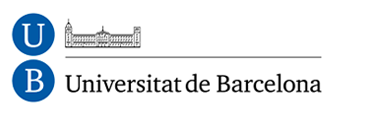}\\

\end{center}
\end{titlepage}

%% file: acknow.tex
Al meu director de tesi, el Dr. Ignasi Mundet, vull agrair-li, a més de la seva paciència i comprensió --sé que sovint no li ho he posat fàcil--, que em proposés d'estudiar un tema ben interessant: quan t'adones que has hagut de fer servir anàlisi per construir geomètricament aquell objecte topològic que et permet trobar l'objecte algebraic que busques, saps que has estat fent matemàtiques.

I would also like to thank Professor Frances Kirwan for her kindness and hospitality. It is a great privilege to be able to work directly with the person who developed the theory on which your dissertation is based upon. During my stay at the University of Oxford I shared many mathematical ideas and many beer pints with Maria, Michael, Markus, Rob, Pádraig, Tim and Thomas.

No Instituto Superior Técnico fui muito bem recebido pelo Professor Rui Loja Fernandes. Mais nunca tenho conhecido alguém com a capacidade dele para conseguir tirar tempo do nada; estou-lhe muito agradecido por isso. Diversos postdoutorandos foram grandes anfitriões: o Rémi, o Gabriele, o Sebastian e particularmente o meu companheiro de gabinete, o Filippo. O responsável da minha estimação por Lisboa e as minhas múltiples visitas posteriores na cidade é o pessoal do Intendente: o tio Jared, as tias Claudia e Carla, os primos Héctor e Teresa, mas muito especialmente a tia Paula; és a melhor, simplesmente.

La part més gran de la tesi, però, s'ha desenvolupat a Catalunya. La meva vida com a matemàtic no s'entendria sense l'Alberto, amb qui he compartit tantes coses des que els trens de Sant Esteve i Tarragona ens van fer arribar abans que ningú a la primera classe de la llicenciatura. Les converses freqüents amb la Joana al balconet del departament van ser un gran ajut; amb ella vam descobrir que de matinada a París els clients dels bars se't posen a parlar de Grothendieck amb tota normalitat. Els altres {\it Bott-Tu}'s, l'Abdó i el Federico, han estat grans companys a l'hora de resoldre dubtes topològics i geomètrics. L'entusiasme i la constància del David i el Dani van fer possible endegar el seminari SIMBa i que fos tot un èxit; he tingut el plaer d'organitzar-lo amb ells dos i amb la Mireia, el Carlos, l'Arturo i l'Eloi. Els assistents incondicionals del seminari han estat clau perquè tota una generació i mitja de doctorands de la facultat ens coneguéssim més tant a nivell matemàtic com a nivell personal. Una menció especial és per a la cosineta Meri, amb qui vam compartir dinars, despatx, monedes de xocolata i problemes de Geometria projectiva. Finalment, al Carles Currás vull agrair-li tot el marge que em va donar per fer i desfer en les classes de l'assignatura de Geometria diferencial.

En els darrers mesos he estat gaudint d'un fantàstic ambient de treball a Gauss \& Neumann, on la majoria dels companys també han tingut l'experiència d'escriure una tesi doctoral en Matemàtiques o Física. Tots ells són una font de bons consells: moltes gràcies a l'Alberto, la Sònia, la Jessica, l'Ignasi, el Marc, el Zorion, la Federica, el Rubén i la Sara.

Fora de l'àmbit matemàtic són molts els companys i amics que m'han sentit parlar durant temps i temps d'aquesta tesi. El Ricard, la Carlota, la Lara, el Javi i el Christian sempre m'esperen amb els braços oberts les vegades --moltes menys de les que voldria-- que torno a Cambrils. Els dos Alberts, el Joan, el Pere Joan, el Néstor, el Tomeu, el Carles i la Mireia, tot i la diàspora, sempre tindrem un punt de trobada al {\it Bar Bodega Javier}. El Cerni i la seva borda fan que veure's un cop l'any sigui apoteòsic. El Pep i l'Oriol són la garantia que mai arribi a tocar massa de peus a terra i que sempre tingui l'oportunitat d'emprendre el vol. Dels anys del Ramon Llull encara en queden molts altres amics -amb alguns dels quals he compartit pis- que sempre m'han fet costat: l'Aleix, el Borja, l'Elisenda, la Paz, el Marc i l'Eli. Els companys i el {\it sifu} de kung-fu m'han fet suar de valent, però sobretot han tingut una constància impertorbable interessant-se a cada entrenament per l'evolució de la tesi.

Els matemàtics d'aquesta universitat tenim el goig de compartir espai amb els nostres veïns filòlegs. En el meu cas, aquest goig se sublima en l'Adelaida.

Tots els membres de la meva família sense excepció m'han donat un enorme suport, sempre han estat disposats a ajudar-me i a empènyer-me per tirar endavant. El Pere i l'Èlia, pare i germana, en són els màxims representants. La Fina, la meva mare, em va ensenyar amb els seus actes, però sobretot amb el seus sacrificis, fins a quin punt valen la pena uns estudis universitaris. L'hauria entusiasmada poder veure aquesta tesi completada.

%% file: Preamble.tex
\chapter{Preamble}\label{ch:preamble}

\fancyhead[RE,LO]{Preamble}


This preamble is the only variation from the text that was originally submitted as a requirement for the author to get his PhD degree. The jury of the corresponding public defence -that took place on 17\textsuperscript{th} December 2015- pointed out that, although it was not a serious issue, it would have been better not to assume that the reader was familiar with Hamiltonian spaces. In order to offset this issue we add this preamble. Doing so the original introduction remains unmodified. 

We are about to explain what Hamiltonian spaces are beginning from the very definition of a symplectic manifold. Very good references for this topic are \cite{Can} and \cite{McSa1}. The only background needed are smooth manifolds, differential forms and Lie groups.

A {\it symplectic manifold} is a pair $(M,\om)$ formed by:
\begin{itemize}
\item A smooth manifold $M$.
\item A differential 2-form $\om\in\Om^2(M)$ which is closed and non-degenerate. We say that $\om$ is a {\it symplectic form}.
\end{itemize}
Symplectic manifolds arise naturally in many different contexts such as Hamiltonian Classical Mechanics or as projective varieties over $\CC$. 

To our purposes we will assume that $(M,\om)$ is a compact and connected symplectic manifold. Let $G$ be a Lie group, which we also assume to be compact and connected. Let $\mg$ be the Lie algebra of $G$ and let $\exp:\mg\rightarrow G$ be the corresponding exponential map. 

If $G$ acts smoothly on $M$, for every $\xi\in\mg$ we get a vector field, which on $p\in M$ is defined by
\[
X_p^\xi=\left.\frac{d}{dt}\right|_{t=0}\exp(t\xi)\cdot p.
\]
The action is {\it symplectic} if $X^\xi$ preserves the symplectic form, i.e. $\LL_{X^\xi}\om=0$, for every $\xi\in\mg$. Then, from the Cartan formula
\[
\LL_{X^\xi}\om=d\io_{X^\xi}\om+\io_{X^\xi}d\om
\] 
and $d\om=0$ ($\om$ is closed) it follows that $d\io_{X^\xi}\om=0$, so the contraction $\io_{X^\xi}\om$ is a closed $1$-form. If, moreover, $\io_{X^\xi}\om$ is exact we say that the action is {\it Hamiltonian}. In this situation there exists a function $h^\xi:M\rightarrow\RR$ such that $dh^{\xi}=\io_{X^\xi}\om$. From this we can define the correspondence
\[
\begin{array}{rccc}
\mu: & M & \longrightarrow & \mg^*\\
 &p&\longmapsto &(\xi\mapsto h^\xi(p))
\end{array},
\]
which is called {\it moment map}. The moment map turns out to be $G$-equivariant with respect to the coadjoint action on $\mg^*$.

We say that $(M,\om,G,\mu)$ is a {\it Hamiltonian space}.

%% file: Introduction.tex
\chapter{Introduction}\label{ch:introduction}

\fancyhead[RE,LO]{Introduction}


The objects studied in this thesis are Hamiltonian spaces and it is assumed that the reader is familiar with them as explained in e.g. \cite[VII-IX]{Can} or \cite[\S 5]{McSa1}. The usual contents of first year graduate courses in differential geometry and algebraic topology are also assumed. We did our best effort to explain or to provide useful references to anything beyond that, including some of the notions that appear in this introduction.

Let $(M,\om,G,\mu)$ be a Hamiltonian space where both $M$ and $G$ are compact and connected. Since the moment map $\mu:M\rightarrow\mathfrak{g}^*$ is equivariant with respect to the coadjoint action on $\mathfrak{g^*}$, the action on $M$ restricts to $\mu^{-1}(0)$. If $0$ is a regular value of $\mu$ this action is locally free and $M\sq G:=\mu^{-1}(0)/G$ is an orbifold which inherits a symplectic structure from $M$. In {\bf rational} cohomology there is an isomophism $H_G^*(\mu^{-1}(0))\simeq H^*(M\sq G)$, called the Cartan isomorphism. Also, the inclusion $\mu^{-1}(0)\hookrightarrow M$ induces a ring homomorphism $H_G^*(M)\rightarrow H_G^*(\mu^{-1}(0))$ in $G$-equivariant cohomology. Composing it with the Cartan isomorphism we get the {\it Kirwan map},
\[
\ka:H_G^*(M)\longrightarrow H^*(M\sq G).
\]
In \cite{Kir} Kirwan shows that $\ka$ is a surjective map. A brief sketch of her proof goes as follows: take a $G$-invariant inner product on $\mathfrak{g}$ and use it to identify $\mathfrak{g}$ and $\mathfrak{g}^*$. Then $\mu$ can be thought as taking values on $\mathfrak{g}$. Using the norm defined by the inner product we can define a $G$-invariant smooth function $f:M\rightarrow\RR$ by $f(p):=\|\mu(p)\|^2$. This function has the property that $\mu^{-1}(0)$, which equals $f^{-1}(0)$, is its minimum level. The function $f$ is not Morse-Bott in general, but it is minimally degenerate --as described in \cite[\S10]{Kir}-- which is sufficient to prove that the stable manifolds of the gradient flow of $f$ with respect to an invariant Riemannian metric form, in a sense, a stratification of $M$ (see \cite[2.11]{Kir}). Using this stratification Kirwan shows that $f$ is equivariantly perfect, meaning that if $C$ is a component of the critical set of $f$ and $M_{\pm}=f^{-1}(-\infty,f(C)\pm\ep)$, for sufficiently small $\ep>0$ the equivariant cohomology exact sequence of the pair $M_+,M_-$,
\[
\cdots\rightarrow H^*_G(M_+,M_-)\rightarrow H^*_G(M_+)\rightarrow H^*_G(M_-)\rightarrow\cdots
\]
splits into short exact sequences. Then, by induction on the critical levels, it follows that $H_G^*(M)\rightarrow H_G^*(f^{-1}(0))$ is surjective. Since $f^{-1}(0)=\mu^{-1}(0)$ this map is nothing else than the map induced by the inclusion $\mu^{-1}(0)\hookrightarrow M$, and surjectivity of $\ka$ follows.

Kirwan's proof does not define, neither explicitly nor implicitly, a particular choice of right inverse of $\ka$. We now describe a geometric construction that defines a right inverse of $\ka$. Let $2m=\dim M$, $d=\dim G$, consider the action of $\GTG$ on $\MTM$ given by $(g,h)(p,q)=(g\cdot p,h\cdot q)$ and let
\[
\ka^2:H_{\GTG}^*(\MTM)\longrightarrow H^*(M\sq G\times M\sq G)
\]
be the Kirwan map for this action. A class $\de\in H_{\GTG}^{2(m-d)}(\MTM)$ such that $\ka^2(\de)$ is Poincaré dual to the homology class $[\De_{M\ssq G\times M\ssq G}]\in H_{2(m-d)}(M\sq G\times M\sq G)$ defined by the diagonal is called a {\it biinvariant diagonal class}. Using Künneth on both the source and the target spaces of $\ka^2$ we can interpret $\ka^2$ as $\ka\ot\ka$. The image of the biinvariant diagonal class $\de$ under the map
\[
\multiply\dgARROWLENGTH by2
\begin{diagram}
\node{\bigoplus_{k=0}^{2(m-d)}H^{2(m-d)-k}_G(M)\ot H_G^k(M)}\arrow{e,t}{(PD\circ\ka)\ot id}\node{\bigoplus_{k=0}^{2(m-d)}H^k(M\sq G)^\vee\ot H_G^k(M)}
\end{diagram}
\]
can be interpreted as a degree preserving linear map $\ell_\de: H^*(M\sq G)\rightarrow H^*_G(M)$. Here $PD:H^*(M\sq G)\rightarrow H^{2(m-d)-*}(M\sq G)^\vee$ denotes Poincaré duality. It turns out that $\ell_\de$ is a right inverse of $\ka$: if $\de=\sum\ep_i\ot\eta_i$ is any decomposition and $\al\in H^*(M\sq G)$ we have that
\[
\ell_\de(\al)=\sum\left(\int_{M\ssq G}\al\smile\ka(\ep_i)\right)\eta_i
\]
and hence
\[
(\ka\circ\ell_\de)(\al)=\sum\left(\int_{M\ssq G}\al\smile\ka(\ep_i)\right)\ka(\eta_i)=\al,
\]
where the second equality is given by the fact that $\sum\ka(\ep_i)\ot\ka(\eta_i)$ is Poincaré dual to the class defined by the diagonal in $M\sq G\times M\sq G$. The lack of uniqueness of a right inverse of $\ka$ is reflected in this construction in the fact that in general there exist many different biinvariant diagonal classes. On the other hand, if we manage to give a direct definition of some biinvariant diagonal class $\de$ then, by the preceding argument, we obtain a new proof of Kirwan surjectivity and at the same time we get a ``natural'' choice of right inverse of $\ka$.

When $G$ is abelian the coadjoint action is trivial and then symplectic reduction $M\sq_cG=\mu^{-1}(c)/G$ can be taken at any regular value $c$ of $\mu$. We get the corresponding Kirwan map $\ka_c: H_G^*(M)\rightarrow H^*(M\sq_c G)$. A class that is biinvariant diagonal for all regular values simultaneously is called a {\it global biinvariant diagonal class}. In \cite{Mun} a global biinvariant diagonal class is constructed geometrically using the moduli space of gradient flow lines of each component of $\mu$.  It turns out that to do so it is sufficient to focus on the case $G=S^1$, the result for a general abelian Lie group following from that case. Note that a great advantage of working with $G=S^1$ is that $\mu$ itself is a Morse-Bott function. Let us explain the the geometric idea behind the construction in \cite{Mun}:

Denote by $X$ the vector field generated by the infinitesimal action of $S^1$ on $M$ and take $J$ an invariant almost complex structure compatible with $\om$. Then $\rho_J=\om(\cdot,J(\cdot))$ is a Riemannian metric on $M$ and $-JX$ is the downward gradient vector field of $\mu$ with respect to this metric. Finally let $\xi^{JX}$ denote the flow of $-JX$. The submanifold of $\MTM$ defined by 
\[
\De_{S^1}=\{(p,q)\in\MTM:\ \exists t\in\RR,\ \al\in S^1\ \mbox{s.t.}\ q=\al\cdot\xi^{JX}_t(p)\}
\]
satisfies the following properties:
\begin{enumerate}
\item It is $\STS$-invariant.
\item It has dimension $2m+2$.
\item The intersection $\De_{S^1}\cap(\mu^{-1}(c)\times\mu^{-1}(c))$ is the preimage of the diagonal under the projection $\mu^{-1}(c)\times\mu^{-1}(c)\rightarrow M\sq_cS^1\times M\sq_cS^1$ for every regular value $c$ of $\mu$.
\end{enumerate}

Using multivalued perturbations --more on that in a moment--  $\De_{S^1}$  is compactified and taking intersection pairing with this compactification one gets a cohomology class
\[
\de_{S^1}\in H^{2m-2}_{\STS}(\MTM).
\]
Property 1 allows to define an equivariant cohomology --instead of ordinary cohomology-- class. Property 2 makes the class have the right degree. Property 3  makes $\de_{S^1}$  to be mapped to the Poincaré dual of $[\De_{M\ssq_cS^1\times M\ssq_cS^1}]$ under $\ka_c^2$. This shows that $\de_{S^1}$ is a global biinvariant diagonal class.

This thesis should be regarded as a continuation of the work carried out in \cite{Mun}. The main tool to construct the global biinvariant diagonal class are multivalued perturbations of the gradient flow equation: to properly compactify $\De_{S^1}$ some transversality conditions on the (un)stable manifolds defined by the flow $\xi^{JX}$ are needed but they are impossible to achieve in general if we require invariance --property 1-- to hold after compactification. However, the only obstruction is due to the presence of finite non-trivial stabilisers of the action of $S^1$ on $M$; if the action is semi-free standard perturbations are sufficient to get transversality and invariance simultaneously. In the general situation where finite non-trivial stabilisers do exist, we need to make use of multivalued perturbations to achieve the same result. Although multivalued perturbations are not a new concept, to our knowledge they were not described in detail for the gradient flow equation before \cite{Mun}. One of the purposes of the present text is to describe them at a slower pace with the humble hope that this work can be useful to anyone with the wish of learning this topic. After some necessary tools are given in Chapter \ref{ch:toolbox}, we devote Chapter \ref{ch:standard} to standard perturbations, which provide a benchmark for Chapter \ref{ch:multivalued}, where multivalued perturbations are presented. 

Chapter \ref{ch:biinvariant} contains the new results of this thesis. Our first main contribution is the geometric construction of a linear map
\[
\boxed{\lm:H_{\STS}^*(\CC P^1)\longrightarrow H^{*+2m-2}_{\STS}(\MTM)}
\]
where $\STS$ acts on $\CC P^1$ by $(\te,\ze)[z:w]=[\te z:\ze w]$ and on $\MTM$ by  $(\te,\ze)(p,q)=(\te\cdot p,\ze\cdot q)$, with the property that $\lm(1)=\de_{S^1}$ and that plays a crucial role in the description of biinvariant diagonal classes of Cartesian products. The geometric idea behind the construction of $\lm$ is the following:

The manifold $\De_{S^1}\subseteq\MTM$ which was defined previously can be identified with the image of the map
\[
\begin{array}{rccc}
F_1: &M\times\RR\times S^1&\longrightarrow &\MTM\\
 &(p,t,\al)&\longmapsto & (p,\al\cdot\xi^{JX}_t(p)) 
\end{array}.
\]
Define also the map
\[
\begin{array}{rccc}
F_2:&M\times\RR\times S^1&\longrightarrow &\CC P^1\\
 &(p,t,\al)&\longmapsto & [1:2^t\al] 
\end{array}.
\]
Both $F_1,F_2$ are equivariant with respect to the action of $\STS$ on $M\times\RR\times S^1$ given by  $(\te,\ze)(p,t,\al)=(\te\cdot p,t,\ze\al\bar{\te})$. Morally speaking $\lm$ is constructed by finding a suitable compactification of $M\times\RR\times S^1$ to which the action of $\STS$ and the maps $F_1,F_2$ extend. Then $\lm$ is set to be an analogue in equivariant cohomology of the map $F_1^!\circ F_2^*$, where $F_1^!=PD^{-1}\circ(F_1)_*\circ PD$ is the shriek map associated to $F_1$. To give a proper sense to this definition of $\lm$ we need to address compactification issues that are dealt with multivalued perturbations (precisely because we want to preserve $\STS$-equivariance). According to this approach $\lm(1)$ should be understood as $F_1^!(1)$, which loosely speaking is an equivariant Poincaré dual of $\im F_1=\De_{S^1}$, so it does make sense that $\lm(1)=\de_{S^1}$. 

In \cite{Mun} no explicit computations of biinvariant diagonal classes are carried out. By further studying the map $\lm$ we get results that give tools to compute global biinvariant diagonal classes in certain situations:

One can show that
\[
H^*_{\STS}(\CC P^1)\simeq \QQ[t_1,t_2,u]/(u-t_1)(u-t_2)
\]
for some degree $2$ classes $u,t_1,t_2$. The classes $t_1,t_2$ are the standard generators of 
\[
H^*(B(\STS))\simeq H^*(BS^1)\ot H^*(BS^1)\simeq\QQ[t_1]\ot\QQ[t_2]=\QQ[t_1,t_2]
\]
and the class $u$ can be identified with the Euler class of the tautological bundle over a certain projective bundle. The second main contribution of this thesis is the following result:

\begin{framed}
\begin{theointro} Let $M,N$ be Hamiltonian $S^1$-spaces and let $\lm^M, \lm^N$ be their respective lambda-maps. Then
\begin{enumerate}
\item $\lm^{\MTN}(1)=(t_1+t_2)\lm^M(1)\ot\lm^N(1)-\lm^M(u)\ot\lm^N(1)-\lm^M(1)\ot\lm^N(u)$
\item $\lm^{\MTN}(u)=t_1t_2\lm^M(1)\ot\lm^N(1)-\lm^M(u)\ot\lm^N(u)$
\end{enumerate}
where $\lm^{\MTN}$ is the lambda-map of the product Hamiltonian space $\MTN$.
\end{theointro}
\end{framed}
In particular, the first formula lets us compute a global biinvariant diagonal class of $\MTN$ in terms of the lambda-maps of its components.

The last question we address in Chapter \ref{ch:biinvariant} is uniqueness of global biinvariant diagonal classes. By means of a very hands-on computation we show that for $M=\CC P^1\times\stackrel{n}{\cdots}\times{\CC P^1}$ and the $S^1$-action given by
\[
\te([z_0:w_0],\ldots,[z_{n-1}:w_{n-1}])=([\te z_0:w_0],\ldots,[\te^{2^{n-1}}z_{n-1}:w_{n-1}]),
\]
global diagonal biinvariant classes form an affine space which has strictly positive dimension in general. 

We finish the thesis with a brief epilogue were we recap the ideas behind some of the results and give some thoughts on what directions can be taken in the future.

%% file: Toolbox.tex
\chapter{Toolkit}\label{ch:toolbox}


This chapter's purpose is to provide some background on three topics: equivariant cohomology, branched manifolds and pseudocycles. It is not our goal to cover them in detail but rather to equip ourselves with a set of tools that will be used in the chapters to follow. Having this idea in mind some core results are stated without proof while some other minor results may receive more attention because they are used later on. Even the readers familiar with these topics are suggested to take at least a brief look to this chapter, because we fix here some notations. On the other hand, any reader who wishes to learn any of the three topics should find here enough content to get a first grasp of them and will be given references were they are treated thoroughly.


\section{Equivariant cohomology}\label{sec:equivcoh}

Throughout this section $G$ will denote a compact and connected Lie group. A smooth manifold $M$ acted by $G$ is called a $G$-{\it manifold}. We will work only with compact, connected and oriented $G$-manifolds. If the action is free, $M/G$ is a smooth manifold, but otherwise it may contain singularities. Let $H^*(\cdot)$ denote singular cohomology with respect to a ring of coefficients that will only be specified when necessary. The idea behind equivariant cohomology is to construct a cohomology $H_G^*(M)$ that captures information of the action. In particular, if the action is free one wants $H_G^*(M)=H^*(M/G)$. The literature is full of excellent references on equivariant cohomology that the reader can consult. Some suggestions are \cite{GuSt} or the lecture notes \cite{Lib}. For the relation between equivariant cohomology and moment maps, good references are \cite{AtBo2} and \cite{GGK}. For a broad and deep treatment of equivariant formality see \cite{GKM}.

Let $EG\rightarrow BG$ be the classifying bundle of $G$, which is a $G$-principal bundle. The group $G$ acts freely on the contractible space $EG$. We can form a new bundle 
\[
M\times_G EG\longrightarrow BG
\]
with fibre $M$.

\begin{defi}The {\it equivariant cohomology} of $M$ is defined to be \[H_G^*(M):=H^*(M\times_G EG).\]\end{defi}

\begin{rema} If the action of $G$ on $M$ is free, then $M\rightarrow M/G$ is also a principal bundle and we can form the bundle $M\times_G EG\rightarrow M/G$ with fibre $EG$. Since the fibres are contractible we get $H_G^*(M)=H^*(M/G)$, which is a  property we wanted equivariant cohomology to satisfy.\end{rema}

The group $G$ has a faithful linear representation, which means that it can be embedded as a Lie subgroup of $U(n)$ for $n$ large enough \cite[III.4.1]{BrTo}. Let $EG_k$ be the space of orthonormal $n$-frames in $\CC^{k+1}$ for $k\geq n$. Then $g\in G\subseteq U(n)$ acts on $v\in EG_k$ by matrix multiplication. The resulting action is free and we denote by $BG_k=EG_k/G$ the corresponding quotient. We then get principal bundles $\pi_k:EG_k\rightarrow BG_k$. The inclusion map 
\[
\begin{array}{rccc}
e_k:&EG_k&\longrightarrow & EG_{k+1}\\
 &(v_0,\ldots,v_k)&\longmapsto&(v_0,\ldots,v_k,0)
\end{array}
\] 
is $G$-equivariant and defining $b_k:BG_k\rightarrow BG_{k+1}$ by $b_k(\pi_k(v)):=\pi_{k+1}(e_k(v))$ we get commutative diagrams 
\[
\begin{diagram}
\node{EG_k}\arrow{e,t}{e_k}\arrow{s,l}{\pi_k}\node{EG_{k+1}}\arrow{s,r}{\pi_{k+1}}\\
\node{BG_k}\arrow{e,b}{b_k}\node{BG_{k+1}}
\end{diagram}.
\]
The direct limit of the principal bundles $\pi_k$ is the classifying bundle of $G$. In this sense we say that the bundles $EG_k\rightarrow BG_k$ {\it approximate} $EG\rightarrow BG$. For a fixed degree $d$ the maps $b_k^d:H^d(BG_{k+1})\rightarrow H^d(BG_k)$ are isomorphisms for all $k$ large enough, so $H^*(BG)$ is the inverse limit of the $H^*(BG_k)$. Then, by a computation on spectral sequences we get the following stabilisation argument, which asserts that $H^*_G(M)$ is the inverse limit of $H^*(M\times_GEG_k)$:

\begin{lemm}[Stabilisation]\label{lemm:stabilisation} Let $M$ be a $G$-manifold and let $EG_k\rightarrow BG_k$ be principal $G$-bundles approximating the classifying bundle $EG\rightarrow BG$. Given a fixed degree $d$, the map
\[
\begin{array}{rccc}
i_k:&M\times_GEG_k&\longrightarrow & M\times_GEG_{k+1}\\
    &\left[(m,v)\right]&\longmapsto&\left[(m,e_k(v))\right]
\end{array}
\]
induces an isomorphism
\[
i_k^d:H^d(M\times_GEG_{k+1})\stackrel{\simeq}{\longrightarrow} H^d(M\times_GEG_k)
\] 
for all $k$ large enough. Thus an element of $H^d_G(M)$ can be defined uniquely from an element of $H^d(M\times_GEG_k)$ provided that $k$ is large enough.
\end{lemm}

\begin{exam}
If $G=S^1$, $ES^1_k$ is just the $(2k+1)$-sphere and the quotient $BS^1_k=S^{2k+1}/S^1=\CC P^k$ is the $k$-dimensional projective space. Then the classifying bundle of $S^1$ is $S^\infty\rightarrow\CC P^\infty$ and $H^*(BS^1;\QQ)\simeq\QQ[t]$, where $t$ can be taken to be the Euler class of the tautological bundle.
\end{exam}

Equivariant cohomology is functorial: suppose that $M,N$ are $G$-manifolds and that $f:M\rightarrow N$ is a $G$-equivariant smooth map. Then 
\[
\begin{array}{ccc}
M\times_GEG & \longrightarrow & N\times_GEG\\
\left[p,g\right] & \longmapsto & [f(p),g]
\end{array}
\] 
is well defined. The morphism it induces in singular cohomology gives a morphism 
\[
f^*_G:H^*_G(N)\longrightarrow H^*_G(M)
\] 
in equivariant cohomology.

\begin{exam}\label{eqcoexam} Let us compute some examples of equivariant cohomology: 
\begin{enumerate}
\item If $G$ acts on itself by left translations, the action is free and transitive, so \[H^*_G(G)=H^*(G/G)=H^*(pt).\]
\item The most simple example of a non-free action one can think of is $G$ acting on a point. Then \[H^*_G(pt)=H^*(pt\times_GEG)\simeq H^*(BG).\] Therefore, for example, $H_{S^1}^*(pt;\QQ)\simeq \QQ[t]$ which shows that, in general, equivariant cohomology is non-zero for degrees as large as one wishes. This differs notably from usual cohomology and prevents some relevant results such as Poincaré duality to hold.
\end{enumerate}
\end{exam}

\begin{defi}
Let $M$ be a $G$-manifold. The morphism 
\[
c_M:H^*(BG)\longrightarrow H^*_G(M)
\]
induced by $M\rightarrow pt$ in equivariant cohomology is the {\it characteristic morphism}\footnote{The characteristic morphism can also be interpreted as the map induced in singular cohomology by the projection $M\times_GEG\rightarrow BG$.} of $M$.
\end{defi}

The characteristic morphism gives $H^*_G(M)$ an $H^*(BG)$-module structure, so $H_G^*(M)$ is an $H^*(BG)$-algebra.

\begin{rema} Let $\CH_G$ be the category with $G$-manifolds as objects and $G$-equivariant smooth maps as morphisms. If $f:M\rightarrow N$ is such a morphism, $f^*_G:H_G^*(N)\rightarrow H_G^*(M)$ commutes with characteristic maps, and thus it is a linear map of $H^*(BG)$-modules. Since $f^*_G$ is also a ring homomorphism we conclude that equivariant cohomology is a contravariant functor from $\CH_G$ to the category of $H^*(BG)$-algebras.\end{rema}

The inclusion of $M$ as the fibre of the bundle $M\times_GEG\rightarrow BG$ induces a morphism
\[
r_M:H^*_G(M)\rightarrow H^*(M)
\]
in singular cohomology that we call the {\it restriction morphism} of $M$.

\begin{defi}\label{defi:formality}
A $G$-manifold is {\it equivariantly formal} if its restriction morphism is surjective.
\end{defi}

\begin{exam} Let us see a couple of examples, one of a non equivariantly formal manifold and one of an equivariantly formal manifold:
\begin{enumerate}
\item Let $T=\STS$ be the $2$-torus and consider the diagonal action of $S^1$ on $T$, which is free. Then $H_{S^1}^*(T)\simeq H^*(T/S^1)\simeq H^*(S^1)$, so $H_{S^1}^1(T;\QQ)\simeq\QQ$. However $H^1(T;\QQ)\simeq\QQ^2$, so the degree $1$ piece of $r_T$ cannot be surjective. Therefore $T$ is not equivariantly formal.
\item Consider the action of $S^1$ on $\CC P^1$ given by $\te[z:w]=[\te z:w]$. Note that $BS^1=\CC P^\infty$ and that the composition $r_{\CC P^1}\circ c_{\CC P^1}$ of the characteristic map with the restriction map is just the map induced by the inclusion $\CC P^1\hookrightarrow\CC P^\infty$, so in rational cohomology we get the surjective map
\[
\begin{array}{rccc}
r_{\CC P^1}\circ c_{\CC P^1}:& \QQ[t] & \longrightarrow & \QQ[t]/t^2\\
 & t &\longmapsto & t
\end{array}.
\]
Since the composition is surjective we also have that $r_{\CC P^1}$ is surjective.
\end{enumerate}
\end{exam}


\subsection{Rational equivariant cohomology}

Throughout this section the coefficient ring is assumed to be $\QQ$ and we will write it every time in order to emphasize this fact. Using rational coefficients gives us the possibility to use the Leray-Hirsch theorem and, in particular, the Künneth isomorphism. One first important consequence of this is the following: if $M$ is equivariantly formal the restriction map of $M$ is surjective and a direct application of Leray-Hirsch gives

\begin{prop}\label{prop:split}
Let $M$ be an equivariantly formal $G$-manifold and let $s_M$ be a section of the restriction map $r_M$. The map
\[
\begin{array}{ccc}
H^*(BG;\QQ)\ot H^*(M;\QQ) &\longrightarrow &H^*_G(M;\QQ)\\
\be\ot\mu & \longmapsto & c_M(\be)\smile s_M(\mu)
\end{array}
\]
is an isomorphism of $H^*(BG;\QQ)$-modules.
\end{prop}

We can use this result to prove an equivariant version of the Künneth isomorphism. First we will need a lemma:

\begin{lemm}\label{lemm:eqfoproduct} 
Let $M,N$ be equivariantly formal $G$-manifolds. Then $\MTN$ is also an equivariantly formal $G$-manifold.
\end{lemm}

\begin{proof}
The projection $\pi_M:\MTN\rightarrow M$ induces maps
\[
\pi_M^*: H^*(M;\QQ)\rightarrow H^*(\MTN;\QQ)
\] 
in singular cohomology and 
\[
\pi_{M,G}^*:H_G^*(M;\QQ)\rightarrow H^*_G(\MTN;\QQ)
\] 
in equivariant cohomology. These maps commute with the restriction maps: $\pi_M^*\circ r_M=r_{\MTN}\circ\pi_{M,G}^*$. In the same way the projection $\pi_N$ onto $N$ satisfies $\pi_N^*\circ r_N=r_{\MTN}\circ\pi_{N,G}^*$. From these facts we deduce that the following diagram is commutative:
\[
\begin{diagram}
\node{H_G^*(M;\QQ)\ot H_G^*(N;\QQ)}\arrow{s,l}{\pi_{M,G}^*\smile\pi_{N,G}^*}\arrow{e,t}{r_M\ot r_N}\node{H^*(M;\QQ)\ot H^*(N;\QQ)}\arrow{s,r}{\pi_M^*\smile\pi_N^*}\\
\node{H_G^*(\MTN;\QQ)}\arrow{e,b}{r_{\MTN}}\node{H^*(\MTN;\QQ)}
\end{diagram}.
\]
Since $M,N$ are equivariantly formal, $r_M,r_N$ are surjective maps, so the top arrow is surjective. The right hand side vertical arrow is the Künneth isomorphism. Hence the composition of both is a surjective map, implying that the bottom arrow, $r_{M\times N}$, is also surjective.
\end{proof}

\begin{prop}[Equivariant Künneth]\label{prop:equikunn} Let $M,N$ be equivariantly formal $G$-manifolds.  Then there is an isomorphism of  $H^*(BG;\QQ)$-algebras
\[
H^*_G(\MTN;\QQ)\simeq H^*_G(M;\QQ)\ot_{H^*(BG;\QQ)}H^*_G(N;\QQ).
\]
\end{prop}

\begin{proof} Since the maps $\pi_{M,G}^*,\pi_{N,G}^*$ are homomorphisms of $H^*(BG;\QQ)$-algebras, the map
\[
\begin{array}{rccc}
h: &H^*_G(M;\QQ)\ot_{H^*(BG;\QQ)}H^*_G(N;\QQ) & \longrightarrow  & H^*_G(\MTN;\QQ)\\
    & m\ot n & \longmapsto & \pi_{M,G}^* m\smile\pi_{N,G}^* n
\end{array},
\]
is a well defined homomorphism of $H^*(BG;\QQ)$-algebras. From Lemma \ref{lemm:eqfoproduct} we know that $\MTN$ is equivariantly formal. Using this fact we derive that the source and target spaces of $h$ are isomorphic as $H^*(BG;\QQ)$-modules:
\[
\begin{array}{l}
H^*_G(\MTN;\QQ)\simeq\\
\hspace{1.3cm}\stackrel{(1)}{\simeq}  H^*(BG;\QQ)\ot H^*(\MTN;\QQ)\\
\hspace{1.3cm}\stackrel{(2)}{\simeq}  H^*(BG;\QQ)\ot H^*(M;\QQ)\ot H^*(N;\QQ)\\
\hspace{1.3cm}\simeq  (H^*(BG;\QQ)\ot H^*(M;\QQ))\ot_{H^*(BG;\QQ)}(H^*(BG;\QQ)\ot H^*(N;\QQ))\\
\hspace{1.3cm}\stackrel{(3)}{\simeq}  H^*_G(M;\QQ)\ot_{H^*(BG;\QQ)}H^*_G(N;\QQ)
\end{array}
\]
where (1) is given by $\MTN$ being equivariantly formal, (2) is the usual Künneth isomorphism and (3) is given by $M,N$ being equivariantly formal.

We shall now see that $h$ is surjective: if $s_M,s_N$ are sections of $r_M,r_N$, 
\[
s_{\MTN}:=(\pi_{M,G}^*\smile\pi_{N,G}^*)\circ(s_M\ot s_N)\circ(\pi_M^*\smile\pi_N^*)^{-1}
\]
is a section of $r_{\MTN}$. Now, given $a\in H_G^*(M\times N;\QQ)$, by Proposition \ref{prop:split} and Künneth there exist unique $\be\in H^*(BG;\QQ)$, $\mu\in H^*(M;\QQ)$ and $\nu\in H^*(N;\QQ)$ such that 
\[
a=c_{\MTN}(\be)\smile s_{\MTN}(\pi_M^*\mu\smile\pi_N^*\nu).
\] 
But then 
\[
a=c_{\MTN}(\be)\smile(\pi_{M,G}^*\smile\pi_{N,G}^*)(s_M\mu\ot s_N\nu)=h(\be\cdot(s_M\mu\ot s_N\nu)).
\]
Since the restriction of $h$ to each degree piece is a surjective map between finite dimensional linear spaces of the same dimension, $h$ must also be injective.
\end{proof}

We now move to another result derived from using rational coefficients: the very definition of equivariant cohomology guarantees that if the action is free there is an isomorphism $H^*_G(M;\QQ)\simeq H^*(M/G;\QQ)$. This isomorphism can be extend to a bit more general situation, namely when the stabilisers of the action are all finite.

\begin{lemm}\label{lemm:finitegroup}
Let $\Ga$ be a finite group. Then $H^*(B\Ga;\QQ)\simeq\QQ$.
\end{lemm}

\begin{proof}
We shall see that $H_*(B\Ga;\QQ)\simeq\QQ$ and the result follows by duality:

Let $\pi:E\Ga\rightarrow B\Ga$ be the classifying bundle of $\Ga$ and let $k>0$. Let $s=\sum_iq_i\sig_i\in S_k(B\Ga;\QQ)$ be a singular $k$-chain such that $\partial\al=0$. Then
\[
\tilde{\al}:=\sum_iq_i\sum_{\pi\circ\tilde{\sig}_i=\sig_i}\tilde{\sig}_i\in S_k(E\Ga;\QQ)
\]
is a $k$-chain such that $\partial\tilde{\al}=0$ and such that $\pi_*(\tilde{\al})=|\Ga|\cdot\al$. Since $E\Ga$ is contractible there exists $\be\in S_{k+1}(E\Ga;\QQ)$ such that $\partial\be=\tilde{\al}$. Then $\partial\pi_*\be=\pi_*\tilde{\al}=|\Ga|\cdot\al$, so $\left[|\Ga|\cdot\al\right]=0\in H_k(B\Ga;\QQ)$. We deduce that multiplication by $|\Ga|$ on $H_k(B\Ga;\QQ)$ is the zero map and it must be the case that $H_k(B\Ga;\QQ)=0$.
\end{proof}

\begin{rema} If we take a ring of coefficients different from $\QQ$, this lemma does not hold in general: we cannot expect the cohomology of the classisfying space of a finite group to be isomorphic to the coefficient ring. For example, take $\Ga=\ZZ/(2)$ as the finite group and take $\ZZ/(2)$ as the coefficient ring. The classifying bundle of $\ZZ/(2)$ is $S^\infty\rightarrow\RR P^\infty$ but $H^*(\RR P^\infty;\ZZ/(2))$ is isomorphic to the polynomial ring\footnote{This can be computed by identifying $S^\infty\rightarrow\RR P^\infty$ with the sphere bundle associated to the tautological bundle $\mathcal{O}(-1)\rightarrow\RR P^\infty$ and using the corresponding Gysin sequence.} $\ZZ/(2)[x]$.\end{rema}

\begin{prop}[Cartan isomorphism]\label{prop:cartaniso} Let $M$ be a $G$-manifold and suppose that the action has finite stabilisers. There is an isomorphism \[H_G^*(M;\QQ)\simeq H^*(M/G;\QQ).\]\end{prop}

\begin{proof}
Let $\pi:M\times_GEG\rightarrow M\stackrel{p}{\rightarrow} M/G$ denote the natural projection. We shall see that $\pi^*:H^*(M/G;\QQ)\rightarrow H^*_G(M;\QQ)$ is an isomorphism. 

Taking a $G$-invariant tubular neighbourhood for each $G$-orbit in $M$ gives an open cover of $M$. Choose a finite subcover  $V_1,\ldots,V_n$. Then the sets $U_i:=p(V_i)$ form a good cover of $M/G$. We will use induction to see that
\[
\pi^*:H^*(U_1\cup\ldots\cup U_s;\QQ)\longrightarrow H^*_G(V_1\cup\ldots\cup V_s;\QQ)
\] 
is an isomorphism for all $s=1,\ldots,n$.

Let us start by the case $s=1$: let $V$ be  a tubular neighbourhood of an orbit $O_p$ in $M$, then $U=V/G$ is contractible and hence $H^*(U;\QQ)\simeq\QQ$. Then we must check  that also $H_G^*(V;\QQ)\simeq\QQ$. We have isomorphisms
\[
H^*_G(V;\QQ)\simeq H^*_G(O_p;\QQ)\simeq H^*_G(G/G_p;\QQ)=H^*((G/G_p)\times_GEG;\QQ).
\] 
Now, $(G/G_p)\times_GEG\simeq(EG\times_GG)/G_p\simeq EG/G_p\simeq BG_p$. From both facts we deduce that $H^*_G(V;\QQ)\simeq H^*(BG_p;\QQ)$. Finally, from Lemma \ref{lemm:finitegroup}, we get $H^*(BG_p;\QQ)\simeq\QQ$ because $G_p$ is finite.

Induction step: suppose that $H^*(U_1\cup\ldots\cup U_{s-1};\QQ)\simeq H_G^*(V_1\cup\ldots\cup V_{s-1};\QQ)$. Using this fact and the case $s=1$ we have a new isomorphism  
\[
H^*(U_1\cup\ldots\cup U_{s-1};\QQ)\oplus H^*(U_s;\QQ)\simeq H_G^*(V_1\cup\ldots\cup V_{s-1};\QQ)\oplus H^*_G(V_s;\QQ).
\] 
The induction hypothesis also gives an isomorphism 
\[
H^*((U_1\cup\ldots\cup U_{s-1})\cap U_s;\QQ)\simeq H_G^*((V_1\cup\ldots\cup V_{s-1})\cap V_s;\QQ).
\] 
Considering the Mayer-Vietoris sequences for $U_1\cup\ldots\cup U_s$ and $V_1\cup\ldots\cup V_s$ and using the five lemma we get the desired result.\end{proof}

In the other chapters we deal mainly with circle actions. We now compute explicitly an example of $S^1$-equivariant rational cohomology that will be used several times:

\begin{exam}\label{exam:eqcoprojexam}

Consider the action of $S^1$ on $\CC P^n$ given by
\[
\te[z_0:\ldots:z_n]=[\te^{j_0}z_0:\te^{j_1}z_1:\ldots:\te^{j_n}z_n]
\] 
where $j_0,\ldots,j_n$ are integers. Let us compute $H_{S^1}^*(\CC P^n;\QQ)$: the bundle
\[
\CC P^n\times_{S^1}ES^1\rightarrow BS^1
\] 
is the projectivisation of the complex rank $n+1$ vector bundle
\[
V=\CC^{n+1}\times_{S^1}ES^1\rightarrow BS^1,
\] 
so, as a consequence of the Leray-Hirsch theorem,
\[
H_{S^1}^*(\CC P^n;\QQ)\simeq H^*(BS^1;\QQ)[x]/(x^n-c_1(V)x^{n-1}+\cdots+(-1)^nc_n(V)),
\] 
where $c_1(V),\cdots,c_n(V)\in H^*(BS^1;\QQ)$ are the Chern classes of $V$ and $x$ is the Euler class of the tautological bundle $\OO_{\CP(V)}(-1)\rightarrow\CP(V)$.

We know that $H^*(BS^1;\QQ)=\QQ[t]$ where $t$ is the Euler class of the tautological bundle $\OO_{\CC P^\infty}(-1)\rightarrow\CC P^\infty\simeq BS^1$. There is a line bundle splitting $V=V_0\oplus\cdots\oplus V_n$, with $c_1(V_i)=j_it$. Then
\[
c_k(V)=\sig_k(c_1(V_0),\ldots,c_1(V_n))=\sig_k(j_0,\ldots,j_n)t^k,
\] 
where $\sig_k$ is the $k$-th symmetric polynomial in $n+1$ variables. Therefore 
\[
x^n-c_1(V)x^{n-1}+\cdots+(-1)^nc_n(V)=\sum_{k=0}^n(-1)^k\sig_k(j_0,\ldots,j_n)t^kx^{n-k}=\prod_{i=0}^n(x-j_it).
\]
Finally, we get $H_{S^1}^*(\CC P^n;\QQ)\simeq \QQ[t,x]/(x-j_0t)\cdots(x-j_nt)$.
\end{exam}


\subsection{The equivariant diagonal class}\label{subsec:edc}

We devote this last section on equivariant cohomology to the construction of a particular class that we will use in Section \ref{sec:biidiagprod}. In order to carry out this construction we need to review shriek maps. We consider again any ring of coefficients, that we will not write. 

If $M$ is a compact, connected and oriented manifold of dimension $m$ we denote by
\[
\begin{array}{rccc}
PD_M: & H^*(M) & \longrightarrow & H_{m-*}(M)\\
 &\al& \longmapsto &\al\frown[M]
\end{array}
\]
the isomorphism given by Poincaré duality. Let $N$ also be a compact, connected and oriented manifold and let $n=\dim N$. If $f:M\rightarrow N$ is continuous, we can define the {\it shriek}\footnote{Depending on the source this map is also called {\it transfer}, {\it umkehrungs} or just {\it push-forward}. We take the name and the notation from \cite{Bre}.} map
\[
f^!:=PD^{-1}_N\circ f_*\circ PD_M: H^*(M)\longrightarrow H^{*+n-m}(N),
\]
where $f_*:H_*(M)\rightarrow H_*(N)$ is the map induced in singular homology by $f$.

\begin{lemm}\label{lemm:thom}
Let $i:M\hookrightarrow N$ be and embedding and let $T$ be a closed tubular neighbourhood of $i(M)$ inside $N$. Then $i^!$ equals the composition
\[
H^*(M)\stackrel{t}{\longrightarrow}H^{*+n-m}(T,\partial T)\stackrel{e^{-1}}{\longrightarrow}H^{*+n-m}(N,N\setminus M)\stackrel{g}{\longrightarrow} H^{*+n-m}(N),
\] 
where $t$ is the Thom isomorphism, $e^{-1}$ is the inverse of the excision isomorphism and $g$ is the natural map in the exact sequence of the pair $(N,N\setminus M)$.
\end{lemm}

\begin{proof}
We assume without loss of generality that $M$ is a submanifold of $N$ and that $i$ is the inclusion. If we denote by $i_M^T,i_T^N$ the obvious inclusions we have that $i=i_T^N\circ i_M^T$ and hence $i^!=(i_T^N)^!\circ(i_M^T)^!$.

Note that $T$ can be regarded as the total space of a rank $n-m$ disk bundle $\pi:T\rightarrow M$ with base space $M$. We shall prove that $(\io_M^T)^!$ is precisely the Thom isomorphism of this bundle:

Since $\pi\circ i_M^T=\mbox{id}_M$ and $i_M^T\circ\pi\sim\mbox{id}_T$ we have that $\pi_*,(i_M^T)_*$ are mutually inverse isomorphisms and so are $\pi^*,(i_M^T)^*$. Hence $(i_M^T)^!$ is a composition of isomorphisms and therefore it is an isomorphism too. The Thom class of the disk bundle is defined by
\[
\tau:=PD_T^{-1}(i_M^T)_*[M]\in H^{n-m}(T,\partial T),
\] 
or, equivalently, $\tau\frown[T]=(i_M^T)_*[M]$. Then the Thom isomorphism is the composition
\[
t:H^*(M)\stackrel{\pi^*}{\longrightarrow}H^*(T)\stackrel{\smile\tau}{\longrightarrow}H^{*+n-m}(T,\partial T),
\]
so we need to show that this composition coincides with $(i_M^T)^!$. If $\al\in H^*(M)$ and $\be=\pi^*\al$ --and hence $(i_M^T)^*\be=\al$-- we have that
\[
\begin{array}{rcl}
PD_T(i_M^T)^!(\al) & = & (i_M^T)_*PD_M(\al)\\
               & = & (i_M^T)_*(\al\frown[M])\\
               & = & (i_M^T)_*((i_M^T)^*\be\frown[M])\\
               & = & \be\frown(i_M^T)_*[M]\\
               & = & \be\frown(\tau\frown[T])\\
               & = & (\be\smile\tau)\frown[T]\\
               & = & PD_T(\be\smile\tau)\\
               & = & PD_T(\pi^*\al\smile\tau)
\end{array},
\]
so $(i_M^T)^!(\al)=\pi^*\al\smile\tau$, which is precisely the Thom isomorphism.

To finish the proof note that the equality $g\circ e^{-1}=(i_T^N)^!$ follows from the commutativity of the diagram
\[
\begin{diagram}
\node{H^{*+n-m}(N,N\setminus M)}\arrow{e,t}{g}\arrow{s,lr}{e}{\simeq}\node{H^{*+n-m}(N)}\arrow[2]{s,lr}{\simeq}{PD_N}\\
\node{H^{*+n-m}(T,\partial T)}\arrow{s,lr}{PD_T}{\simeq}\\
\node{H_{m-*}(T)}\arrow{e,b}{(i_T^N)^*}\node{H_{m-*}(N)}
\end{diagram}
\]
\end{proof}

\begin{lemm}\label{lemm:commutes} Let $P$ a compact connected and oriented manifold of dimension $p$ and let $V,W$ be connected submanifolds of $P$ of dimensions $v,w$. If $V,W$ intersect transversally, the diagram
\[
\begin{diagram}
\node{H^*(V)}\arrow{e,t}{(i_V^P)^!}\arrow{s,l}{(i_{V\cap W}^V)^*}\node{H^{*+p-v}(P)}\arrow{s,r}{(i_W^P)^*}\\
\node{H^*(V\cap W)}\arrow{e,b}{(i_{V\cap W}^W)^!}\node{H^{*+p-v}(W)}
\end{diagram}
\]
induced by the inclusion maps is commutative.
\end{lemm}

\begin{proof} Let $T$ be a tubular neighbourhood of $V$ inside $P$. Then $T_0=T\cap W$ is a tubular neighbourhood of $V\cap W$ inside $W$. Put $d=p-v$. We get the commutative diagram
\[
\divide\dgARROWLENGTH by3
\begin{diagram}
\node{H^*(V)}\arrow{e,t}{t}\arrow{s,l}{(i_{V\cap W}^V)*}\node{H^{*+d}(T,\partial T)}\arrow{e,t}{e}\arrow{s,l}{(i_{T_0}^T)^*}\node{H^{*+d}(P,P\setminus V)}\arrow{e,t}{g}\arrow{s,r}{(\io_W^P)^*}\node{H^{*+d}(P)}\arrow{s,r}{(i_W^P)^*}\\
\node{H^*(V\cap W)}\arrow{e,b}{t}\node{H^{*+d}(T_0,T_0\setminus V)}\arrow{e,b}{e}\node{H^{*+d}(W,W\setminus V)}\arrow{e,b}{g}\node{H^{*+d}(W)}
\end{diagram}.
\]
By Lemma \ref{lemm:thom} the top arrow is $(i_V^P)^!$ and the bottom arrow is $(i_{V\cap W}^W)^!$, which proves the result.
\end{proof}

Let $M$ be a $G$-manifold of dimension $m$. Denote by $\io_{\De}:M\rightarrow\MTM$ the diagonal embedding and let $EG_k\rightarrow BG_k$ be $G$-principal bundles approximating the universal bundle $EG\rightarrow BG$. For each $k$ we have an induced map
\[
\io_{\De,k}:M\times_GEG_k\longrightarrow(\MTM)\times_GEG_k,
\]
where the action of $G$ on $\MTM$ is the diagonal action. We get shriek maps
\[
\io_{\De,k}^!:H^*(M\times_GEG_k)\longrightarrow H^{*+m}((\MTM)\times_GEG_k),
\]
which we want to see that stabilise: take the inclusions 
\[
(i_M)_k:M\times_GEG_k\hookrightarrow M\times_GEG_{k+1}, 
\] 
\[
(i_{\MTM})_k:(\MTM)\times_GEG_k\hookrightarrow (\MTM)\times_GEG_{k+1}.
\] 
Applying Lemma \ref{lemm:commutes} to $P=(\MTM)\times_GEG_{k+1}$, $W=(\MTM)\times_GEG_k$ and $V=\De_{\MTM}\times_GEG_{k+1}$ we get that the diagram
\[
\begin{diagram}
\node{H^*(M\times_GEG_{k+1})}\arrow{e,t}{\io_{\De,k+1}^!}\arrow{s,l}{(i_M)_k^*}\node{H^{*+m}((\MTM)\times_GEG_{k+1})}\arrow{s,r}{(i_{\MTM})_k^*}\\
\node{H^*(M\times_GEG_k)}\arrow{e,b}{\io_{\De,k}^!}\node{H^{*+m}((\MTM)\times_GEG_k)}
\end{diagram}
\]
commutes. By means of stabilisation --Lemma \ref{lemm:stabilisation}-- we get a well defined equivariant shriek map
\[
\io_{\De,G}^!:H_G^*(M)\longrightarrow H_G^{*+m}(\MTM).
\] 
Using this map we can define the object claimed at the beginning of this section:

\begin{defi}\label{defi:edc} Let $M$ be an $m$-dimensional $G$-manifold. The equivariant cohomology class
\[
\io_{\De,G}^!(1)\in H_G^m(\MTM)
\]
is called the {\it equivariant diagonal class} of $M$.
\end{defi}

\begin{rema}\label{rema:edcapprox}
Note that the equivariant diagonal class is obtained by stabilisation of the classes
\[
\begin{array}{rl}
\io_{\De,k}^!(1) & =(PD^{-1}_{(\MTM)\times_GEG_k}\circ(\io_{\De,k})_*\circ PD_{M\times_GEG_k})(1)\\
                 & =(PD^{-1}_{(\MTM)\times_GEG_k}\circ(\io_{\De,k})_*)[M\times_GEG_k]\\
                 & =PD^{-1}_{(\MTM)\times_GEG_k}[\De_{\MTM}\times_GEG_k].
\end{array}
\]
\end{rema}


\section{Branched manifolds}\label{sec:branched}

We switch now to a completely different topic. In this section we introduce {\it branched manifolds}, a generalisation of smooth manifolds which appears naturally in our context. This notion is by no means new: the basic definitions are inspired by those in \cite[\S 1]{Wil}. Another source is \cite[\S 5.4]{Sal}, where branched manifolds are defined with a weighting from the very beginning. However, we will only use weightings on compact one-dimensional branched manifolds with boundary --the so called {\it train tracks}--, which is sufficient to our needs.

\subsection{Smooth ramifications}

We begin by defining the basic building blocks we will need to define branched manifolds:

\begin{defi}
A {\it smooth ramification} of dimension $n\geq 1$ and rank $r\geq 1$ is a triple $(X,V,\pi)$ where 
\begin{itemize}
\item $X$ is a topological space
\item $V\subseteq X$ 
\item $\pi:X\rightarrow D^n$ is a continuous map to the standard open disk in $\RR^n$ 
\end{itemize}
satisfying
\begin{enumerate}
\item $\pi_{|V}:V\rightarrow\pi(V)$ is a homeomorphism and $\pi(V)$ is a smooth $n$-dimensional manifold with boundary
\item there exist $r$ sets $D_1,\ldots,D_r$ subject to the following conditions:
\begin{enumerate}
\item $X=D_1\cup\cdots\cup D_r$
\item $D_i\cap D_j=V$, $\forall i\neq j$
\item $\pi_{|D_i}:D_i\rightarrow D^n$ is a homeomorphism
\end{enumerate}
\end{enumerate}
\end{defi}

\begin{figure}[H]
	\centering
		\includegraphics{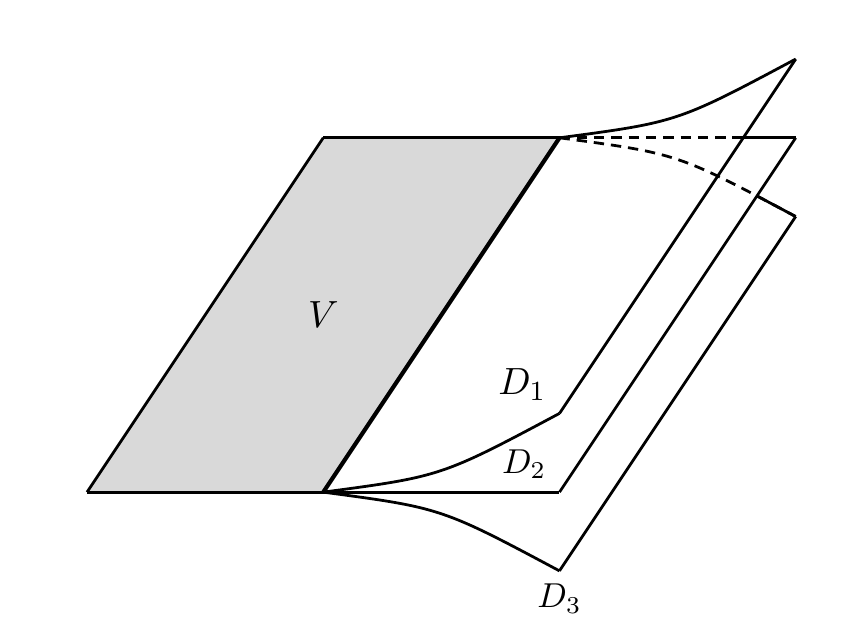}
	\caption{A smooth ramification}
	\label{fig:ramification}
\end{figure}

Later on we will use the following property of the local structure of a smooth ramification:

\begin{prop}\label{prop:lochombran}
Let $(X,V,\pi)$ be a smooth ramification of dimension $n$ and rank $r$ and let $x\in X$ satisfy $\pi(x)\in\partial\pi(V)$. Then the reduced local homology of $x$ satisfies
\[
\tilde{H}_k(X,X\setminus\{x\};\QQ)\simeq\left\{
\begin{array}{ll}
\QQ^r & \mbox{if}\ k=n\\
0 &\mbox{otherwise}  
\end{array}
\right..
\]
\end{prop}

\begin{proof}
By excision we can assume that $X$ is contractible and that $V$ is homeomorphic to the half-disk. Then the reduced homology of $X$ vanishes. Using this fact on the long exact sequence of relative homology of $X$ and $X\setminus\{x\}$ we find out that there is an isomorphism $\tilde{H}_k(X,X\setminus\{x\};\QQ)\simeq\tilde{H}_{k-1}(X\setminus\{x\};\QQ)$.

If $r=1$ then $X\simeq D^n$ and hence 
\[
\tilde{H}_{k-1}(X\setminus\{x\};\QQ)\simeq\left\{
\begin{array}{ll}
\QQ & \mbox{if}\ k=n\\
0 &\mbox{otherwise}  
\end{array}
\right..
\]
This proves the proposition for $r=1$. We will prove the remaining cases by induction on $r$. There are two key observations:
\begin{enumerate}[(i)]
\item $(X':=D_2\cup\ldots\cup D_r,V,\pi_{|X'})$ is a smooth ramification of dimension $n$ and rank $r-1$. 
\item $V\setminus\{x\}$ is contractible because $V$ is homeomorphic to the half-disk and $x\in\partial V$. 
\end{enumerate}

Using the Mayer-Vietoris sequence for the sets $D_1\setminus\{x\}$ and $X'\setminus\{x\}$ --whose union is $X\setminus\{x\}$ and whose intersection is $V\setminus\{x\}$--, excision and observation (ii) we get an isomorphism 
\[
\tilde{H}_{k-1}(X\setminus\{x\};\QQ)\simeq\tilde{H}_{k-1}(D_1\setminus\{x\};\QQ)\oplus\tilde{H}_{k-1}(X'\setminus\{x\};\QQ).
\]
Now from observation (i) and induction it follows that 
\[
\tilde{H}_{k-1}(X\setminus\{x\};\QQ)\simeq\left\{
\begin{array}{ll}
\QQ^r & \mbox{if}\ k=n\\
0 &\mbox{otherwise}  
\end{array}
\right.
\]
as we wished to show.
\end{proof}

It is also convenient to make the following definition:

\begin{defi} Let $(X_1,V_1,\pi_1),\ldots,(X_s,V_s,\pi_s)$ be smooth ramifications of dimension $n$. We say that its fibred product $X_1\times_{D^n}\ldots\times_{D^n}X_s$ is {\it smooth} if $\partial\pi_{i_1}(V_{i_1})\pitchfork\cdots\pitchfork\partial\pi_{i_j}(V_{i_j})$ for all $\{i_1,\ldots,i_j\}\subseteq\{1,\ldots,s\}$. 
\end{defi}

\subsection{Branched atlases}

In this section we define branched manifolds. As in the common definition of a smooth manifold we make use of atlases:

\begin{defi}
A {\it branched $n$-atlas} on a topological space $M$ is a collection of pairs $\{(U_i,h_i)\}$, called {\it branched charts}, where 
\begin{itemize}
\item $\{U_i\}$ is a collection of open subsets of $M$ such that $\bigcup_i U_i=M$
\item $h_i: U_i\rightarrow X_{i1}\times_{D^n}\cdots\times_{D^n}X_{is_i}$ is a homeomorphism from $U_i$ to a smooth fibred product of smooth ramifications $(X_{i1},V_{i1},\pi_{i1}),\ldots,(X_{is},V_{is},\pi_{is_i})$ of dimension $n$
\end{itemize}
subject to the following condition: let $\pi_i:X_{i1}\times_{D^n}\cdots\times_{D^n}X_{is_i}\rightarrow D^n$ denote the projection, then if $U_i\cap U_{i'}\neq\emptyset$ there exists a diffeomorphism
\[
\vp_{i'i}:(\pi_i\circ h_i)(U_i\cap U_{i'})\rightarrow(\pi_{i'}\circ h_{i'})(U_i\cap U_{i'})
\] 
such that $\pi_{i'}\circ h_{i'}=\vp_{i'i}\circ\pi_i\circ h_i$.
\end{defi}

Two branched atlases on the same topological space are {\it equivalent} if their union is again a branched atlas. This is an equivalence relation and so we can make the following definition:

\begin{defi}
An {\it $n$-dimensional branched manifold} is a second countable Hausdorff topological space $M$ together with an equivalence class of smooth branched $n$-atlases.
\end{defi}

On branched manifolds we find a special type of points that we do not find on genuine manifolds:

\begin{defi}\label{defi:branchingpoint}
Let $M$ be a branched manifold. We say that $x\in M$ is a {\it branching point} if there exists a branched chart $(U,h)$ such that $x\in U$ and if $h:U\rightarrow X_1\times_{D^n}\cdots\times_{D^n}X_s$ with $(X_j,V_j,\pi_j)$ smooth ramifications of rank $r_j$, then $x\in\partial V_j$ for some $j$ such that $r_j\geq 2$.
\end{defi}

We denote by $M^\prec$ the set of branching points of a branched manifold. Observe that if $x\in M\setminus M^\prec$ is a not branching point, there exists a chart $(U,h)$ with $x\in U$ and $U\simeq D^n$. Since $(U,h)$ is a chart in the usual sense, we deduce that the complement $M\setminus M^\prec$ of the set of branching points is a genuine manifold. The connected components of $M\setminus M^\prec$ are called {\it branches} of $M$.

On a genuine manifold $M$ of dimension $n$ all the points have the same reduced local homology,
\[
\tilde{H}_k(M,M\setminus\{x\};\QQ)=\left\{
\begin{array}{ll}
\QQ & \mbox{if}\ k=n\\
0 &\mbox{otherwise}  
\end{array}
\right..
\]
However, from Proposition \ref{prop:lochombran} it follows that branching points have more complicated local homology. Therefore local homology can be used to distinguish branching points from non-branching points.

\subsection{Tangent bundle and smooth maps on branched manifolds}

On branched manifolds we can also define a notion of tangent bundle. In order to do so we start by defining the tangent bundle on a smooth ramification:

Let $(X,V,\pi)$ be a smooth ramification of dimension $n$. Its {\it tangent bundle} is defined to be the pull-back $TX:=\pi^*(TD^n)$. Similarly, the tangent bundle on a smooth fibred product $X_1\times_{D^n}\cdots\times_{D^n} X_s$ of smooth ramifications is also defined to be the pull-back of $TD^n$ under the projection.

Since branched atlases are modelled on smooth fibred products of smooth ramifications we can now define the tangent bundle on a branched manifold: let $M$ be a branched manifold with atlas $\{(U_i,h_i)\}_{i\in I}$. If $x\in U_i$ we define the tangent space of $M$ at $x$ as $T_{h_i(x)}h_i(U_i)$. We can glue together these tangent spaces onto a tangent bundle on $M$ by means of the transition functions $\vp_{i'i}$.

We are also interested in defining smooth maps on branched manifolds. To our purposes it will suffice to assume the target manifold is a genuine smooth manifold.

\begin{defi} 
Let $M$ be a branched manifold with atlas $\{(U_i,h_i)\}_{i\in I}$ such that $h_i:U_i\rightarrow X_{i1}\times_{D^n}\cdots\times_{D^n}X_{is_i}$ with $(X_{ij},V_{ij},\pi_{ij})$ a smooth ramification of rank $r_{ij}$ (so that $X_{ij}=D_{ij1}\cup\cdots\cup D_{ijr_{ij}}$). Let $N$ be a smooth manifold. A map $f:M\rightarrow N$ is {\it smooth} if 
\begin{enumerate}
\item The map
\[
\begin{array}{rccc}
f^i_{k_1,\ldots,k_{s_i}}: & D^n&\longrightarrow & N\\
 & y &\longmapsto & f((\pi_{i1|D_{i1k_1}})^{-1}(y),\ldots,(\pi_{is_i|D_{is_ik_{s_i}}})^{-1}(y))
\end{array}
\]
is smooth for all $i\in I$ and for all $k_j\in\{1,\ldots,r_{ij}\}$.
\item For each $x\in U_i$, the maps $f^i_{k_1,\ldots,k_{s_i}}$ for various choices of $k_1,\ldots,k_{s_i}$ have the same germ at $(\pi_i\circ h_i)(x)$.
\end{enumerate}
\end{defi}

Note that if $x\in U_i$, then each $w\in T_xM$ determines a unique vector $\bar{w}\in T_{(\pi_i\circ h_i)(x)}D^n$. If $f:M\rightarrow N$ is a smooth map from a branched manifold $M$ to a smooth manifold $N$, its differential at $x\in U_i$ is defined by:
\[
\begin{array}{rccc}
d_xf: & T_xM & \longrightarrow &T_{f(x)}N\\
 &w&\longmapsto & d_{(\pi_i\circ h_i)(x)}f^i_{k_1\ldots k_{s_i}}(\bar{w})
\end{array},
\]
which is well defined because if also $x\in U_{i'}$, then
\[
\begin{array}{rcl}
d_xf(w')	& = & d_{(\pi_{i'}\circ h_{i'})(x)}f^{i'}_{k_1'\ldots k_{s_{i'}}'}(\bar{w}')\\
            & = & d_{(\pi_{i'}\circ h_{i'})(x)}f^{i'}_{k_1'\ldots k_{s_{i'}}'}(d_{(\pi_i\circ h_i)(x)}\vp_{i'i}(\bar{w}))\\
            & = & d_{(\pi_i\circ h_i)(x)}(f^{i'}_{k_1'\ldots k_{s_{i'}}'}\circ\vp_{i'i})(\bar{w})\\
            & = & d_{(\pi_i\circ h_i)(x)}f^i_{\bar{k}_1\ldots\bar{k}_{s_i}}(\bar{w})
\end{array}
\]
for some $\bar{k}_1\ldots\bar{k}_{s_i}$ and the germ of $f^i_{\bar{k}_1\ldots\bar{k}_{s_i}}$ at $(\pi_i\circ h_i)(x)$ is the same as the germ of $f^i_{k_1\ldots k_s}$.

With the notions of smooth map, tangent bundle and differential of a smooth map at hand, many other common notions of manifolds can be generalised verbatim to branched manifolds. Those of submanifold, immersion, embedding, submersion, regular value, transversality and orientability are examples of such.

\subsection{Train tracks}\label{subsec:traintracks}

Branched manifolds as we defined them have no boundary but, as in the usual situation, allowing charts to be modelled on half-spaces we can extend our definition to get a notion of branched manifold with boundary. As it happens for genuine manifolds, the reduced local homology of a boundary point vanishes. As a consequence, we can decompose $M$ into the union of three disjoint subsets which can be distinguished by their reduced local homology:
\begin{itemize}
\item The boundary $\partial M$ 
\item The set of branching points $M^\prec$
\item The complementary of the two sets above, $M\setminus(\partial M\cup M^\prec)$, which is a genuine manifold without boundary. The connected components of this set are called {\it branches} and the set of branches is denoted $M^B$.
\end{itemize}

In this section we are concerned with the following objects:

\begin{defi}\label{defi:traintrack}
A {\it train track} is a compact one dimensional oriented branched manifold with boundary.
\end{defi}

The following picture shows a train track with its branching points in orange and its boundary points in blue:

\begin{figure}[H]
	\centering
		\includegraphics[scale=0.9]{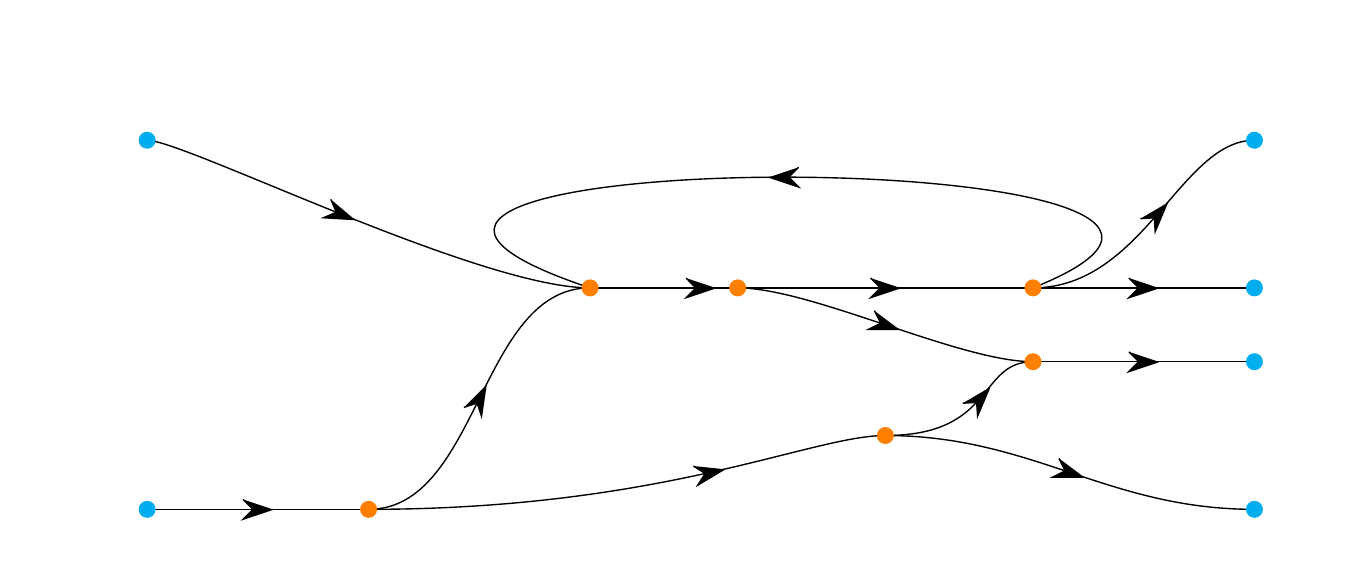}
	\caption{A train track.}
	\label{fig:traintrack}
\end{figure}

If $b\in T^B$ is a branch, either $b$ is diffeomorphic to a circle or $b$ has two extremes $h(b),t(b)\in\partial T\cup T^\prec$, called the {\it head} and the {\it tail} of $b$ according to the orientation of $b$. In these terms we make the following definition:

\begin{defi}\label{defi:weighting}
A {\it weighting} on a train track $T$ is a map $w:T^B\rightarrow\QQ_{>0}$ such that for all $x\in T^\prec$ it holds that
\[
\sum_{h(b)=x}w(b)=\sum_{t(b)=x}w(b).
\]
\end{defi}

The following picture is an example of weighting on the previous train track:

\begin{figure}[H]
	\centering
		\includegraphics[scale=0.9]{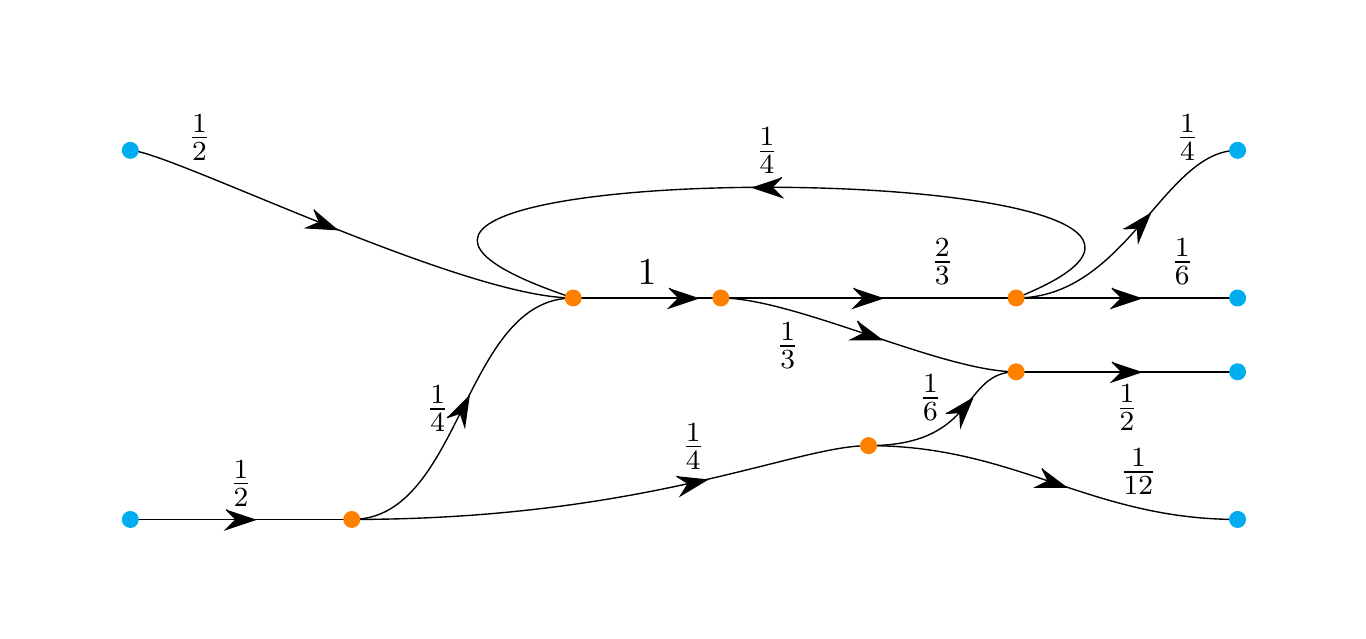}
	\caption{A train track with a weighting.}
	\label{fig:weightedtraintrack}
\end{figure}

\begin{lemm} The boundary set $\partial T$ of a train track is the disjoint union of the sets
\[
\partial T^+=\{x\in\partial T: \exists b\in T^B\ \mbox{s.t.}\ h(b)=x\},
\]
\[
\partial T^-=\{x\in\partial T: \exists b\in T^B\ \mbox{s.t.}\ t(b)=x\}.
\]
\end{lemm}

\begin{proof}
A point $x\in\partial T$ is not a branching point, so there exists a unique branch $b\in T^B$ such that $x$ is an extreme of $b$. According to if $x=h(b)$ or $x=t(b)$, $x$ lies in the first or the second set.
\end{proof}

The following result is the generalisation of the fact that compact one dimensional manifolds have an even number of boundary points:

\begin{prop}\label{prop:ttboundary} Let $T$ be a train track. For $x\in\partial T$ let $b_x\in T^B$ denote the unique branch such that $x$ is an extreme $b_x$. If $w$ is a weighting on $T$, then
\[
\sum_{x\in\partial T^+}w(b_x)=\sum_{x\in\partial T^-}w(b_x).
\]
\end{prop}

\begin{proof}
Note that the total weight of the train track satisfies both
\[
\sum_{b\in T^B}w(b)=\sum_{x\in \partial T\cup T^\prec}\sum_{h(b)=x}w(b)\ \ \mbox{and}\ \ \sum_{b\in T^B}w(b)=\sum_{x\in \partial T\cup T^\prec}\sum_{t(b)=x}w(e),
\]
from where it follows that
\[
\sum_{x\in \partial T\cup T^\prec}\left(\sum_{h(b)=x}w(b)-\sum_{t(b)=x}w(b)\right)=0,
\]
so we get
\[
\sum_{x\in\partial T}\left(\sum_{h(b)=x}w(b)-\sum_{t(b)=x}w(b)\right)+\sum_{x\in T^\prec}\left(\sum_{h(b)=x}w(b)-\sum_{t(b)=x}w(b)\right)=0.
\]
Since $w$ is a weighting, the second term of the sum vanishes and it follows that
\[
\sum_{x\in\partial T}\left(\sum_{h(b)=x}w(b)-\sum_{t(b)=x}w(b)\right)=0.
\]
The proof concludes by observing that
\[
\sum_{x\in\partial T}\sum_{h(b)=x}w(b)=\sum_{x\in\partial T^+}w(b_x)\ \ \mbox{and}\ \ \sum_{x\in\partial T}\sum_{t(b)=x}w(b)=\sum_{x\in\partial T^-}w(b_x).
\qedhere
\]
\end{proof}


\section{Pseudocycles}\label{sec:pseudocycles}

The last part of the toolbox chapter deals with pseudocycles. If $f:D\rightarrow M$ is a smooth map where $D$ is compact and oriented of dimension $d$, this map defines naturally an element $f_*[D]\in H_d(M;\ZZ)$ of the integral singular homology of $M$. In the theory of pseudocycles the compactness condition is substituted by a weaker condition. All along this section we take $M$ to be a compact, connected and oriented manifold.

\begin{defi} Let $f:N\rightarrow M$ be a smooth map. The {\it omega-limit set} of $f$ is the set 
\[
\Om_f=\{p\in M:\ \exists\{q_j\}_{j\geq1}\subseteq N\ \mbox{with no converging subseq. s.t. } p=\lim_{j\rightarrow \infty} f(q_j)\}.
\]
A smooth map $\vp:W\rightarrow M$ is said to be an {\it omega-map} of $f$ if $\Om_f\subseteq\vp(W)$.
\end{defi}

\begin{rema}\label{rema:generalisedpseudo}
Actually we let the source space $W$ of an omega-map have several --but a finite number of-- connected components, not necessarily of the same dimension. In that case when we talk about the ``dimension'' of $W$ we will be referring to the maximum of the dimensions of its components.
\end{rema}

\begin{rema}
When the target manifold is known we usually denote by $(f,N)$ the smooth map and by $(\vp,W)$ the omega-map.
\end{rema}

\begin{defi} Let $D$ be an oriented manifold of dimension $d$. A smooth map $f:D\rightarrow M$ is a $d$-{\it pseudocycle} of $M$ if there is an omega-map $\vp:W\rightarrow M$ of $f$ such that $\dim W\leq d-2$.
\end{defi}

\begin{defi}
Two $d$-pseudocycles $(f_0,D_0)$, $(f_1,D_1)$ are {\it bordant} if there exist a smooth map $\tilde{f}:\widetilde{D}\rightarrow M$, where $\widetilde{D}$ is $(d+1)$-dimensional manifold with boundary such that:
\begin{enumerate}
\item $\partial\widetilde{D}=D_1-D_0$ and $\tilde{f}_{|D_0}=f_0$, $\tilde{f}_{|D_1}=f_1$,
\item There exists an omega-map $\tilde{\vp}:\widetilde{W}\rightarrow M$ of $\tilde{f}$ with $\dim\widetilde{W}\leq d-1$.
\end{enumerate}
\end{defi}

Bordism gives an equivalence relation: we denote the set of its classes by $\BB_*(M)$. This is a $\ZZ$-graded $\ZZ$-module with operation given by disjoint union of pseudocycles: the inverse element is given by reversing the orientation and the neutral element is the empty pseudocycle. It turns out that there is a natural isomorphism of graded $\ZZ$-modules between $\BB_*(M)$ and  $H_*(M;\ZZ)$. A detailed proof of this fact can be found in \cite{Zin} (see also \cite{Kah}):

\begin{theo}\label{theo:mainpseudo}
There exist natural homomorphisms of graded $\ZZ$-modules
\[
\Psi_*^M:H_*(M;\ZZ)\longrightarrow\BB_*(M)\ \mbox{and}\ \Phi_*^M:\BB_*(M)\longrightarrow H_*(M;\ZZ),
\]
such that $\Phi_*^M\circ\Psi_*^M=\mbox{id}_{H_*(M;\ZZ)}$ and $\Psi_*^M\circ\Phi_*^M=\mbox{id}_{\BB_*(M)}$.
\end{theo}

\begin{rema}\label{rema:pseudocat} If $N$ is another compact, connected and oriented manifold, a smooth map $h:M\rightarrow N$ induces a map $h_*^{\BB}:\BB_*(M)\rightarrow\BB_*(N)$ via $[f]\mapsto[h\circ f]$ because $\Om_{h\circ f}=h(\Om_f)$ and, hence, $h\circ f$ is a pseudocycle of $N$. This construction commutes with the map $h_*:H_*(M;\ZZ)\rightarrow H_*(N;\ZZ)$ induced in cohomology under the correspondences of the theorem: $\Phi_*^N\circ h_*^\BB=h_*\circ\Phi_*^M$ and, hence, $\Psi_*^N\circ h_*=h_*^\BB\circ\Psi_*^M$. In other words $\BB_*$ and $H_*(\cdot;\ZZ)$ are equivalent as functors from the category of compact, connected and oriented manifolds to the category of $\ZZ$-graded $\ZZ$-modules.\end{rema}

\subsection{Strong transversality and intersection pairing of\\ pseudocycles}\label{sec:intersectionpairing}

We need some transversality results for pseudocycles. We begin defining the right transversality notion in this context:

\begin{defi}\label{defi:stronglytrans}
Let $f:N\rightarrow M$, $g:P\rightarrow M$ be smooth maps. We say that $f$ and $g$ are {\it strongly transverse} if
they are transverse and $\Om_f\cap\ov{g(P)}=\emptyset$, $\ov{f(N)}\cap\Om_g=\emptyset$.
\end{defi}

\begin{lemm}\label{lemm:stcsfinite}
If $f:N\rightarrow M$, $g:P\rightarrow M$ are strongly transverse maps, then the Cartesian square
\[
CS(f,g)=\{(x,y)\in N\times P:\ f(x)=g(y)\}
\] 
is a compact manifold of dimension $\dim N+\dim P-\dim M$.
\end{lemm}

\begin{proof}
Since $f,g$ are transverse, $CS(f,g)$ is a manifold of the aforementioned dimension. To see that it is compact we use the extra condition of strong transversality:

Suppose that $(x_j,y_j)\in CS(f,g)$ has no convergent subsequences. Then either $x_j\in N$ or $y_j\in P$ have no convergent subsequences. Say it is $x_j$. Then $x:=\lim_j f(x_j)\in M$ is a point of $\Om_f$. On the other hand $y:=\lim_j g(y_j)\in\ov{g(P)}$. Since $\Om_f\cap\ov{g(P)}=\emptyset$ we have that $x\neq y$. However, since $f(x_j)=g(y_j)$ for all $j$, their limits have to coincide. This is a contradiction. The same argument proves the case where $y_j$ has no convergent subsequences.
\end{proof}

For the seek of completeness we now prove a couple of general results on transversality that the reader may be familiar with:

\begin{lemm}\label{lemm:surjev} There exists a finite-dimensional linear subspace $A\subseteq C^\infty(TM)$ such that the evaluation map
\[
\begin{array}{rccc}
ev_p^A:&A&\longrightarrow&T_pM\\
     &X&\longmapsto&X_p
\end{array}
\]
is surjective for all $p\in M$.
\end{lemm}

\begin{proof} Let $m=\dim M$. Take an atlas $\{(U_i,\psi_i)\}_{i\in I}$ of $M$. Then, for every $i\in I$ there exist vector fields $X^{i,1},\ldots,X^{i,m}\in C^\infty(TM)$ such that $X_p^ {i,1},\ldots,X_p^{i,m}$ is a base of $T_pM$ for all $p\in U_i$ (one only needs to pull-back coordinate vector fields by $\psi_i$). Since $M$ is compact we can take a finite subcover $U_{i_1},\ldots,U_{i_N}$ of the open cover $\{U_i\}_{i\in I}$. Let $\lm_{i_1},\ldots,\lm_{i_N}$ be a partition of unity subordinate to this subcover. Define $A$ to be the vector space generated by the vector fields of the form $\lm_{i_j}X^{i_j,k}=:Y^{j,k}$ for $j=1,\ldots,N$ and $k=1,\ldots,m$. We claim that $A$ satisfies the desired conditions:

Let $p\in M$ and $v\in T_pM$. There exists $j$ such that $p\in U_{i_j}$ and $\lm_{i_j}(p)\neq 0$. Since $p\in U_{i_j}$ there exist real numbers $r_1,\ldots,r_m$ such that $v=r_1X_p^{i_j,1}+\cdots+r_mX_p^{i_j,m}.$ Then 
\[
v=ev_p\left(\frac{1}{\lm_{i_j}(p)}(r_1Y^{j,1}+\cdots+r_dY^{j,d})\right).\qedhere
\]
\end{proof} 

If $X\in C^\infty(TM)$ is a vector field and $\xi_t^X$ is its flow, then $\exp X:=\xi_1^X$ is a diffeomorphism of $M$. If $A\subseteq C^\infty(TM)$ is as in the lemma, the set $\CA=\{\exp X:\ X\in A\}\subseteq\mbox{Diff}(M)$ admits a smooth manifold structure with tangent spaces isomorphic to $A$. In these terms the differential of the evaluation map
\[
\begin{array}{rccc}
ev_p^{\CA}:&\CA&\longrightarrow & M\\
 &\eta&\longmapsto &\eta(p)
\end{array}
\]
is precisely $ev_p^A$. From the lemma it follows that $ev_p^{\CA}$ is a submersion.

\begin{lemm}\label{lemm:generictransverse}
Let $f:N\rightarrow M$, $g:P\rightarrow M$ be smooth and let $A\subseteq C^\infty(TM)$ be as in Lemma \ref{lemm:surjev}. Denote $\CA=\exp A$. There exists a residual subset $R\subseteq\CA$ such that for all $\eta\in R$, $f$ and $\eta\circ g$ are transverse.
\end{lemm}

\begin{proof}
The differential of the map
\[
\begin{array}{rccc}
H:&N\times P\times\CA & \longrightarrow & \MTM\\
 &(x,y,\eta)&\longmapsto & (f(x),(\eta\circ g)(y))
\end{array}
\]
is $d_{(x,y,\eta)}H(u,v,w)=(d_xf(u),d_y(\eta\circ g)(v)+d_\eta ev_{g(y)}^{\CA}(w))$. 

Let us show that $H$ is transverse to $\De_{\MTM}$: suppose $f(x)=(\eta\circ g)(y)=:z$ and let $a,b\in T_zM$. Since $ev_{g(y)}^{\CA}$ is a submersion there exists $w\in T_\eta\CA$ such that $d_\eta ev_{g(y)}^{\CA}(w)=b-a$. Then $(a,b)=(a,a)+d_{(x,y,\eta)}H(0,0,w)$.

Now, since $H$ is tranverse to the diagonal, we have that $H^{-1}(\De_{\MTM})$ is a smooth submanifold of $N\times P\times\CA$. Consider the projection $H^{-1}(\De_{\MTM})\rightarrow\CA$. From Sard's theorem it follows that the regular values of this map form a residual subset $R\subseteq \CA$. Then, given $\eta\in R$, for all $(x,y)\in N\times P$ such that $f(x)=\eta(g(y))=:z$ and for all $X\in A$ there exist $u\in T_xN$, $v\in T_yP$ such that $d_xf(u)=d_y(\eta\circ g)(v)+X_z$. Hence
\[
\mbox{im}\,d_xf+\mbox{im}\,d_y(\eta\circ g)=T_zM,
\] 
which proves that $f$ and $\eta\circ g$ are transverse maps.
\end{proof}

With suitable extra conditions, this result can be generalised to obtain strong transversality:

\begin{lemm}\label{lemm:genericstronglytransverse}
Let $(f,N),(g,P)$ be smooth maps to $M$ and let  $(\vp,W),(\ga,Z)$ be respective omega-maps such that the values $\dim W+\dim P$, $\dim N+\dim Z$ and $\dim W+\dim Z$ are all strictly smaller than $\dim M$. Let $A\subseteq C^\infty(TM)$ be as in Lemma \ref{lemm:surjev} and take $\CA=\exp A$. There exists a residual subset $R\subseteq\CA$ such that for all $\eta\in R$, the maps $f$ and $\eta\circ g$ are strongly transverse.
\end{lemm}

\begin{proof}
From Lemma \ref{lemm:generictransverse} we deduce that there exists a residual subset $R\subseteq\CA$ such that for all $\eta\in R$ the following pairs of maps are transverse: $(f,g)$, $(\vp,\eta\circ g)$, $(f,\eta\circ\ga)$ and $(\vp,\eta\circ\ga)$.

The transversality of the second pair tells us that $CS(\vp,\eta\circ g)$ is a smooth manifold of dimension $\dim W+\dim P-\dim M<0$, so it must be empty. Thus $\vp(W)\cap(\eta\circ g)(P)=\emptyset$. Since $\vp$ is an omega-map of $f$ it also holds that $\Om_f\cap(\eta\circ g)(P)=\emptyset$. Similarly, from the third pair we deduce that $f(N)\cap\Om_{\eta\circ g}=\emptyset$ and from the fourth that $\Om_f\cap\Om_{\eta\circ g}=\emptyset$. 

Finally we have that
\[
\left.
\begin{array}{r}
\Om_f\cap(\eta\circ g)(P)=\emptyset\\
\Om_f\cap\Om_{\eta\circ g}=\emptyset
\end{array}
\right\}
\Rightarrow\Om_f\cap((\eta\circ g)(P)\cup\Om_{\eta\circ g})=\emptyset\Rightarrow\Om_f\cap\ov{(\eta\circ g)(P)}=\emptyset,
\]
and $\ov{f(N)}\cap\Om_{\eta\circ g}=\emptyset$ is proved in the same way.
\end{proof}

From this result it follows that we can always assume two pseudocycles of complementary dimension to be strongly transverse by composing one of them with a generic diffeomorphism. Then we can make the following definition:

\begin{prop}\label{prop:intersectionproduct}
Let $(f,D),(g,E)$ be strongly transverse pseudocycles of $M$ such that $\dim D+\dim E=\dim M$. The value
\[
f\cdot g:=\sum_{(x,y)\in CS(f,g)}\nu(x,y),
\]
where $\nu(x,y)$ is the intersection number of $f(D)$ and $g(E)$ at $f(x)=g(y)$, is an integer that only depends on the bordism classes of $f$ and $g$.
\end{prop}

\begin{proof}
Since $f$ and $g$ are transverse, the intersection numbers $\nu(x,y)$ are well defined. Moreover, since $f$ and $g$ are strongly transverse and of complementary dimension, according to Lemma \ref{lemm:stcsfinite}, $CS(f,g)$ is a finite set, so the sum that defines $f\cdot g$ is finite and hence $f\cdot g$ is an integer number.

Let $(g',E')$ be a pseudocycle and let $(\tilde{g},\widetilde{E})$ be a bordism between $g$ and $g'$. Lemma \ref{lemm:genericstronglytransverse} shows that we can take $\tilde{g}$ strongly transverse to $f$. Then, from Lemma \ref{lemm:stcsfinite} it follows that $CS(f,\tilde{g})$ is a compact and oriented $1$-dimensional manifold. Its boundary is $CS(f,g')\cup-CS(f,g)$ because $\partial\widetilde{E}=E'\cup E$. Therefore $f\cdot g=f\cdot g'$. This proves that $f\cdot g$ only depends on the bordism class of $g$ but not on $g$ itself. A symmetric argument shows that it only depends on the bordism class of $f$ but not in $f$ itself.
\end{proof}

\begin{rema} 
Since the number $f\cdot g$ does not depend on the bordism class, we shall write it $[f]\cdot[g]$ from now on.
\end{rema}

Given $a\in H_q(M;\ZZ)$ we can consider the homomorphism
\[
\begin{array}{rccc}
I_a:& H_{m-q}(M;\ZZ) &\longrightarrow & \ZZ\\
    & b & \longmapsto & \Psi_*^M(a)\cdot\Psi_*^M(b)
\end{array}.
\]
Hence $I_a$ can be seen as an integral cohomology class modulo torsion:
\[
I_a\in\mbox{Hom}(H_{m-q}(M;\ZZ);\ZZ)\simeq H^{m-q}(M;\ZZ)/\mbox{Tor}.
\]
For integral (co)homology this is not a gain because we can take $PD_M^{-1}(a)\in H^{m-q}(M;\ZZ)$ and we do not lose any torsion information: note that when $(f,D)$ is a cycle of $M$ --i.e. $D$ is compact-- then $I_{\Phi_*^M[f]}=PD_M^{-1}f_*[D]$ or, more briefly, $\Phi^M_*[f]=f_*[D]$. 

On the other hand, for rational cohomology things are different because there is no torsion: via the natural inclusion of $\ZZ$ in $\QQ$ we can think 
\[
I_a\in\mbox{Hom}(H_{m-q}(M;\ZZ);\QQ)\simeq H^{m-q}(M;\QQ).
\]
Then any rational cohomology class is a linear combination with rational coefficients of elements of the form $I_{\Phi_*^M(\be)}$ with $\be\in\BB_*(M)$.

\begin{rema}
The notions of omega-limit sets and strong transversality can be generalised to branched manifolds and the transversality results proved here also apply in that case. We will use this in Chapter \ref{ch:biinvariant}.
\end{rema}

\subsection{Pseudocycles and products of $G$-manifolds}

It is very convenient to our purposes to understand how do pseudocycles behave with respect to Cartesian products of manifolds and also with respect to $G$-manifolds. We study these situations here.

\begin{prop}
Let $(f\times h,D)$ be a pseudocycle of $\MTN$. Then $(f,D)$ is a pseudocycle of $M$ and $(h,D)$ is a pseudocycle of $N$.
\end{prop}

\begin{proof}
Let $\pi_M,\pi_N$ denote the projections from $\MTN$ onto $M,N$. Then $\pi_M\circ(f\times h)=f$ and $\pi_N\circ(f\times h)=h$ are pseudocycles (see Remark \ref{rema:pseudocat}).
\end{proof}

For the converse of this proposition to hold we need to add some extra conditions:

\begin{prop}\label{prop:productpseudo}
Let $(f,D)$ with omega-map $(W_f,\vp)$ be a pseudocycle of $M$ and let $(h,D)$ with omega-map $(W_h,\ga)$ be a pseudocycle of $N$. Then $(f\times h,D)$ is a pseudocycle of $\MTN$ with omega-map $(\vp\times\ga,W_f\times W_h)$ provided that $\dim W_f+\dim W_h\leq\dim D-2$.
\end{prop}

\begin{proof}
Let $(p,q)\in\Om_{f\times h}$. Then there exists $d_j\in D$ with no convergent subsequences such that
\[
(f\times h)(d_j)=(f(d_j),h(d_j))\mathop{\longrightarrow}\limits_j(p,q).
\]
Therefore $f(d_j)\rightarrow p$ and $h(d_j)\rightarrow q$, so $p\in\Om_f\subseteq\vp(W_f)$ and $q\in\Om_h\subseteq\ga(W_h)$. In other words $\Om_{f\times h}$ is contained in the image of $\vp\times\ga:W_f\times W_h\rightarrow\MTN$. Since $\dim W_f\times W_h=\dim W_f+\dim  W_h\leq\dim D-2$, $f\times h$ is a pseudocycle.
\end{proof}

\begin{prop}\label{prop:equivpseudo}
Let $M$ be a $G$-manifold and let $(f,D)$ be a pseudocycle of $M$ with omega-map $(\vp,W)$. Assume that $G$ also acts on $D$ and $W$ and that $f$ and $\vp$ are equivariant maps. If $P$ is a compact smooth manifold where $G$ acts freely, the map
\[
\begin{array}{rccc}
f_G:&D\times_GP&\longrightarrow & M\times_GP\\
    &[d,p] &\longmapsto & [f(d),p]
\end{array}
\]
is a pseudocycle with omega-map
\[
\begin{array}{rccc}
\vp_G:&W\times_GP&\longrightarrow & M\times_GP\\
    &[w,p] &\longmapsto & [\vp(w),p]
\end{array}.
\]
\end{prop}

\begin{proof}
The maps $f_G$, $\vp_G$ are smooth maps between smooth manifolds because $G$ acts freely on $P$ and they are well-defined because $f$ and $\vp$ are equivariant. Since 
\[
\begin{array}{rcl}
\dim W\times_GP &  = & \dim W+\dim P-\dim G\\
                &\leq& \dim D-2+\dim P-\dim G\\
                &  = & \dim D\times_GP-2
\end{array}
\]
we only need to show that $\Om_{f_G}\subseteq\im\vp_G$:

Let $[m,p]\in\Om_f$. Then, there exist $[d_j,p_j]\in D\times_GP$ with no converging subsequences satisfying
\[
[f(d_j),p_j]\mathop{\longrightarrow}\limits_j[m,p].
\]
Thus there are $g_j\in G$ such that $g_jf(d_j)\rightarrow m$ and $g_jp_j\rightarrow p$. In particular $f(g_jd_j)\rightarrow m$ because $f$ is equivariant. We will see that the sequence $g_jd_j\in D$ has no convergent subsequences, implying that $m\in\Om_f$. Hence, there exists $w\in W$ such that $\vp(w)=m$, and we conclude that $\vp_G([w,p])=[m,p]$.

It only remains to justify why $g_jd_j$ has no convergent subsequences: since $G$ and $P$ are compact, the sequences $g_j\in G$ and $p_j\in P$ have some convergent subsequence with limit $g\in G$ and $p'\in P$, respectively. Assume that $g_jd_j$ had limit $d$. Then $[d_j,p_j]$ would converge to $[g^{-1}d,p']$, giving a contradiction with our assumptions.
\end{proof}

Now we prove a result that is, in a way, a combination of the two previous results:

\begin{prop}\label{prop:pseudocombi}
Let $M,N$ be $G$-manifolds. Let $(f,D)$ with omega-map $(W_f,\vp)$ be a pseudocycle of $M$ and let $(h,D)$ with omega-map $(W_h,\ga)$ be a pseudocycle of $N$.  Let $G$ act on $D$, $W_f$ and $W_h$ making $f$, $h$, $\vp$ and $\ga$ into equivariant maps. If $P$ is a compact smooth manifold where $G$ acts freely and $\dim W_f+\dim W_h\leq\dim D-2$, the map
\[
\begin{array}{rccc}
f_G\times h_G:&D\times_G P&\longrightarrow &(M\times_GP)\times(N\times_GP)\\
    & [d,p] &\longmapsto & ([f(d),p],[h(d),p])
\end{array}
\]
is a pseudocycle with omega-map
\[
\begin{array}{rccc}
(\vp\times\ga)_G:&(W_f\times W_h)\times_GP & \longrightarrow & (M\times_GP)\times(N\times_GP)\\
 &[(a,b),p]&\longrightarrow &([\vp(a),p],[\ga(b),p])
\end{array}.
\]
\end{prop}

\begin{proof}
The proof of this result is not an immediate combination of propositions \ref{prop:productpseudo} and \ref{prop:equivpseudo}: from Proposition \ref{prop:equivpseudo} we deduce that $f_G$, $h_G$ are pseudocycles with omega-maps $\vp_G$, $\ga_G$ respectively. We cannot apply \ref{prop:productpseudo} to conclude that $f_G\times h_G$ is a pseudocycle because dimension count fails: $\Om_{f_G\times h_G}$ is contained in the image of the map $\vp_G\times\ga_G$, but the dimension of the source space of this map is too large. We shall see instead that $\Om_{f_G\times h_G}$ is contained in the image of $(\vp\times\ga)_G$:

Let $([m,p],[n,p'])\in\Om_{f_G\times h_G}$. Then there exists $[d_j,p_j]\in D\times_GP$ with no converging subsequences such that
\[
(f_G\times h_G)([d_j,p_j])=([f(d_j),p_j],[h(d_j),p_j])\mathop{\longrightarrow}\limits_j([m,p],[n,p']).
\]
Therefore there exist $g_j,g_j'\in G$ such that
\[
\left\{
\begin{array}{ll}
g_jf(d_j)=f(g_jd_j)\rightarrow m, &g_jp_j\rightarrow p\\
g'_jh(d_j)=h(g'_jd_j)\rightarrow n, &g'_jp_j\rightarrow p'\\
\end{array}.
\right.
\]
Since neither $g_jd_j$ nor $g_j'd_j$ are have convergent subsequences (see the argument at the end of the proof of Proposition \ref{prop:equivpseudo}) we deduce that $m\in\Om_f$ and $n\in\Om_h$, so there exist $a\in W_f,b\in W_h$ such that $\vp(a)=m,\ga(b)=n$. From this we deduce that $[m,p]=[\vp(a),p]$ and $[n,p']=[\ga(b),p']$. 

Now, since $G$ is compact, taking subsequences if necessary, we can assume that $g_j\rightarrow g$, $g_j'\rightarrow g'$. Similarly, since $P$ is also compact we can assume that $p_j\rightarrow q$. Then $gq=p$ and $g'q=p'$, from what it follows that $p'=g'g^{-1}p$. Hence
\[
[n,p']=[\ga(b),g'g^{-1}p]=[g(g')^{-1}\ga(b),p]=[\ga(g(g')^{-1}b),p].
\]
In conclusion, we have $([m,p],[n,p'])=(\vp\times\ga)_G([(a,g(g')^{-1}b),p])$.
\end{proof}

One final result relates pseudocycles and shriek maps:

\begin{prop}\label{prop:exerciseeight}
Let $M,N,D$ be oriented and connected manifolds of dimensions $m,n,d$ and let $f:D\rightarrow M$, $g:D\rightarrow N$ be smooth maps. Consider the commutative diagram
\[
\begin{diagram}
\node{}\node{D}\arrow{sw,t}{f}\arrow{se,t}{g}\arrow{s,r}{f\times g}\node{}\\
\node{M}\node{M\times N}\arrow{e,b}{\pi_N}\arrow{w,b}{\pi_M}\node{N}
\end{diagram}.
\]
Assume that $M,N$ are compact and that $(f\times g,D)$ is a pseudocycle of $\MTN$. If $D$ is also compact,  for $\al\in H^*(N;\ZZ)$,
\[
f^!g^*(\al)=PD_M^{-1}(\pi_M)_*(\pi_N^*\al\frown\Phi^{\MTN}_*[f\times g])\in H^{*+m-d}(M;\ZZ).
\]
\end{prop}

\begin{proof}
The proof is quite straightforward if we realise that $f_*=(\pi_M)_*(f\times g)_*$ and $g^*=(f\times g)^*\pi_N^*$. Then
\[
\begin{array}{rcl}
f^!g^*\al & = & PD_M^{-1}f_*(g^*\al\frown[D])\\
          & = & PD_M^{-1}(\pi_M)_*(f\times g)_*((f\times g)^*\pi_N^*\al\frown [D])\\
          & = & PD_M^{-1}(\pi_M)_*(\pi_N^*\al\frown(f\times g)_*[D])\\
          & = & PD_M^{-1}(\pi_M)_*(\pi_N^*\al\frown\Phi^{\MTN}_*[f\times g])
\end{array}
\]
\end{proof}

\begin{rema}To define $f^!g^*$ we need $D$ to be compact, but the expression
\[
PD_M^{-1}(\pi_M)_*(\pi_N^*\al\frown\Phi^{\MTN}_*[f\times g])
\] 
only needs $f\times g$ to be a pseudocycle. Therefore it can be taken as a generalisation of $f^!g^*$ when $D$ is not compact.
\end{rema}

%% file: Standard.tex
\chapter{Standard perturbations}\label{ch:standard}


A {\it Morse-Bott pair} $(f,V)$ on a smooth manifold $M$ is formed by a Morse-Bott function $f$ and a gradient-like vector field $V$ for $f$. This chapter consists of two parts. In the first part we describe the moduli space of broken gradient lines of $(f,V)$. This moduli space admits a smooth stratification provided that the flow of $V$ satisfies certain transversality conditions. We prove that these conditions are generically satisfied. In the second part we study the particular case in which the Morse-Bott function is the moment map of a Hamiltonian circle action on a symplectic manifold. In this context it is natural to ask $V$ to be invariant under the action. We show that it may be impossible to find an invariant $V$ satisfying the desired transversality conditions, but that the sole obstruction to this fact is that the action has finite non-trivial stabilisers.

\section{Gradient lines in Morse-Bott theory}\label{sec:gradmorsebott}
\subsection{Stratification of the moduli space of broken gradient lines}\label{subsec:bgd}

Let $M$ be a compact and connected smooth manifold of dimension $m$ and let $f:M\rightarrow\RR$ be a Morse-Bott function. The set $F$ of critical points of $f$ decomposes into finitely many connected submanifolds of $M$. We denote them $C_1,\ldots,C_n$ in such a way that $f(C_1)\geq\cdots\geq f(C_n)$ and we let $c_i=\dim C_i$. Given $p\in C_i$ the Hessian of $f$ at $p$ satisfies $\ker H_pf=T_pC_i$ and the dimension $k_i$ of the space where $H_pf$ is negative-definite is called the {\it index} of $C_i$.

\begin{defi}\label{defi:gradlikevf}
A vector field $V\in C^\infty(TM)$ is a {\it gradient-like vector field} for $f$ if
\begin{enumerate}
\item For $p\in M\setminus F$, $d_pf(V_p)>0$. 
\item For $p\in C_i$ there are coordinates $(x_1,\ldots,x_m)$ in a neighbourhood $U$ of $p$ such that 
\[
f=f(p)-x_{c_i+1}^2-\cdots-x_{c_i+k_i}^2+x_{c_i+k_i+1}^2+\cdots+x_m^2
\] 
throughout $U$.
\end{enumerate}
\end{defi}

A pair $(f,V)$ formed by a Morse-Bott function $f$ and a gradient-like vector field $V$ for $f$ is called a {\it Morse-Bott pair}.

The {\it (un)stable} submanifolds of $(f,V)$ are defined as follows: if $\xi_t^V$ is the flow of $-V$ then
\[
S_i^V:=\{p\in M:\ \lim_{t\rightarrow +\infty}\xi_t^V(p)\in C_i\},
\]
\[
U_i^V:=\{p\in M:\ \lim_{t\rightarrow -\infty}\xi_t^V(p)\in C_i\}.
\]
The corresponding projections $s_i^V:S_i^V\rightarrow C_i$, $u_i^V:U_i^V\rightarrow C_i$ are submersions (see \cite[4.1]{HPS} for the proof of the fact that $S_i^V$, $U_i^V$ are submanifolds and that $s_i^V,u_i^V$ are submersions). Also, $\dim S_i^V=m-k_i$ and $\dim U_i^V=k_i+c_i$.

A gradient line of $(f,V)$ is the result of following the flow of $V$ from a point of $M$ for a certain amount of time. We want to parametrise such objects. The result will turn out to be a non-compact manifold. To compactify it we will need the notion of broken gradient line:

\begin{defi}\label{defi:brokenline}
An {\it oriented broken gradient line} of $(f,V)$ is a pair $(K,b)$, where $K\subseteq M$ is compact and $b\in K$, such that either $K=\{b\}$ or, if not, then:
\begin{enumerate}
\item $K\cap F$ is finite.
\item There exists a homeomorphism $h_K:[0,1]\rightarrow K$, smooth on $h_K^{-1}(K\setminus F)$, such that $\forall t\in h_K^{-1}(K\setminus F)$, $\exists \lm>0$ s.t. $h_K'(t)=-\lm V_{h_K(t)}$.
\item Either $b=h_K(0)$ or $b=h_K(1)$.
\end{enumerate}
The point $b$ is called the {\it beginning} of $(K,b)$. Similarly, we define the {\it end} of $(K,b)$ to be the point
\[
e=\left\{
\begin{array}{llccr}
     b &\mbox{if} & K=\{b\}    &            &  \\
h_K(1) &\mbox{if} & K\neq\{b\} & \mbox{and} & b=h_K(0)\\
h_K(0) &\mbox{if} & K\neq\{b\} & \mbox{and} & b=h_K(1)\\
\end{array}
\right..
\]
In the first case we say that $(K,b)$ is {\it point-like}, in the second case that it is {\it descending} and in the last case that it is {\it ascending}.
\end{defi}

The following figure depicts a typical broken gradient line:
\begin{figure}[H]
	\centering
		\includegraphics{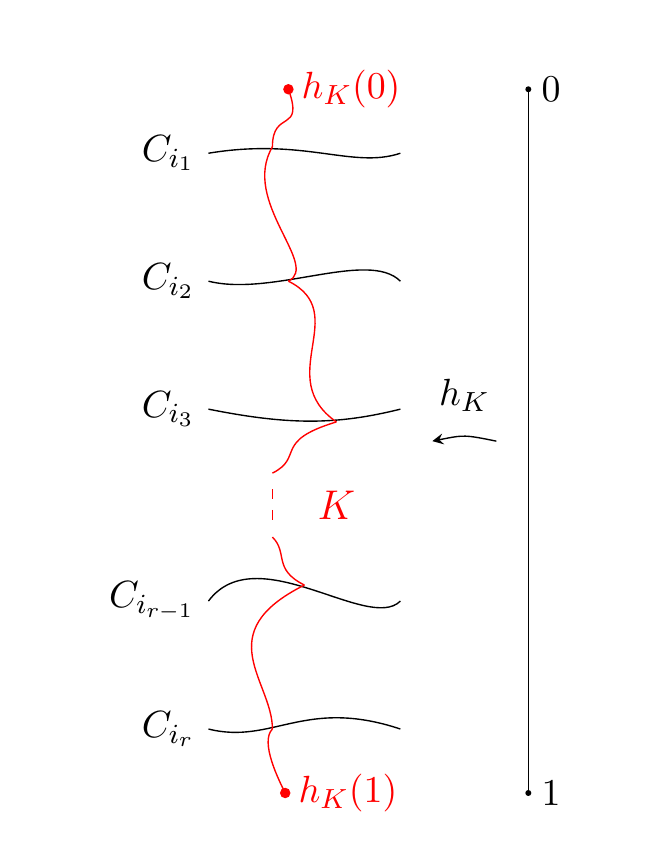}
	\caption{A broken gradient line.}
	\label{fig:obroken}
\end{figure}

We denote by $\MM^{f,V}$ the set of all broken gradient lines of $(f,V)$. We can define a distance on $\MM^{f,V}$ as follows: a Riemannian metric $\rho$ on $M$ determines a distance $d_\rho$ on $M$ and the corresponding Hausdorff distance $d^H_\rho$ between subsets of $M$. On $\MM^{f,V}$ we take the distance 
\[
d((K,b),(K',b'))=d_\rho^H(K,K')+d_\rho(b,b').
\]
The topology defined by this metric does not depend on the choice of $\rho$. Our goal is to find a smooth stratification on $\MM^{f,V}$:

\begin{defi} Let $X$ be a topological space. A {\it smooth stratification} of $X$ is a disjoint union
\[
X=\bigsqcup_{\sig\in\Sig}X_{\sig},
\]
where
\begin{enumerate}
\item Each $X_\sig$ has a structure of smooth manifold compatible with the topology induced by the topology of $X$.
\item There exists a partial order in $\Sig$ such that
\[
\ov{X}_\sig\subseteq X_{\sig}\cup\bigcup_{\tau>\sig}X_{\tau}.
\]
\end{enumerate}
\end{defi}

Given $I\subseteq\{1,\ldots,n\}$, define 
\[
\MM_I^{f,V}:=\{(K,b)\in\MM^{f,V}:\ i\in I\Leftrightarrow K\cap C_i\neq\emptyset\}, 
\]
the set of broken gradient lines with breaking points at those $C_i$ with $i\in I$. Then there is a disjoint union
\[
\MM^{f,V}=\bigsqcup_{I\subseteq\{1,\ldots,n\}}\MM_I^{f,V}.
\]
In order to achieve the first condition of smooth stratification it is necessary to further split them into smaller subsets: we can write
\[
\MM_{I}^{f,V}=\MM_{I\downarrow}^{f,V}\sqcup\MM_{I\uparrow}^{f,V}\sqcup\MM_{I-}^{f,V},
\]
according to the elements of $\MM_{I}^{f,V}$ being descending, ascending or point-like. Note that if $\# I>1$ an element cannot be point-like, so $\MM_I^{f,V}$ splits only into ascending and descending elements. On the other hand observe that the following maps are homeomorphisms:
\[
\begin{array}{ccc}
\MM_{\emptyset-}^{f,V} & \longrightarrow & M\setminus F\\
(\{b\},b)&\longmapsto & b
\end{array},
\hspace{1cm}
\begin{array}{ccc}
\MM_{\{i\}-}^{f,V} & \longrightarrow & C_i\\
(\{b\},b)&\longmapsto & b
\end{array}.
\]

\begin{rema}\label{rema:flip}$\MM_{I\downarrow}^{f,V}$ is homeomorphic to $\MM_{I\uparrow}^{f,V}$ via $(K,b)\mapsto(K,e)$, where $e$ is the end of $(K,b)$. Its inverse is defined in the same way. These transformations consist in flipping the orientation of the broken gradient line.\end{rema}

We will describe the smooth stratification on $\MM^{f,V}$ after a series of lemmas. In all of them we write
\[
\begin{array}{rccc}
\vp_p^V:&\RR&\longrightarrow & M\\
        &  t&\longmapsto     &\xi_t^V(p)
\end{array}
\]
for the integral curve of $-V$ starting at $p$.

\begin{lemm}\label{lemm:densestratum}
$\MM^{f,V}_\emptyset$ is homeomorphic to the $(m+1)$-dimensional manifold $\RR\times (M\setminus F)$.
\end{lemm}

\begin{proof} 
The map
\[
\begin{array}{ccc}
\MM_\emptyset^{f,V}&\longrightarrow &\RR\times(M\setminus F)\\
(K,b)&\longmapsto & ((\vp_b^V)^{-1}(e),b)
\end{array},
\]
where $e$ denotes the end of $(K,b)$ is continuous with continuous inverse given by
\[
(t,p)\longrightarrow(\vp_p^V([\min\{0,t\},\max\{0,t\}]),p).\qedhere
\]
\end{proof}

The set $\MM^{f,V}_\emptyset$ is the space of `genuine' gradient lines. It is open inside $\MM^{f,V}$.

Now we move on to characterising $M_{I}^{f,V}$ for $\#I=1$. Let $I=\{i\}$. Since $s_i^V,u_i^V$ are submersions it follows that $S_i^V\times_{C_i} U_i^V$ is an $m$-dimensional manifold. We have the following result:

\begin{lemm}\label{lemm:codimonestratum}
$\MM_{\{i\}-}^{f,V}$ is homeomorphic to $C_i$ and $\MM_{\{i\}\downarrow}^{f,V}$, $\MM_{\{i\}\uparrow}^{f,V}$ are homeomorphic to $(S_i^V\times_{C_i} U_i^V)\setminus\De_{C_i\times C_i}$.
\end{lemm}

\begin{proof} We already showed the homeomorphism $\MM_{\{i\}-}^{f,V}\simeq C_i$. Consider now the map
\[
\begin{array}{ccc}
\MM_{\{i\}\downarrow}^{f,V} &\longrightarrow &(S_i^V\times_{C_i} U_i^V)\setminus\De_{C_i\times C_i}\\
(K,b)&\longmapsto &(b,e)
\end{array}.
\]
where $e$ is the end of $(K,b)$. It is well defined because since any $(K,b)\in\MM_{\{i\}\downarrow}^{f,V}$ is not point-like, then $e\neq b$. It is a homeomorphism with inverse defined as follows: given $(p,q)\in(S_i^V\times_{C_i} U_i^V)\setminus\De_{C_i\times C_i}$ we have that 
\[
\lim_{t\rightarrow+\infty}\vp_p^V(t)=\lim_{t\rightarrow-\infty}\vp_q^V(t)=:\ell.
\]
Define 
\[
K=\vp_p^V(\RR_{\geq 0})\sqcup\{\ell\}\sqcup\vp_q^V(\RR_{\leq 0})
\]
and send $(p,q)$ to $(K,p)$. Finally, since $\MM_{\{i\}\uparrow}^{f,V}\simeq \MM_{\{i\}\downarrow}^{f,V}$ --see Remark \ref{rema:flip}-- we get all the desired homeomorphisms.
\end{proof}

Finally we characterise the case $\#I>1$: let $f(F)=\{a_1,\ldots,a_\nu\}$ with $a_1>\ldots>a_\nu$ be the image of the critical set $F$ under $f$. Fix regular levels $Y_{a_1}^-,\ldots,Y_{{a_\nu}-1}^-$ and $Y_{a_2}^+,\ldots,Y_{a_\nu}^+$ of $f$ such that $a_i>f(Y_{a_i}^-)>f(Y_{a_{i+1}}^+)>a_{i+1}$ for $i=1,\ldots,\nu-1$. The following figure depicts these levels in green and $F$ in red:
\begin{figure}[H]
	\centering
		\includegraphics{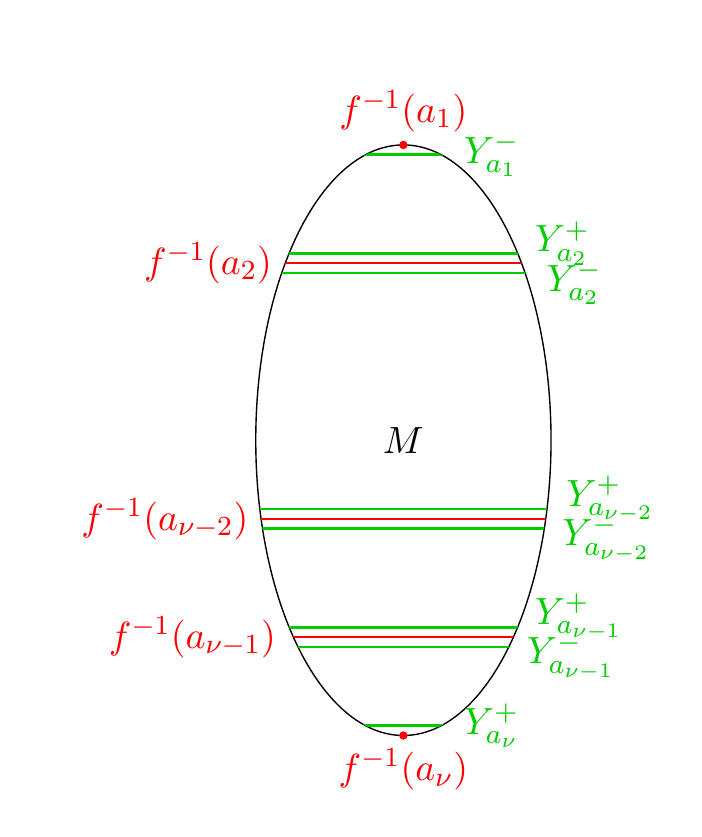}
	\caption{Auxiliary regular levels.}
	\label{fig:auxlevels}
\end{figure}

\begin{lemm}\label{lemm:codimtwostrata}
Let $I=\{i_1,\ldots,i_r\}$ with $r>1$ and $f(C_{i_1})>\cdots>f(C_{i_r})$. Then $\MM_{I\downarrow}^{f,V}$ and $\MM_{I\uparrow}^{f,V}$ are homeomorphic to the fibred product 
\[
\WW_I^{f,V}:=S_{i_1}^V\times_{C_{i_1}}(U_{i_1}^V\cap S_{i_2}^V\cap Y_{f(C_{i_2})}^+)\times_{C_{i_2}}\cdots\times_{C_{i_{r-1}}}(U_{i_{r-1}}^V\cap S_{i_r}^V\cap Y_{f(C_{i_r})}^+)\times_{C_{i_r}}U_{i_r}^V.
\]
\end{lemm} 

\begin{proof}
Given $(K,b)\in\MM_{I\downarrow}^{f,V}$ denote by $y^K_j$ the only point in $K\cap Y_{f(C_{i_j})}^+$. The continuous map
\[
\begin{array}{ccc}
\MM_{I\downarrow}^{f,V} &\longrightarrow & \WW_I^{f,V}\\	
(K,b) & \longmapsto & (b,y^K_2,\cdots,y^K_r,e)
\end{array}
\]
where $e$ is the end of $(K,b)$ admits a continuous inverse constructed as follows: for $(p_1,\ldots,p_{r+1})\in\WW_I^{f,V}$ we define
\[
\ell_j:=\lim_{t\rightarrow+\infty}\vp_{p_j}^V(t)=\lim_{t\rightarrow-\infty}\vp_{p_{j+1}}^V(t)\in C_{i_j}\ (j=1,\ldots,r)
\]
and 
\[
K_{(p_1,\ldots,p_{r+1})}=\vp_{p_1}^V(\RR_{\geq0})\sqcup\{\ell_1\}\sqcup\vp_{p_2}^V(\RR)\sqcup\cdots\sqcup\{\ell_{r-1}\}\sqcup\vp_{p_r}^V(\RR)\sqcup\{\ell_r\}\sqcup\vp_{p_{r+1}}^V(\RR_{\leq 0}).
\]
Then the inverse is given by sending $(p_1,\ldots,p_{r+1})$ to $(K_{(p_1,\ldots,p_{r+1})},p_1)$. This shows that the set of descending broken gradient lines is homeomorphic to $\WW_{I}^{f,V}$. The same is true for ascending lines thanks to Remark \ref{rema:flip}.
\end{proof}

In general $\WW_I^{f,V}$ is not a smooth manifold unless some transversality condition is satisfied. For $i>j$, consider the open subset of $Y_{a_i}^-$ defined by 
\[
Y_{{a_ia_j}}^V:=\{p\in Y_{a_i}^-:\ \exists t\ \mbox{s.t.}\ \xi_t^V(p)\in Y_{a_j}^+\}
\]
and let $\psi_{a_ia_j}^V:Y_{a_ia_j}^V\rightarrow Y_{a_j}^+$ be the map defined by following the flow $\xi^V$.

\begin{defi}\label{defi:regularity} Let $I=\{i_1,\ldots,i_r\}$ with $r>1$ and $f(C_{i_1})>\cdots>f(C_{i_r})$. The pair $(f,V)$ is {\it $I$-regular} if the map $H_I^{f,V}$ from
\[
\begin{array}{l}
{\cal S}_I^{f,V}:=S_{i_1}^V\times(U_{i_1}^V\cap Y_{f(C_{i_1})f(C_{i_2})}^V)\times (S_{i_2}^V\cap Y_{f(C_{i_2})}^+)\times\cdots\\
\hspace{3cm}\cdots\times(U_{i_{r-1}}^V\cap Y_{f(C_{i_{r-1}})f(C_{i_r})}^V)\times(S_{i_r}^V\cap Y_{f(C_{i_r})}^+)\times U_{i_r}^V
\end{array}
\]
to 
\[
{\cal T}_I^f:=(C_{i_1}\times C_{i_1})\times\cdots\times(C_{i_r}\times C_{i_r})\times(Y_{f(C_{i_2})}^+\times Y_{f(C_{i_2})}^+)\times\cdots\times(Y_{f(C_{i_r})}^+\times Y_{f(C_{i_r})}^+)
\]
defined by sending $(p_1,q_1,\cdots,p_r,q_r)$ to
\[
\begin{array}{l}
(s_{i_1}^V(p_1),u_{i_1}^V(q_1),\cdots,s_{i_r}^V(p_r),u_{i_r}^V(q_r),\\
\hspace{3cm}\psi_{f(C_{i_1})f(C_{i_2})}^V(q_1),p_2,\ldots,\psi_{f(C_{i_{r-1}})f(C_{i_r})}^V(q_{r-1}),p_r)
\end{array}
\]
is transverse to
\[
{\bf\De}_I:=\De_{C_{i_1}\times C_{i_1}}\times\cdots\times\De_{C_{i_r}\times C_{i_r}}\times\De_{Y_{f(C_{i_2})}^+\times Y_{f(C_{i_2})}^+}\times\cdots\times\De_{Y_{f(C_{i_r})}^+\times Y_{f(C_{i_r})}^+}.
\] 
The pair $(f,V)$ is {\it regular} if it is $I$-regular for every $I\subseteq\{1,\ldots,n\}$ with $\#I>1$.
\end{defi}

\begin{rema}
Note that the regularity of $(f,V)$ does not depend on the choice of the auxiliary regular levels (Figure \ref{fig:auxlevels}) because following the flow $\xi^V$ gives a diffeomorphism between any two regular levels with no critical points in between.
\end{rema}

The following lemma shows why regularity is the transversality condition we need:

\begin{lemm}\label{lemm:homeo} Let $I=\{i_1,\ldots,i_r\}$ with $r>1$ and $f(C_{i_1})>\cdots>f(C_{i_r})$. If $(f,V)$ is $I$-regular, then
\begin{enumerate}
\item $(H_I^{f,V})^{-1}({\bf\De}_I)$ is a smooth manifold of dimension $m-r+1$.
\item $\WW_I^{f,V}$ is homeomorphic to $(H_I^{f,V})^{-1}({\bf\De}_I)$.
\end{enumerate}
\end{lemm}

\begin{proof} Since $H_I^{f,V}$ is transverse to ${\bf\De}_I$, we have that $(H_I^{f,V})^{-1}({\bf\De}_I)$ is a smooth manifold of dimension $\dim{\cal S}_I^{f,V}-\dim {\bf\De}_I$. We have 
\[
\begin{array}{rcl}
\dim{\cal S}_I^{f,V} & = & \displaystyle (m-k_{i_1})+\sum_{j=2}^r(m-k_{i_j}-1)+\sum_{j=1}^{r-1}(k_{i_j}+c_{i_j}-1)+(k_{i_r}+c_{i_r})\\
                 & = & \displaystyle \sum_{j=1}^rc_{i_j}+rm-2(r-1)
\end{array},
\]
while $\dim {\bf\De}_I=\sum_{j=1}^rc_{i_j}+(r-1)(m-1)$. The difference of the two values is $m-r+1$, which proves the first part of the lemma.

The homeomorphism claimed in the second part is
\[
\begin{array}{ccc}
(H_I^{f,V})^{-1}({\bf\De}_I) & \longrightarrow & \WW_I^{f,V}\\
(p_1,q_1,\ldots,p_r,q_r) & \longmapsto    &(p_1,\ldots,p_r,q_r)
\end{array}
\]
with inverse given by sending $(p_1,\ldots,p_{r+1})$ to 
\[
(p_1,(\psi_{f(C_{i_1})f(C_{i_2})}^V)^{-1}(p_2),\ldots,p_{r-1},(\psi_{f(C_{r-1})f(C_{i_r})}^V)^{-1}(p_r),p_r,p_{r+1}).\qedhere
\]
\end{proof}

Now we can describe the smooth stratification structure of the {\it moduli space of broken gradient lines}:

\begin{theo}\label{theo:stratification}
If $(f,V)$ is a regular Morse-Bott pair, then $\MM^{f,V}$ is a compact topological space that admits a smooth stratification with strata
\begin{itemize}
\item $\MM_{\emptyset}^{f,V}$ of dimension $m+1$.
\item $\MM_{I\downarrow}^{f,V},\MM_{I\uparrow}^{f,V}$ for non-empty $I\subseteq\{1,\ldots,n\}$, of dimension $m-\#I+1$.
\item $\MM_{\{i\}-}^{f,V}$ for $i\in\{1,\ldots,n\}$, of dimension $c_i$.
\end{itemize}
\end{theo}

\begin{proof}
That the strata are smooth manifolds of the claimed dimensions follows from the lemmas. The described strata are indexed by the set $\Sig$ that contains the following elements:
\begin{itemize}
\item The empty set $\emptyset$.
\item $(I,\downarrow),(I,\uparrow)$ for non-empty $I\subseteq\{1,\ldots,n\}$.
\item $(\{i\},-)$ for $i\in\{1,\ldots,n\}$.
\end{itemize}
On $\Sig$ we define a partial order according to the following rules:
\begin{itemize}
\item $\emptyset<\sig$ for all $\sig\in\Sig\setminus\{\emptyset\}$.
\item $(\{i\},\downarrow)<(\{i\},-)$ and $(\{i\},\uparrow)<(\{i\},-)$ for all $i\in\{1,\ldots,n\}$.
\item $(I,\downarrow)<(J,\downarrow)$ and $(I,\uparrow)<(J,\uparrow)$ for all $I,J\neq\emptyset$ such that $I\subsetneq J$.
\end{itemize}

Since gradient lines converge to broken gradient lines \cite[Appendix A]{AuBr}, the second condition to have a smooth stratification is also satisfied and $\MM^{f,V}$ must be compact.
\end{proof}

\subsection{Perturbations of the gradient-like vector field}\label{subsec:perturb}

Regularity is essential to get a smooth stratification of the moduli space of broken gradient lines. In the present section we prove that for a fixed $(f,V)$ regularity is achieved for a generic choice of a perturbation of the gradient-like vector field $V$.

Let $(f,V)$ be a Morse-Bott pair. The first thing we want to do is to perturb $V$ in such a way that the result still is a gradient-like vector field. In order to do so we need to impose some control on what perturbations can be taken. In particular we set a {\it perturbation zone} as follows: take $Z_{a_2},\ldots,Z_{a_\nu}$ regular levels of $f$ and fix $\be>0$ such that $a_{i-1}>f(Z_{a_i})+\be>f(Z_{a_i})-\be>a_i$ for every $i=2,\ldots,\nu$. Take $Z'_{a_i}=f^{-1}((f(Z_{a_i})-\be,f(Z_{a_i})+\be))$ a band around each of these regular levels and define the perturbation zone as the disjoint union $Z'=\bigcup_{i=2}^\nu Z'_{a_i}$. Note that that the closure of $Z'$ is contained in $M\setminus F$. The following figure depicts $Z'$ in blue and $F$ in red:
\begin{figure}[H]
	\centering
		\includegraphics{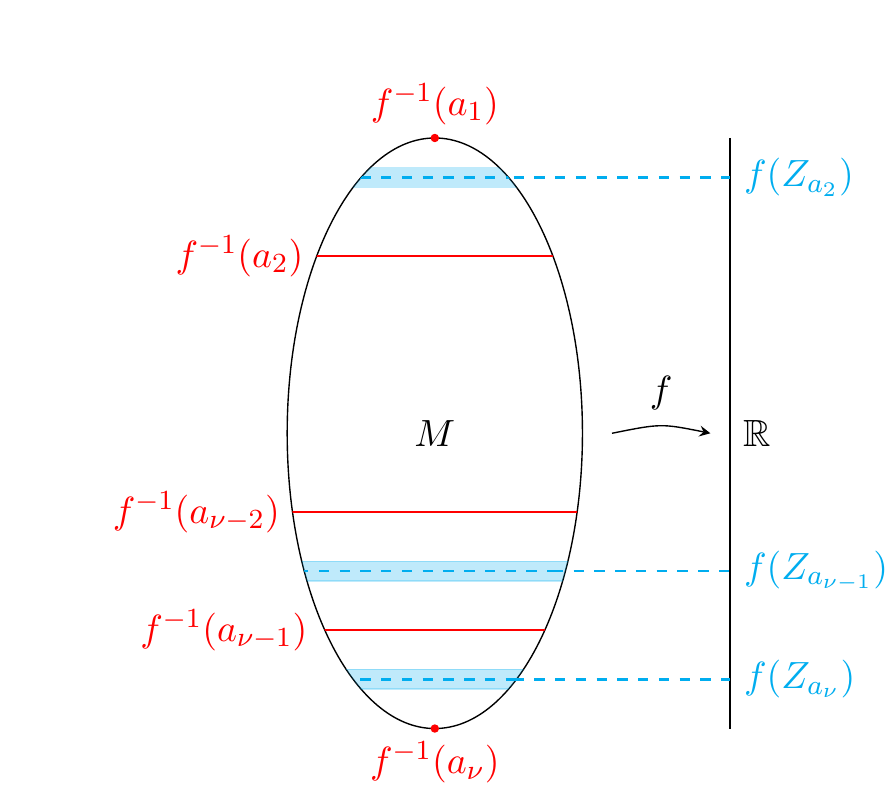}
	\caption{The perturbation zone.}
	\label{fig:perturbationzone}
\end{figure}

The reason to define $Z'$ as we did may be clearer after the following result:

\begin{prop}\label{prop:perturbation}
Let $E$ be a vector field supported on $Z'$ such that 
\[
\sup_{Z'}|df(E)|<\inf_{Z'}df(V).
\]
Then $V+E$ is a gradient-like vector field for $f$.
\end{prop}

\begin{proof} We shall see that $V+E$ satisfies the two conditions described in Definition \ref{defi:gradlikevf}:

Let $p\in M\setminus F$. From the hypothesis it follows that $|d_pf(E_p)|<d_pf(V_p)$ or, equivalently, $d_pf(V_p)-|d_pf(E_p)|>0$. Therefore
\[
d_pf((V+E)_p)=d_pf(V_p)+d_pf(E_p)\geq d_pf(V_p)-|d_pf(E_p)|>0.
\]
This shows that $V+E$ satisfies the first condition. The second condition is satisfied because $V+E$ coincides with $V$ near the critical set $F$.
\end{proof}

Using arguments similar to those in Lemma \ref{lemm:surjev} we can construct linear subspaces
\[
L_{a_i}\subseteq \{E\in\CH^\infty(TM):\ \mbox{supp}\ E\subseteq Z_{a_i}'\}
\]
such that for all $z\in Z_{a_i}$ there is a neighbourhood $U_z\subseteq Z_{a_i}'$ such that $\{E_x: E\in L_{a_i}\}$ generates $T_xM$ for all $x\in U_z$. We take $L_{a_i}^0$ a neighbourhood of $0\in L_{a_i}$ such that for all $E\in L_{a_i}^0$ it is satisfied that $\sup_{Z'_{a_i}}|df(E)|<\inf_{Z_{a_i}'}df(V)$. Then $V+E$ is a gradient-like vector field for $f$ for any $E\in L_{a_i}^0$.

For $q\in M\setminus F$ and $\tau>0$ such that $f(q)>f(Z_{a_i})+\be$ and such that $f(\xi^V_\tau(q))<f(Z_{a_i})-\be$ define
\[
\begin{array}{rccc}
\Phi_{q,\tau}:&L_{a_i}^0&\longrightarrow &M\\
     &E&\longmapsto &\xi_\tau^{V+E}(q)
\end{array}.
\]
Our next goal is to prove that $d_0\Phi_{q,\tau}$ is surjective, and hence that $\Phi_{q,\tau}$ is a submersion in a neighbourhood of $0\in L_{a_i}^0$. The geometric idea behind the definition of this map and the result we pursue is that if an integral curve of $V$ crosses the perturbation zone, then there is a perturbation that deviates the curve into any desired direction. Put in precise terms, the surjectivity of $d_0\Phi_{q,\tau}$ is equivalent to the fact that given any $v\in T_{\xi_\tau^V(q)}M$ there exists $E\in L_{a_i}^0$ such that $v=\frac{d}{ds}_{|s=0}\xi_t^{V+sE}(q)$. This holds because $\xi_t^{V+sE}(q)=\Phi_{q,\tau}(sE)$.

To prove that $d_0\Phi_{q,\tau}$ is surjective we will use the following convexity result in $\RR^k$:

\begin{lemm}\label{lemm:convex}
Let
\[
\begin{array}{rccc}
g:&[a,b]&\longrightarrow&\RR^k\\
    &s&\longmapsto & (g_1(s),\ldots,g_k(s))
\end{array}
\]
be a continuous map. The average of $g$,
\[
A_g=\frac{1}{b-a}\left(\int_a^bg_1(s)\,ds,\ldots,\int_a^bg_k(s)\,ds\right)\in\RR^k,
\]
is contained in the convex hull of \mbox{\rm im} $g$.
\end{lemm}

\begin{proof}
First of all note that since $g$ is continuous and $[a,b]$ is compact, the image $\im g$ is compact. Then the convex hull $CH(\im g)$ is also compact.

Now, given $N\geq 1$, take $x_j=a+j\frac{b-a}{N}$ for $j=0,\ldots,N$. This defines a partition of the interval $[a,b]$. The expression 
\[
A_{g,N}=\frac{1}{N}\sum_{j=1}^Ng(x_j)
\] 
satisfies the following two properties:
\begin{enumerate}
\item $A_{g,N}$ is the barycentre of the points $g(x_1),\ldots,g(x_N)$.
\item $(b-a)A_{g,N}$ is the right Riemann sum of $g$ associated to the partition $x_0<\cdots<x_N$.
\end{enumerate}
From the first property we deduce $A_{g,N}\in CH(g(x_1),\ldots,g(x_N))\subseteq CH(\im g)$ and from the second one that $A_g=\lim_{N\rightarrow\infty}A_{g,N}$. Combining both deductions with the fact that $CH(\im g)$ is compact we get that $A_g\in CH(\im g)$.
\end{proof}

Now we can prove the claimed result:

\begin{prop}\label{prop:phisurj}
Given $q\in M\setminus F$ and $\tau>0$ such that $f(q)>f(Z_{a_i})+\be$ and $f(\xi^V_\tau(p))<f(Z_{a_i})-\be$, the map
\[
\begin{array}{rccc}
\Phi_{q,\tau}:&L_{a_i}^0&\longrightarrow &M\\
     &E&\longmapsto &\xi_\tau^{V+E}(q)
\end{array}
\]
satisfies that $d_0\Phi_{q,\tau}$ is surjective and hence it is a submersion in a neighbourhood of $0\in L_{a_i}^0$.
\end{prop}

\begin{proof}
Let $\vp_q^V:\RR\rightarrow M$ the integral curve of $-V$ starting at $q$. The intersection $\vp_q^V([0,\tau])\cap Z_{a_i}$ consists of a unique point that we denote by $z$. Let $\tau_z>0$ be the value that makes $z=\vp_q^V(\tau_z)$. Let $\de>0$ and let $\eta_z:O_z\subseteq\RR^{m-1}\rightarrow Z_{a_i}$ be a coordinate neighbourhood of $z$ in $Z_{a_i}$ centred at $0$. The map 
\[
\io:(-\de-\tau_z,\tau+\de-\tau_z)\times O_z\rightarrow M
\] 
defined by $\io(t,u)=\xi_t^V(\eta_z(u))$ is an embedding provided that $\de$ and $O_z$ are small enough. Moreover, $\im \io$ is a neighbourhood of $\vp_q^V([0,\tau])$ in $M$ and the integral curves of $d\io^{-1}(V)$ are of the form $(t,u_0)$ for some constant $u_0\in O_z$.

In these coordinates it is satisfied that $q=(0,\ldots,0)$, $-V=\frac{\partial}{\partial x_1}$ and $\vp_q^V(t)=(t,0,\ldots,0)$. Also, for $E\in L_{a_i}^0$ we have $E=\sum_{k=1}^mg_E^k\frac{\partial}{\partial x_k}$ for some smooth functions $g^k_E$. Hence, the integral curve of $V+E$ beginning at $q$ is defined by
\[
\left\{
\begin{array}{l}
\ga_E(0)=(0,\ldots,0)\\
\ga_E'(t)=(1,0,\ldots,0)+(g^1_E(\ga_E(t)),\ldots,g^m_E(\ga_E(t)))
\end{array}
\right..
\]
In particular $\vp_q^V=\ga_0$.

Given a small $\de>0$ and $s\in(-\de,\de)$ consider the curve in $L_{a_i}^0$ defined by $\al_E(s)=sE$. Then 
\[
d_0\Phi_{q,\tau}(\al_E'(0))=(\Phi_{q,\tau}\circ\al_E)'(0)=\frac{d}{ds}_{|s=0}\ga_{sE}(\tau).
\] 
In view of this equality we are interested in computing $\ga_{sE}(t)$. Note first that $g_{sE}^k=sg_E^k$ and therefore that
\[
\ga_{sE}'(t)=(1,0,\cdots,0)+s(g^1_E(\ga_{sE}(t)),\ldots,g^{2m}_E(\ga_{sE}(t))).
\]

If $s>0$, consider the function
\[
\begin{array}{ccc}
[0,s]&\longrightarrow&\RR\\
r&\longmapsto &g^k_E(\ga_{rE}(t))
\end{array}.
\]
From the mean value theorem we get $g^k_E(\ga_{sE}(t))-g^k_E(\ga_0(t))=s(g^k_E\circ\ga_{rE})'(c)$ for some fixed $c\in [0,s]$. This can be written more succinctly as
\[
g^k_E(\ga_{sE}(t))=g^k_E(\ga_0(t))+O(s).
\] 
For $s<0$ we apply the argument to the interval $[s,0]$ to get the same result. Substituting in the expression of $\ga_{sE}'(t) $ we get 
\[
\ga_{sE}'(t)=(1,0,\ldots,0)+s(g^1_E(\ga_0(t))+O(s),\ldots,g^m_E(\ga_0(t))+O(s)).
\]
Integrating, and using the initial condition to determine constants, we find out that
\[
\ga_{sE}(t)=(t,0,\cdots,0)+s\left(\int_0^tg^1_E(\ga_0(r))dr+O(s),\ldots,\int_0^tg^m_E(\ga_0(r))dr+O(s)\right). 
\]
This means that
\[
\frac{d}{ds}_{|s=0}\ga_{sE}(\tau)=\left(\int_0^\tau g^1_E(\ga_0(r))dr,\int_0^\tau g^2_E(\ga_0(r))dr,\ldots,\int_0^\tau g^m_E(\ga_0(r))dr\right).
\] 
To conclude the proof we only need to see that any vector in $\RR^m$ is of this form for some $E\in L_{a_i}^0$. Observe that if $g_E$ denotes the map with component functions $g_E^k$ what we obtained is precisely 
\[
\frac{d}{ds}_{|s=0}\ga_{sE}(\tau)=\tau A_{g_E\circ\ga_0},
\] 
where $A_{g_E\circ\ga_0}\in\RR^m$ is the average of $g_E\circ\ga_0$. Hence all we need to do is to prove that
\[
\begin{array}{ccc}
L_{a_i}^0&\longrightarrow &\RR^m\\
E&\longmapsto &A_{g_E\circ\ga}
\end{array}
\]
is onto.

From the construction of $L_{a_i}^0$ it follows that there exists $[a,b]\subseteq(\vp_p^V)^{-1}(Z_{a_i}')$ and $E_1,\ldots,E_m\in L_{a_i}^0$ such that for all $t\in[a,b]$,
\[
\la(E_1)_{\vp_p^V(t)},\ldots,(E_m)_{\vp_p^V(t)}\ra=T_{\vp_p^V(t)}M.
\]
From this we get functions $g_{E_1}\circ\ga_0,\ldots,g_{E_m}\circ\ga_0:[0,\tau]\rightarrow\RR^m$, all of them supported on $(\vp_p^V)^{-1}(Z_{a_i}')$. Multiplying them by a plateau function with small support contained in $[a,b]$ we can assume that all these functions are supported on $[a,b]$ and that the cones
\[
C_j:=\{\lm x:\ \lm\geq 0,\ x\in\im (g_{E_j}\circ\ga)\}\subseteq\RR^m
\]
are such that $C_j\cap C_{j'}=\{0\}$ for any $j\neq j'$ in $\{1,\ldots,m\}$. We have that
\[
A_{g_{E_j}\circ\ga_0}\in CH(\im (g_{E_j}\circ\ga_0))\subseteq C_j,
\]
thanks to Lemma \ref{lemm:convex}. Since the cones have pairwise zero intersection it must be the case that $A_{g_{E_1}\circ\ga_0},\ldots,A_{g_{E_m}\circ\ga_0}$ are linearly independent in $\RR^m$.
\end{proof}

As in the previous section, we take auxiliary regular levels $Y_{a_1}^-,\ldots,Y_{a_{\nu-1}}^-$ and $Y_{a_2}^+,\ldots,Y_{a_\nu}^+$ of $f$. We choose them so that, for $i=2,\ldots,\nu$,
\[
a_{i-1}>f(Y_{a_{i-1}}^-)>f(Z_{a_i})+\be\ \mbox{and}\ f(Z_{a_i})-\be>f(Y_{a_i}^+)>a_i.
\]
Hence they are away from the closure of $Z'$ as it is shown in the following figure (which essentially combines figures \ref{fig:auxlevels} and \ref{fig:perturbationzone}):
\begin{figure}[H]
	\centering
		\includegraphics{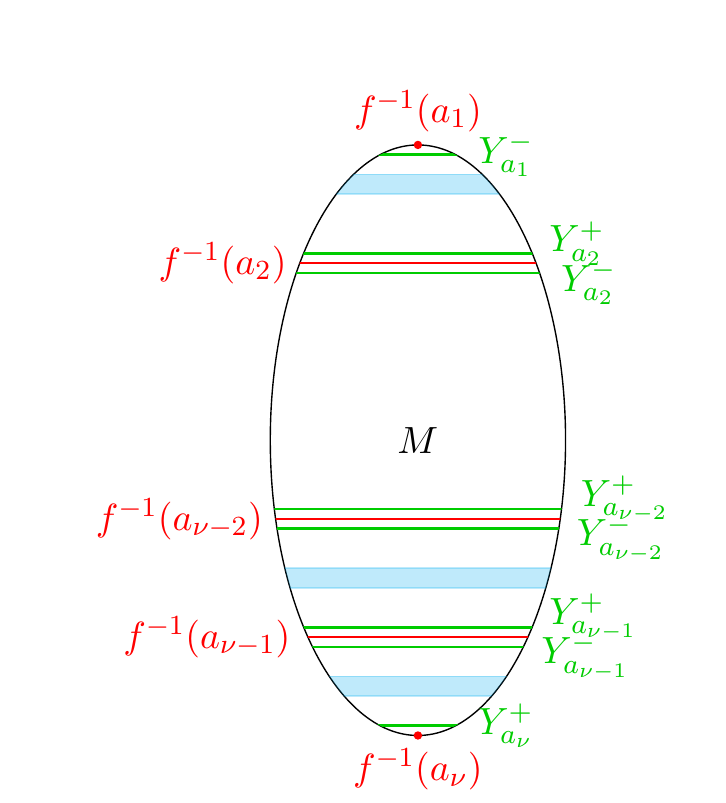}
	\caption{Auxiliary regular levels away from the perturbation zone.}
	\label{fig:auxlevelspertzone}
\end{figure}

Remember also from the previous section the notation
\[
Y_{a_ia_j}^{V+E}=\{p\in Y_{a_i}^-:\ \exists t\ \mbox{s.t.}\ \xi_t^{V+E}(p)\in Y_{a_j}^+\}
\]
for $i>j$ and the map $\psi^{V+E}_{a_ia_j}:Y_{a_ia_j}^{V+E}\rightarrow Y_{a_j}^+$ defined by following the flow $\xi^{V+E}$. Since there are no critical levels between $Y_{a_{i-1}}^-$ and $Y_{a_i}^+$ we have that $Y_{a_{i-1}a_i}^{V+E}=Y_{a_{i-1}}^-$ for any $E\in L_{a_i}^0$ and therefore $\psi^{V+E}_{a_{i-1}a_i}:Y_{a_{i-1}}^-\rightarrow Y_{a_i}^+$. Given $p\in Y_{a_{i-1}}^-$  We have the following:

\begin{lemm}\label{lemm:psisurj} Let $p\in Y_{a_{i-1}}^-$. The differential of the map 
\[
\begin{array}{rccc}
\Psi_p:&L_{a_i}^0&\longrightarrow &Y_{a_i}^+\\
 &E&\longmapsto&\psi_{a_{i-1}a_i}^{V+E}(p)
\end{array}
\]
at $0\in L_{a_i}^0$ is surjective and hence $\Psi_p$ is a submersion around that point.
\end{lemm}

\begin{proof}
Let $t_{p,E}>0$ be the value satisfying satisfying $\psi^{V+E}_{a_{i-1}a_i}(p)=\xi^{V+E}_{t_{p,E}}(p)$ and define
\[
\hspace{1cm}
\begin{array}{rccc}
\sig_p:&L_{a_i}^0&\longrightarrow &\RR\\
 &E&\longmapsto&t_{p,E}
\end{array}.
\]

Take the map
\[
\begin{array}{rccc}
G_p:&L_{a_i}^0\times\RR&\longrightarrow &M\\
 &E&\longmapsto&\xi^{V+E}_t(p)
\end{array}.
\]
Then its restrictions to each factor are $G_p(\cdot,t)=\Phi_{p,t}$ and $G_p(E,\cdot)=\vp_p^{V+E}$. Note also that $\Psi_p(E)=G_p(E,\sig_p(E))$. From these facts it follows that
\[
d_0\Psi_p(\ep)=d_0\Phi_{p,t_{p,E}}(\ep)-d_0\sig_p(\ep)\cdot(V+E)_{\psi_{a_{i-1}a_i}^{V+E}(p)},
\]
which rearranging gives
\[
d_0\Phi_{p,t_{p,E}}(\ep)=d_0\Psi_p(\ep)+d_0\sig_p(\ep)\cdot(V+E)_{\psi_{a_{i-1}a_i}^{V+E}(p)}.
\]
Since gradient-like vector fields of $f$ are transverse to regular levels we have that
\[
T_{\psi_{a_{i-1}a_i}^{V+E}(p)}M=T_{\psi_{a_{i-1}a_i}^{V+E}(p)}Y_i^+\oplus\la(V+E)_{\psi_{a_{i-1}a_i}^{V+E}(p)}\ra.
\]
From this fact and the equality above we deduce that $d_0\Psi_p(E)$ is just the projection of $d_0\Phi_{p,t_{p,E}}(\ep)$ onto $T_{\psi_{a_{i-1}a_i}^{V+E}(p)}Y_{a_i}^+$. Since $d_0\Phi_{p,t_{p,E}}$ is surjective so it must be $d_0\Psi_p$.
\end{proof}

By compactness of regular levels, shrinking $L_{a_i}^0$ a bit if necessary we have that $\Psi_p$ is a submersion for all $p\in Y_{a_{i-1}}^-$. Define
\[
\PP\subseteq L_{a_2}\oplus\cdots\oplus L_{a_\eta}.
\] 
The last result we want to prove in this section is that for a generic choice of $E\in\PP$, $(f,V+E)$ is a regular Morse-Bott pair. In preparation for this we introduce a bit more of notation:

We combine all the possible (un)stable manifolds of a critical component into the single manifolds
\[
S_i:=\{(p,E)\in M\times\PP:\ p\in S_i^{V+E}\},
\]
\[
U_i:=\{(p,E)\in M\times\PP:\ p\in U_i^{V+E}\}.
\]
and consider the natural projections $s_i:S_i\rightarrow C_i$, $u_i:U_i\rightarrow C_i$ defined by $s_i(p,E)=s_i^{V+E}(p)$, $u_i(p,E)=u_i^{V+E}(p)$ respectively.
Similarly, for $i>j$ define
\[
Y_{a_ia_j}:=\{(p,E)\in Y_{a_i}^-\times\PP:\ p\in Y_{a_ia_j}^{V+E}\}
\]
and let $\psi_{a_ia_j}:Y_{a_ia_j}\rightarrow Y_{a_j}^+$ be defined by $\psi_{a_ia_j}(p,E)=\psi_{a_ia_j}^{V+E}(p)$. 

As usual let $I=\{i_1,\ldots,i_r\}$ be such that $r>1$ and $f(C_{i_1})>\cdots>f(C_{i_r})$. Let ${\cal S}_I^f$ denote the fibre bundle over $\PP$ with total space
\[
\begin{array}{l}
S_{i_1}\times_\PP(U_{i_1}\cap Y_{f(C_{i_1})f(C_{i_2})})\times_\PP(S_{i_2}\cap (Y_{f(C_{i_2})}^+\times\PP))\times_\PP\cdots\\
\hspace{1.5cm}\cdots\times_\PP(U_{i_{r-1}}\cap Y_{f(C_{i_{r-1}})f(C_{i_r})})\times_\PP(S_{i_r}\cap (Y_{f(C_{i_r})}^+\times\PP))\times_\PP U_{i_r},
\end{array}
\]
and define $H_I^f:{\cal S}_I^f\rightarrow{\cal T}_I^f$ by
\[
H_I^f(p_1,q_1,\cdots,p_r,q_r;E):=H_I^{f,V+E}(p_1,q_1,\cdots,p_r,q_r).
\]

Since the map resulting from restricting $H_I^f$ to the fibre over $E$ is precisely the map $H_I^{f,V+E}$, if we prove that $H_I^f$ is transverse to ${\bf \De}_I$, then by Sard's theorem there exists a residual subset $R\subseteq\PP$ such that for all $E\in R$ the map $H_I^{f,V+E}$ is transverse to ${\bf \De}_I$. This is precisely the definition of $(f,V+E)$ being $I$-regular. Then, the only remaining thing to prove is

\begin{theo}\label{theo:genericity}
$H_I^f$ is transverse to ${\bf \De}_I$.
\end{theo}

\begin{proof} 
Let $(p_1,q_1,\ldots,p_r,q_r;E)$ be such that $H_I^f(p_1,q_1,\ldots,p_r,q_r;E)\in{\bf \De}_I$. This is
\[
s_{i_1}^{V+E}(p_1)=u_{i_1}^{V+E}(q_1)=:x_1,\ldots,s_{i_r}^{V+E}(p_r)=u_{i_r}^{V+E}(q_r)=:x_r,
\]
\[
\psi_{f(C_{i_1})f(C_{i_2})}^{V+E}(q_1)=p_2,\ldots,\psi_{f(C_{i_{r-1}})f(C_{i_r})}^{V+E}(q_{r-1})=p_r.
\]
We need to show that given
\[
\al_1,\be_1\in T_{x_1}C_{i_1},\ldots, \al_r,\be_r\in T_{x_r}C_{i_r},
\]
\[
y_2,z_2\in T_{p_2}Y_{f(C_{i_2})}^+,\ldots,y_r,z_r\in T_{p_r}Y_{f(C_{i_r})}^+,
\]
there exist
\[
\ga_1\in T_{x_1}C_{i_1},\ldots,\ga_r\in T_{x_r}C_{i_r},
\]
\[
\ze_2\in T_{p_2}Y_{f(C_{i_2})}^+,\ldots,\ze_r\in T_{p_r}Y_{f(C_{i_r})}^+,
\]
and
\[
(v_1,w_1,\ldots,v_r,w_r;\ep)\in T_{(p_1,q_1,\ldots,p_r,q_r;E)}{\cal S}_I^f
\]
such that $(\al_1,\be_1,\ldots,\al_r,\be_r,y_2,z_2,\ldots,y_r,z_r)$ equals
\[
(\ga_1,\ga_1,\ldots,\ga_r,\ga_r,\ze_2,\ze_2,\ldots,\ze_r,\ze_r)+d_{(p_1,q_1,\ldots,p_r,q_r;E)}H_I^f(v_1,w_1,\ldots,v_r,w_r;\ep).
\]
Note that $d_{(p_1,q_1,\ldots,p_r,q_r;E)}H_I^f(v_1,w_1,\ldots,v_r,w_r;\ep)$ equals
\[
\begin{array}{l}
(d_{(p_1,E)}s_{i_1}(v_1,\ep),d_{(q_1,E)}u_{i_1}(w_1,\ep),\ldots,d_{(p_r,E)}s_{i_r}(v_r,\ep),d_{(q_r,E)}u_{i_r}(w_r,\ep),\\
\hspace{1.4cm}d_{(q_1,E)}\psi_{f(C_{i_1})f(C_{i_2})}(w_1,\ep),v_2,\ldots,d_{(q_{r-1},E)}\psi_{f(C_{i_{r-1}})f(C_{i_r})}(w_{r-1},\ep),v_r).
\end{array}
\]
Putting it all together we see that our task is to solve two sets of equations:
\[
\left\{
\begin{array}{rl}
\al_j=&\ga_j+d_{(p_j,E)}s_{i_j}(v_j,\ep)\\
\be_j=&\ga_j+d_{(q_j,E)}u_{i_j}(w_j,\ep)
\end{array}
\right.
\]
for $j=1,\ldots,r$ and
\[
\left\{
\begin{array}{rl}
y_j=&\ze_j+d_{(q_{j-1},E)}\psi_{f(C_{i_{j-1}})f(C_{i_j})}(w_{j-1},\ep)\\
z_j=&\ze_j+v_j
\end{array}
\right.
\]
for $j=2,\ldots,r$.

We start solving the first set of equations. First of all choose $\ga_j=0$. Now observe that
\[
d_{(p_j,E)}s_{i_j}(v_j,\ep)=d_{p_j}s_{i_j}^{V+E}(v_j)+d_Es_{i_j}(p_j,\cdot)(\ep).
\]
If $j=2,\ldots,r$ we have that $p_j\in Y_{f(C_{i_j})}^+$ and then $s_{i_j}(p_j,\cdot)$ is constant because any $E\in\PP$ vanishes between $Y_{f(C_{i_j})}^+$ and $C_{i_j}$ (recall Figure \ref{fig:auxlevelspertzone}). Therefore its differential vanishes and we just need $v_j$ such that $d_{p_j}s_{i_j}^{V+E}(v_j)=\al_j$, which exists because $s_{i_j}^{V+E}$ is a submersion. With this we determine $v_2,\ldots,v_r$. With an analogous argument we determine $w_1,\ldots,w_{r-1}$. However, we still need to find $v_1$ and $w_r$.

Now we solve the second set of equations, which involves the already chosen $v_2,\ldots,v_r$ and $w_1,\ldots,w_{r-1}$ but not $v_1,w_r$: we take $\ze_j=z_j-v_j$ and then we only need to solve
\[
\begin{array}{rcl}
y_j-z_j+v_j & = & d_{(q_{j-1},E)}\psi_{f(C_{i_{j-1}})f(C_{i_j})}(w_{j-1},\ep)\\
            & = & d_{q_{j-1}}\psi_{f(C_{i_{j-1}})f(C_{i_j})}^{V+E}(w_{j-1})+d_E\psi_{f(C_{i_{j-1}})f(C_{i_j})}(q_{j-1},\cdot)(\ep).
\end{array}
\]
So we need to find $\ep$ solving simultaneously the equations
\[
d_E\psi_{f(C_{i_{j-1}})f(C_{i_j})}(q_{j-1},\cdot)(\ep)=y_j-z_j+v_j-d_{q_{j-1}}\psi_{f(C_{i_{j-1}})f(C_{i_j})}^{V+E}(w_{j-1})
\]
for $j=2,\ldots,r$. The right hand side, which is already determined, will be denoted by $\eta_j$ to simplify notation.

We shall find an $\ep_j$ solving each equation separately and then properly combine them into a single $\ep$: note that we can write $E=(E_2,\ldots,E_n)$ with $E_i\in L_{a_i}^0$ in a unique way. If we are able to find $\ep_j\in T_{E_{i_j}}P_{i_j}$ solving $d_{E_{i_j}}\psi_{f(C_{i_{j-1}})f(C_{i_j})}(q_{j-1},\cdot)(\ep_j)=\eta_j$, then 
\[
\ep:=\sum_{j=2}^r(0,\ldots,\stackrel{i_j-1}{\wideparen{\ep_j}},\ldots,0)
\]
is a solution for all the equations above.

Let $t_j\geq 0$ be the unique value such that $\xi^{V+E}_{t_j}(q_{j-1})\in Y_{f(C_{i_j-1})}^-$. Then 
\[
\psi_{i_{j-1},i_j}(q_{j-1},\cdot)=\psi_{i_j-1,i_j}(\xi^{V+E}_{t^0_j}(q_{j-1}),\cdot),
\]
which in the notation of Lemma \ref{lemm:psisurj} equals $\Psi_{\xi^{V+E}_{t^0_j}(q_{j-1})}$. According to the same lemma, we get that $d_{E_{i_j}}\psi_{f(C_{i_{j-1}})f(C_{i_j})}(q_{j-1},\cdot)$ is surjective, so there indeed exists and $\ep_j$ such that $d_{E_{i_j}}\psi_{f(C_{i_{j-1}})f(C_{i_j})}(q_{j-1},\cdot)(\ep_j)=\eta_j$.

Recall that to conclude the proof we still need to find suitable $v_1$ and $w_r$. The equation to solve for $v_1$ is
\[
\al_1=d_{p_1}s_{i_1}^{V+E}(v_1)+d_Es_{i_1}(p_1,\cdot)(\ep),
\]
which can be rearranged to
\[
d_{p_1}s_{i_1}^{V+E}(v_1)=d_Es_{i_1}(p_1,\cdot)(\ep)-\al_1
\]
with the right-hand side already determined. Since $s_{i_1}^{V+E}$ is a submersion there certainly exists $v_1$ solving it. The argument to find $w_r$ is analogous.\end{proof}

\section{Hamiltonian circle actions}\label{sec:hamiltonian}

Let $(M,\om,S^1,\mu)$ be a Hamiltonian $S^1$-space. The following properties are well-known and widely found in the literature:

\begin{prop}\label{prop:muismb}The following assertions hold:
\begin{enumerate}[1)]
\item The moment map $\mu$ is a Morse-Bott function with $\mbox{Crit}(\mu)=M^{S^1}$. 
\item The connected components of $M^{S^1}$ are symplectic submanifolds of $M$.
\item Each critical component of $\mu$ has even index.
\item Each level, regular or critical, of $\mu$ is connected.
\end{enumerate}  
\end{prop}

\begin{proof} E.g. \cite[2.2,\ 2.3]{Ati}, \cite[IV.1-IV.3]{Aud}, \cite[Hmwk. 21]{Can}, \cite[5.47,\ 5.51]{McSa1}.\end{proof}

Since $\mu$ is a Morse-Bott function we can study its moduli space of broken gradient lines with respect to a suitable gradient-like vector field $V$. The moment map is invariant and it is only natural to want all the constructions to be invariant under the action. In particular we would like $V$ to be invariant as well. However, this has a serious drawback: it may be impossible to find an invariant gradient-like vector field $V$ such that $(\mu,V)$ is a regular Morse-Bott pair if the action on $M$ has finite non-trivial stabilisers. The following example illustrates this situation:

\begin{exam}\label{exam:counterexample} Consider the symplectic manifold 
\[
M=\{((z_0:z_1:z_2),(w_0:w_1))\in\CC P^2\times\CC P^1:\ z_0w_0+z_1w_1=0\},
\] 
which can be identified with the blow up of $\CC P^2$ at a point. Endow $M$ with the Hamiltonian $S^1$-action given by 
\[
\te((z_0:z_1:z_2),(w_0:w_1))=((\te z_0:\te^{-1}z_1:z_2),(w_0:\te^2w_1)).
\] 
The moment map is 
\[
\mu:((z_0:z_1:z_2),(w_0:w_1))\longmapsto\frac{|z_0|^2-|z_1|^2}{|z_0|^2+|z_1|^2+|z_2|^2}+\frac{2|w_1|^2}{|w_0|^2+|w_1|^2}.
\] 
Recall that the critical points of $\mu$ coincide with the fixed point set $M^{S^1}$. In this example it consists of four isolated points. 

Let $V$ be an invariant gradient-like vector field for $\mu$. The following table contains some information about the critical set of $\mu$ and the (un)stable manifolds with respect to $V$:
\[
\begin{array}{|c|c|c|c|}
\hline
\mbox{Critical component}\ C_i & \mu(C_i) &\dim S^{V}_i & \dim U^{V}_i\\
\hhline{|=|=|=|=|}
C_1=\{([1:0:0],[0:1])\} & 3 & 0 & 4\\
\hline
C_2=\{([0:0:1],[0:1])\} & 2 & 2 & 2\\
\hline
C_3=\{([0:0:1],[1:0])\} & 0& 2 & 2\\
\hline
C_4=\{([0:1:0],[1:0])\} & -1 & 4 & 0\\
\hline
\end{array}
\]

The $S^1$-action restricts to the sphere $B=\{(0:0:1)\}\times\CC P^1\subseteq M$, which contains $C_2,C_3$ but not $C_1,C_4$. The stabiliser for any non-fixed point of $B$ is $\{-1,1\}\simeq\ZZ/(2)$. Let $p\in B$ be a non-fixed point and let
\[
\begin{array}{rccc}
\al_{-1}:&M&\longrightarrow& M\\
 &((z_0:z_1:z_2),(w_0:w_1))&\longmapsto &((-z_0:-z_1:z_2),(w_0:w_1))
\end{array}
\] 
be the map defined by the action of $-1\in S^1$ on $M$. Writing $T_pM\simeq T_pB\oplus N_pB$ we get $d_p\al_{-1}=(id,-id)$. Therefore any $S^1$-invariant vector field -in particular $V$- must be tangent to $B$. Hence there exist integral curves of $V$ connecting $C_2$ and $C_3$ or, equivalently, $U^V_2\cap S^V_3\neq\emptyset$.

Now let $Y_2^-=\mu^{-1}(3/2),Y_0^+=\mu^{-1}(1/2)$ be auxiliary regular levels. Since there are no critical levels in between we have that $Y_{2,0}^V=Y_2^-$. In order to have regularity we need 
\[
H^{\mu,V}_{\{2,3\}}:S_2^V\times(U^V_2\cap Y_2^-)\times(S_3^V\cap Y_0^+)\times U_3^V\rightarrow(C_2\times C_2)\times (C_3\times C_3)\times(Y_0^+\times Y_0^+)
\]
to be transverse to ${\bf \De}_{\{2,3\}}$. Since $C_2,C_3$ are $0$-dimensional this is equivalent to the fact that $S_3^V\cap Y_0^+$ and $\psi^V_{2,0}(U^V_2\cap Y_2^-)$ intersect transversally inside $Y_0^+$. The dimension of the former two is $1$ and the dimension of the latter is $3$, so transversality happens only if the intersection is empty, but this is impossible because $U^V_2\cap S^V_3\neq\emptyset$.
\end{exam}

\subsection{Semi-free actions}\label{subsec:semifree}

The last example shows that invariant regular Morse-Bott pairs may not exist if the action of $S^1$ on $M$ has finite non-trivial stabilisers. In this section we prove that this is the only possible obstruction: if the action is semi-free it is always possible to achieve both invariance and regularity. To show this we will use similar techniques to the ones in Section \ref{subsec:perturb}, taking care that everything remains invariant --or equivariant-- with respect to the action.

Let us start by stating clearly the set-up and the notation to be used: let $(M,\om)$ be a compact connected symplectic manifold of dimension $2m$, endowed with an effective and {\bf semi-free} Hamiltonian $S^1$-action. Denote by $\mu:M\rightarrow\RR$ the moment map. Fix an invariant almost complex structure $J$ compatible with $\om$ and let $\rho_J:=\om(\cdot,J(\cdot))$ be the corresponding invariant Riemannian metric on $M$. Let 
\[
X_p=\frac{d}{dt}_{|t=0}e^{2\pi it}p
\]
be the vector field generated by the infinitesimal action. Then $d\mu=\io_X\om$ and the gradient vector field of $\mu$ with respect to $\rho_J$ is $JX$. In particular $JX$ is an invariant gradient-like vector field for $\mu$. We denote by $\xi^{JX}$ the flow of $-JX$ and by $\vp^{JX}_p$ the integral line of $-JX$ starting at $p$, i.e. $\vp^{JX}_p(t)=\xi^{JX}_t(p)$.

We want to prove that $(\mu,\widetilde{JX})$ is regular for a generic invariant perturbation $\widetilde{JX}$ of $JX$. The perturbation $\widetilde{JX}$ will be of the form $EX$, where $E\in C^\infty(\mbox{End}\ TM)$ is a perturbation of $J$ as an invariant almost complex structure compatible with $\om$. To understand what properties such an $E$ needs to satisfy we make a brief digression on symplectic linear algebra: let $(W,\sig)$ be a symplectic vector space and let $I\in\mbox{End}(W)$ be a complex structure on $W$ compatible with $\sig$, meaning that
\begin{enumerate}
\item $I^2=-id_W$
\item $\sig(\cdot,I(\cdot))$ is a positive inner product on $V$.
\end{enumerate}
Linearising these two conditions, i.e. applying them to $I+se$ with $s\in\RR$, $e\in\mbox{End}(W)$ and making the linear terms on $s$ vanish, we get
\begin{enumerate}
\item $eI+Ie=0$,
\item $e+e^{*\sig}=0$.
\end{enumerate}
Here, $e^{*\sig}$ denotes the $\sig$-dual of $e$, which is defined by $\sig(e^{*\sig}(\cdot),\cdot)=\sig(\cdot,e(\cdot))$. Therefore the vector space 
\[
D=\{e\in\mbox{End}(W):\ eI+Ie=0,\ e+e^{*\sig}=0\}
\] 
can be viewed as the space of infinitesimal deformations of $I$ as a complex structure compatible with $\sig$. Indeed, if $e\in D$ and $s$ is sufficiently small, $I\exp(seI)$ is a complex structure compatible with $\sig$. Going back to the manifold context, this discussion implies that the sections of the vector bundle 
\[
{\cal D}=\{e\in\mbox{End}\ TM:\ eJ+Je=0,\ e+e^{*\om}=0\}
\] 
should be regarded as the perturbations of $J$ as an almost complex structure compatible with $\om$. Then the perturbations of $J$ we will consider are equivariant sections of ${\cal D}$.

Recall from Section \ref{subsec:perturb} that we need two things in order to apply the results proved there to the present context:
\begin{enumerate}
\item A perturbation zone where the perturbations have to be supported in order that the property of being a gradient-like vector field is preserved (as in Proposition \ref{prop:perturbation}).
\item A surjectivity result like Lemma \ref{lemm:surjev} in order to get submersions as in Proposition \ref{prop:phisurj} and Lemma \ref{lemm:psisurj}.
\end{enumerate}

As in former sections let
\begin{enumerate}[-]
\item $C_1,\ldots,C_n$ be the connected components of the critical set of $\mu$ labelled so that $\mu(C_1)\geq\cdots\geq\mu(C_n)$
\item $a_2>\ldots>a_\nu$ be the points in $\mu(F)$ and $Z_{a_2},\ldots,Z_{a_\nu}$ be regular levels such that $a_{i-1}>\mu(Z_{a_i})>a_i$ and take $\be>0$ such that \[a_{i-1}>\mu(Z_{a_i})+\be>\mu(Z_{a_i})-\be>a_i\] for every $i=2,\ldots,\nu$. Finally let $Z'_{a_i}:=\mu^{-1}((\mu(Z_{a_i})-\be,\mu(Z_{a_i})+\be))$ and define the perturbation zone by $Z'=\bigcup_{i=2}^\nu Z_{a_i}'$.
\end{enumerate}

The surjectivity result is

\begin{lemm}\label{lemm:ontoness} There exists a finite dimensional subspace $A\subseteq C^\infty({\cal D})$ such that for every $p\in M$ and every in $v\in T_pM\setminus\{0\}$ the evaluation map
\[
\begin{array}{rccc}
ev_{p,v}: & A& \longrightarrow & T_pM\\
 & E & \longmapsto & E_p(v)
\end{array}
\]
is surjective.
\end{lemm}

\begin{proof}
Using compactness and partitions of unity as in Lemma \ref{lemm:surjev} the result follows once we have proved that there is a finite dimensional subspace $A_p\subseteq {\cal D}_p$ such that
\[
\begin{array}{rccc}
A_p & \longrightarrow & T_pM\\
e & \longmapsto & e(v)
\end{array}
\]
is surjective. Let us prove this fact:

Fix a basis $\{u_1,\ldots,u_m,v_1,\ldots,v_m\}$ of $T_pM$ such that $J_p$ is represented by the matrix 
\[
I_0=
\begin{pmatrix} 
  0 & Id\\ 
-Id & 0
\end{pmatrix}
\] 
and such that $\om_p$ is represented by $I_0^T=-I_0$. That such a basis exists is a well known result (see e.g. \cite[Hmwk. 8]{Can}). In this basis the induced inner product $\om_p(\cdot,J_p(\cdot))$ is represented by the identity matrix. 

Consider now the endomorphisms of $T_pM$ that in the chosen basis are represented by the matrices 
\[
P_{ij}=u_ju_i^T+u_iu_j^T-v_jv_i^T-v_iv_j^T,\ Q_{ij}=u_iv_j^T+u_jv_i^T+v_iu_j^T+v_ju_i^T.
\] 
Both endomorphisms belong\footnote{In the chosen basis, a matrix $E_0$ represents an element of ${\cal D}_p$ if and only if $E_0I_0+I_0E_0=0$ and $E_0^TI_0+I_0E_0=0$, which is equivalent to $E_0=E_0^T$ and $E_0=I_0E_0I_0$.} to ${\cal D}_p$. Let 
\[
v=\lm_1u_1+\cdots+\lm_mu_m+\mu_1v_1+\cdots+\mu_mv_m.
\] 
If $\lm_i^2+\mu_i^2\neq 0$ we have that 
\[
\frac{1}{2}\frac{\lm_iP_{ii}+\mu_iQ_{ii}}{\lm_i^2+\mu_i^2}(v)=u_i,\ \frac{1}{2}\frac{\lm_iQ_{ii}-\mu_iP_{ii}}{\lm_i^2+\mu_i^2}(v)=v_i.
\] 
Otherwise, i.e. if $v$ is orthogonal to both $u_i$ and $v_i$, take $j$ such that $\lm_j^2+\mu_j^2\neq 0$, which exists because $v$ is not zero. Then 
\[
\frac{\lm_jP_{ij}+\mu_jQ_{ij}}{\lm_j^2+\mu_j^2}(v)=u_i,\ \frac{\lm_jQ_{ij}-\mu_jP_{ij}}{\lm_j^2+\mu_j^2}(v)=v_i.
\] 
This shows that the vector space $A_p$ generated by the matrices $P_{ij},Q_{ij}$ satisfies the required conditions.
\end{proof} 

Following the steps of Section \ref{subsec:perturb}, the next thing we want to do is to construct suitable finite dimensional vector spaces 
\[
L_{a_i}\subseteq\{E\in C^\infty({\cal D}):\ \mbox{supp}\ E\subseteq Z_{a_i}'\}.
\]

In particular we want the elements of $L_{a_i}$ to be $S^1$-invariant, so we make the following construction: given $z\in Z_{a_i}$ and $\de>0$, define the set
\[
S(z,\de)=\{\exp_z^{\rho_J}(v):\ v\in T_zM,\ |v|\leq\de,\ \rho_J(v,X_z)=0\}.
\] 
Let also $R(z,\de)=S^1\cdot S(z,\de)$. If $\de$ is smaller than the injectivity radius of the metric $\rho_J$ (which we assume) then $S(z,\de)$ is a slice of the $S^1$-action, meaning that at all $p\in S(z,\de)$ the orbit through $p$ and $S(z,\de)$ intersect transversally. Also if $\de$ is small enough, $R(z,\de)$ is a solid $2m$-dimensional torus contained in $Z_{a_i}'$. 
\begin{figure}[H]
	\centering
		\includegraphics{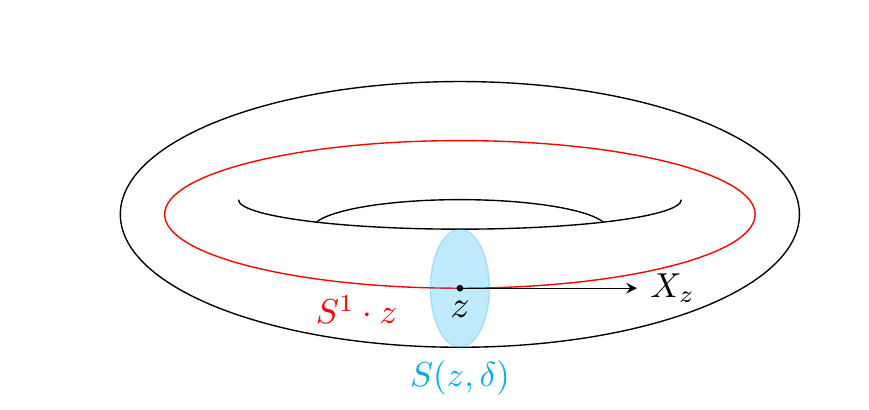}
	\caption{The solid torus $R(z,\de)$.}
	\label{fig:solidtorus}
\end{figure}

Fix $\de_z>0$ small enough so that $R(z,\de_z)$ is a solid torus contained in $Z_{a_i}'$ and so that there exist sections $E^1,\ldots,E^s\in C^{\infty}(S(z,\de_z),{\cal D})$ satisfying $T_xM=\la E^1_x(u),\ldots,E^s_x(u)\ra$ for all $x\in S(z,\de_z)$ and $u\in T_xM\setminus\{0\}$. This can be done thanks to the surjectivity Lemma \ref{lemm:ontoness}. Now we want to extend these sections to equivariant sections in $C^\infty(M,{\cal D})$ supported on $R(z,\de_z)$:

Let $\eta_z:S(z,\de_z)\rightarrow[0,1]$ be a smooth function supported on the interior of $S(z,\de_z)$ and such that $\eta_z(z)=1$. Then the sections $\eta_zE^j$ are supported on $S(z,\de_z)$. Since the action on $R(z,\de_z)$ is free\footnote{Here is where we use that the the circle action on $M$ is semi-free.}, they extend uniquely to equivariant sections in $C^\infty(R(z,\de_z),{\cal D})$. Finally we extend them by zeroes to all of $M$. We denote by $W_z\subseteq C^\infty(M,{\cal D})$ the vector space generated by the sections $\eta_zE^j$.

The union of the sets $\stackrel{\circ}{R}(z,\de_z)\cap Z_{a_i}$ as $z$ runs over all points of $Z_{a_i}$ is an open covering of $Z_{a_i}$. Since $Z_{a_i}$ is compact there exists a finite number of points $\{z_k\}\subseteq Z_{a_i}$ such that $Z_{a_i}$ is contained in the union of the interiors of the sets $R(z_k,\de_{z_k})$. We define
\[
L_{a_i}:=W_{z_1}\oplus\cdots\oplus W_{z_k}.
\]
The elements of $L_{a_i}$ are $S^1$-equivariant and supported on $Z_{a_i}'$. Now take $L_{a_i}^0\subseteq L_{a_i}$ a neighbourhood of $0$ such that for all $E\in L_{a_i}^0$ it is satisfied that 
\[
\sup_{Z_{a_i}'}|d\mu(EX)|<\inf_{Z_{a_i}'}d\mu(JX)=\inf_{Z_{a_i}'}\|JX\|^2.
\]
The inequality is a condition analogous to that of Proposition \ref{prop:perturbation}, while the equality is given by the fact that $JX$ is the gradient vector field of $\mu$ with respect to $\rho_J$.

If $\PP=L_{a_2}^0\oplus\cdots\oplus L_{a_\nu}^0$ now we can prove everything as we did in Section \ref{subsec:perturb} to conclude that

\begin{theo}\label{theo:generictranssemifree}
There exists a regular subset $R\subseteq\PP$ such that for all $E\in R$ the Morse-Bott pair $(\mu,(J+E)X)$ is regular.
\end{theo}

%% file: Multivalued.tex
\chapter{Multivalued perturbations for Hamiltonian circle actions}\label{ch:multivalued}


Let $(M,\om,S^1,\mu)$ be a Hamiltonian space. If $X$ is the vector field generated by the action and $J$ is an invariant almost complex structure we have that $(\mu,JX)$ is a Morse-Bott pair. The main result in Chapter \ref{ch:standard} was to prove that if the action is semi-free we can perturb $J$ in such a way that the Morse-Bott pair is regular. However, Example \ref{exam:counterexample} shows that if the action is not semi-free --it contains finite non-trivial stabilisers-- it is in general not possible to achieve both regularity and invariance simultaneously. The technique developed in this chapter, multivalued perturbations, solves this problem.

The order of exposition in this chapter is somehow reversed from that of Chapter \ref{ch:standard}. We first review our construction of perturbations in the semi-free case and explain what modifications are required for the general case, leading to what we call multivalued perturbations. Once this is done we will study again the moduli space of broken gradient lines and see what changes are needed in order to multivalued perturbations to fit in. Chapter \ref{ch:standard} is in the end a particular case of what it is studied here and will be a good guideline to understand what steps we are following. Many of the results of the present chapter have a counterpart in Chapter \ref{ch:standard} and for this reason the references to that chapter are frequent.

\section{Understanding the problem}\label{sec:understanding}

A significant part of the notation to be used in the present chapter is borrowed from Chapter \ref{ch:standard}. Let us make a reminder of this notation to have it collected in a single place for a better reference:

\newpage
Our setup is the following (see the introduction of Section \ref{sec:hamiltonian}):
\begin{itemize}
\item $(M^{2m},\om,S^1,\mu)$ is a Hamiltonian space with $M$ compact and connected.
\item $J$ is an invariant almost complex structure compatible with the symplectic form $\om$ and $\rho_J=\om(\cdot,J(\cdot))$ is the corresponding Riemmanian metric.
\item $X$ is the vector field generated by the infinitesimal action. 
\end{itemize}
From this, several facts are derived and some more notation is introduced (see Proposition \ref{prop:muismb}):
\begin{itemize}
\item $\mu$ is a Morse-Bott function and $JX$ is a gradient-like vector field for $\mu$. We denote by $\xi^{JX}_t$ the flow of $-JX$ at time $t\in\RR$ and by $\vp^{JX}_p$ the integral line of $-JX$ starting at $p\in M$.
\item The critical set $F$ of $\mu$ coincides with the set of fixed points $M^{S^1}$. We denote its connected components $C_1,\ldots,C_n$ so that $\mu(C_1)\geq\ldots\geq\mu(C_n)$.
\end{itemize}
We also set a perturbation zone (see the beginning of Section \ref{subsec:perturb}):
\begin{itemize}
\item $a_1>\cdots>a_\nu$ are the critical values of $\mu$, i.e. $\mu(F)=\{a_1,\ldots,a_\nu\}$.
\item $Z_{a_2},\ldots,Z_{a_\nu}$ are regular levels of $\mu$ and $\be>0$ is a positive number such that $a_{i-1}>\mu(Z_{a_i})+\be>\mu(Z_{a_i})-\be>a_i$ for every $i=2,\ldots,\nu$. 
\item We define $Z_{a_i}':=\mu^{-1}((\mu(Z_{a_i})-\be,\mu(Z_{a_i})+\be))$ and $Z'=\bigcup_{i=2}^\nu Z_{a_i}'$.
\item To simplify some notation, in the present chapter, we set $Z=\bigcup_{i=2}^\nu Z_{a_i}$.
\end{itemize}
With all these elements we define certain solid tori of $M$ where perturbations of $J$ are to be defined (see Figure \ref{fig:solidtorus} and the text afterwards):
\begin{itemize}
\item For $z\in Z_{a_i}$ and $\de>0$, let
\[
S(z,\de)=\{\exp_z^{\rho_J}(v):\ v\in T_zM,\ |v|\leq\de,\ \rho_J(v,X_z)=0\}.
\] 
If $\de$ is small enough, $S(z,\de)$ is a slice of the action and $R(z,\de)=S^1\cdot S(z,\de)$ is a solid torus contained in $Z_{a_i}'$. 
\item The vector bundle ${\cal D}=\{e\in\mbox{End}\ TM:\ eJ+Je=0,\ e+e^{*\om}=0\}$ parametrises perturbations of $J$ and, from Lemma \ref{lemm:ontoness} it follows that if $\de$ is small enough there exist sections $E^1,\ldots,E^s\in C^{\infty}(S(z,\de),{\cal D})$ such that $T_xM=\la E^1_x(u),\ldots,E^s_x(u)\ra$ for all $x\in S(z,\de)$ and $u\in T_xM\setminus\{0\}$.
\item We fix $\de=\de_z$ such that all these conditions are satisfied.
\end{itemize}

At this point, in Chapter \ref{ch:standard}, we extended the sections $E^j$ to equivariant sections of $C^\infty(M,{\cal D})$ supported on $R(z,\de_z)$. To do so we used that the action is free on $R(z,\de_z)$ which was a consequence of $S^1$ acting semi-freely on $M$. In the present chapter the semi-freeness condition is dropped, so we need to do something else. What we will do now is to define a covering $R^\#(z,\de_z)$ of the torus $R(z,\de_z)$ which carries a free action of $S^1$. It is on this covering that we will extend the sections. Let
\[
R^\#(z,\de_z):=\{(\te,p)\in S^1\times R(z,\de_z):\ \te\cdot p\in S(z,\de_z)\}.
\]
Note that the set $\{\te\in S^1:\ \te\cdot S(z,\de_z)=S(z,\de_z)\}$ coincides with the stabiliser of $z$. Let $o_z$ be the number of elements of the stabiliser of $z$. Then the projection to the second factor
\[
\pi:R^\#(z,\de_z)\longrightarrow R(z,\de_z)
\]
is an $o_z$ to $1$ map. Moreover, the map $\pi^{-1}(S^1\cdot z)\rightarrow S^1\cdot z$ can be identified with the map $S^1\rightarrow S^1$ that sends $\te$ to $\te^{o_z}$. From this we deduce that $\pi^{-1}(S^1\cdot z)$ is connected and hence so is $R^\#(z,\de_z)$. Therefore, $R^\#(z,\de_z)$ is also a solid torus. Consider the action of $S^1$ on $R^\#(z,\de_z)$ defined by $\al\cdot(\te,p)=(\al\te,\al\cdot p)$ for $\al\in S^1$. This action is free, and with respect to it, $\pi$ is equivariant.

Let $\eta_z:S(z,\de_z)\rightarrow[0,1]$ be a smooth function supported on the interior of $S(z,\de_z)$ and such that $\eta_z(z)=1$. Then the sections $\eta_zE^j$ are supported on $S(z,\de_z)$. Consider $(\eta_zE^j)^\#\in C^\infty(\{1\}\times S(z,\de_z),\pi^*{\cal D})$ be their pull-backs under $\pi$. Since the $S^1$ action on $R^\#(z,\de_z)$ is free, they extend uniquely to equivariant sections in $C^\infty(R^\#(z,\de_z),\pi^*{\cal D})$. Finally $W_z\subseteq C^\infty(R^\#(z,\de_z),\pi^*{\cal D})$ will denote the vector space generated by the sections $(\eta_zE^j)^\#$.

Now it is a good moment to stop and compare the present situation with what we did in Chapter \ref{ch:standard} for semi-free actions. In Section \ref{subsec:semifree}, the gradient-like vector field $JX$ was perturbed by adding a vector field of the form $EX$ where $E$ is an invariant perturbation of $J$ supported on the perturbation zone $Z'$ and such that $\sup_Z'|d\mu(EX)|<\inf_Z'\|JX\|^2$. In this chapter we change our point of view a little bit and instead of focusing on which are the vector fields we use to perturb $JX$ we focus on their integral lines. The following section is devoted to such objects.


\section{Perturbed gradient segments}\label{sec:gradseg}

We define here the notion of $\ep$-perturbed gradient segment, which is basically a portion of an integral curve of a vector field resulting from adding an arbitrary small perturbation --the parameter $\ep>0$ determines how small-- to $-JX$. Then some properties of such segments are given in a series of lemmas leading to the final result, which states that if $\ep$ and $\de$ are small enough the intersection of an $\ep$-perturbed gradient segment with the torus $R(z,\de)$ behaves in a controlled manner (in particular it is connected).

\begin{defi}\label{defi:epgd} Let $\ep>0$ and let $A\subseteq\RR$ be an interval. A smooth map $\ga:A\rightarrow M$ is an $\ep$-{\it perturbed gradient segment} if there exists a vector field $Y$ supported on $Z'$ such that 
\begin{enumerate}
\item $\ga'(t)=-(JX+Y)_{\ga(t)}$.
\item $\sup_{Z'}|d\mu(Y)|<\ep$.
\item $\sup_{Z'}\|Y\|^2<\ep$.
\end{enumerate}
\end{defi}

We shall define two quantities that are going to be used along this section, namely
\[
s:=\inf_{Z'}\|JX\|^2\ \hspace{1cm}\ S:=\sup_{Z'}\|JX\|^2.
\]

The first of the lemmas shows that if we take $\ep$ small enough, an $\ep$-perturbed gradient segment is monotonically decreasing along $\mu$, implying in particular that it is an embedding and its image is an embedded simply connected curve. The following picture shows the wrong and the right behaviour:

\begin{figure}[H]
	\centering
		\includegraphics{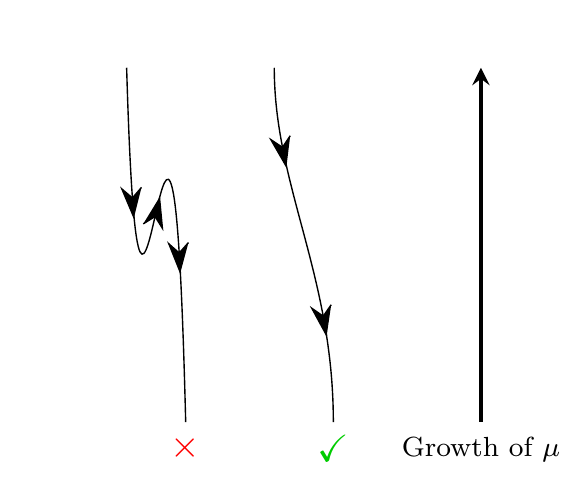}
	\caption{$\ep$-perturbed gradient segments with respect to the growth of $\mu$.}
	\label{fig:growth}
\end{figure}

\begin{lemm}\label{lemm:monodecreasing} Let $\ep<s$. If $\ga:A\rightarrow M$ is an $\ep$-perturbed gradient segment, then $(\mu\circ\ga)'<0$.
\end{lemm}

\begin{proof}
Since $JX$ is the gradient vector field of $\mu$ with respect to $\rho_J$ we get
\[
\begin{array}{rcl}
(\mu\circ\ga)'(t) & = & d_{\ga(t)}\mu(\ga'(t))\\
                  & = & d_{\ga(t)}\mu(-(JX+Y)_{\ga(t)})\\
                  & = &-d_{\ga(t)}((JX)_{\ga(t)})-d_{\ga(t)}\mu(Y_{\ga(t)})\\
                  & = &-\|(JX)_{\ga(t)}\|^2-d_{\ga(t)}\mu(Y_{\ga(t)})\\
                  & \leq & -\|(JX)_{\ga(t)}\|^2+|d_{\ga(t)}\mu(Y_{\ga(t)})|\\
                  & < & -s+\ep\\
                  & < & 0
\end{array}
\]
\end{proof}

Suppose now that an $\ep$-perturbed gradient segment crosses the perturbation zone $Z'$. We can then control the length of the interval in which it is defined provided that $\ep$ is small enough: 

\begin{lemm}\label{lemm:inequalities} Let $\ep<s$. If $\ga:A\rightarrow M$ is an $\ep$-perturbed gradient segment such that $B:=\ga^{-1}(Z_{a_i}')$ is non-empty, then $B$ is connected and 
\[
\frac{E(\ga(B))}{S+3\ep}<|B|<\frac{\sup\mu(\ga(B))-\inf\mu(\ga(B))}{s-\ep},
\] 
where $E(\ga(B))$ is the energy of $\ga(B)$.
\end{lemm}

\begin{proof} We show that $B$ is connected by contradiction: if $t_0<t_1<t_2$ were elements of $A$ such that $t_0,t_2\in B$ but such that $\ga(t_1)\not\in Z_{a_i}'$, then either 
\[
\mu(\ga(t_1))<\inf_{Z_{a_i}'}\mu\leq\mu(\ga(t_2))\ \hspace{0.5cm} \mbox{or} \hspace{0.5cm} \mu(\ga(t_1))>\sup_{Z_{a_i}'}\mu\geq\mu(\ga(t_0)).
\] 
Both cases contradict the fact that $\mu\circ\ga$ is monotonically decreasing stated in Lemma \ref{lemm:monodecreasing}. The following picture illustrates the situations we just described:

\begin{figure}[H]
	\centering
		\includegraphics{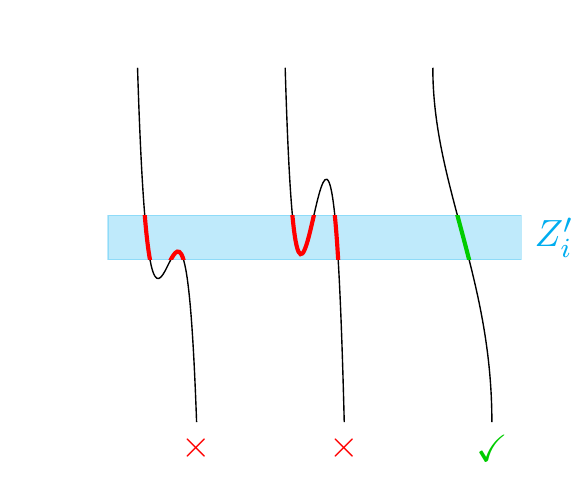}
	\caption{The intersection of an $\ep$-perturbed gradient segment with $Z_{a_i}'$.}
	\label{fig:connected}
\end{figure}

Also in Lemma \ref{lemm:monodecreasing} we proved that $(\mu\circ\ga)'(t)<\ep-s$. Then 
\[
\int_B s-\ep\, dt<\int_B-(\mu\circ\ga)'(t)\,dt,
\] 
which implies 
\[
|B|(s-\ep)<\mu(\ga(\inf B))-\mu(\ga(\sup B))=\sup(\mu(\ga(B))-\inf(\mu(\ga(B)),
\] 
where the last equality follows from the fact that $\ga$ is monotonically decreasing along $\mu$. This already proves one of the two inequalities.

The other inequality is also obtained from integrating along $B$, in this case the estimate $\|\ga'(t)\|^2<S+3\ep$, which is obtained as follows:
\[
\begin{array}{rcl}
\|\ga'(t)\|^2 & = & \|-(JX+Y)_{\ga(t)}\|^2\\
              & = & \rho_J((JX+Y)_{\ga(t)},(JX+Y)_{\ga(t)})\\
              & = & \|(JX)_{\ga(t)}\|^2+\|Y_{\ga(t)}\|^2+2(\rho_J)_{\ga(t)}((JX)_{\ga(t)},Y_{\ga(t)})\\
              & = & \|(JX)_{\ga(t)}\|^2+\|Y_{\ga(t)}\|^2+2d_{\ga(t)}\mu(Y_{\ga(t)})\\
              &\leq& \|(JX)_{\ga(t)}\|^2+\|Y_{\ga(t)}\|^2+2|d_{\ga(t)}\mu(Y_{\ga(t)})|\\
              & < & S+\ep+2\ep\\
			  & = & S+3\ep
\end{array}
\]
\end{proof}

Another property of $\ep$-perturbed gradient segments is that if they are close to the orbit $S^1\cdot z$, then the square of the distance function is convex. This is stated with precision in the next lemma:

\begin{lemm}\label{lemm:sqconvex} Let $z\in Z$. Define $f:A\rightarrow\RR$ by
\[
f(t):=\inf\{\de^2:\ \ga(t)\in R(z,\de)\}.
\]
There exist $\ep,\de_0$ such that if $\ga:A\rightarrow M$ is an $\ep$-perturbed and $g(t)\leq\de_0^2$, then $g$ is strictly convex.
\end{lemm}

\begin{proof} We take local coordinates around $z$ adapted to the action as follows:

Consider $L\subseteq T_zM$ the span of the vectors $X_z$ and $JX_z$, which are tangent to the orbit $S^1\cdot z$ and the curve $\vp_z^{JX}$ respectively. Let also $U\subseteq L^\perp$ be a small neighbourhood of $0$. For small $\eta>0$ consider $V=(-\eta,\eta)\times(-\eta,\eta)\times U$. 
\begin{figure}[H]
	\centering
		\includegraphics{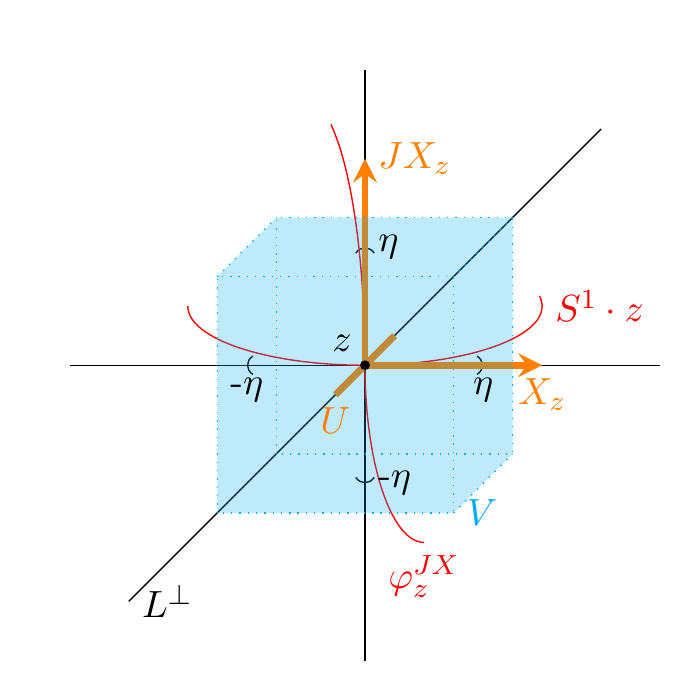}
	\caption{Local coordinates for a Hamiltonian $S^1$-action.}
	\label{fig:loccoords}
\end{figure}

Then $\io:V\rightarrow M$ defined by $\io(t,\te,u)=\xi_t^{JX}(e^{i\te}\cdot\exp^{\rho_J}_p u)$ is an embedding and $\io(V)$ is a neighbourhood of $p$ in $M$. If we take the metric $dt^2+d\te^ 2+du^2$ on $V$, where $du^2$ is the restriction of $(\rho_J)_p$ to $U$, we have that the integral curves of $d\io^{-1}(JX)$ on $V$ are of the form $\ga(t)=(t,\te_0,u_0)$ for some constant $(\te_0,u_0)\in (-\eta,\eta)\times U$.

In these coordinates $O=\{0\}\times(-\eta,\eta)\times\{0\}\subseteq V$ coincides with $\io^{-1}(S^ 1\cdot z)$. If $d_0$ is the distance on $V$ induced by the chosen metric we have that
\[
g_0(t):=d_0(\al(t),O)^2=t^2+\|u_0\|^2
\] 
and then $g_0''(t)=2$. Now let $d$ be the distance on $V$ induced via $\io$ by the distance $d_{\rho_J}$ on $M$. If $\eta$ and $\ep$ are small enough the function $g_\ga(t)=d((\io^{-1}\circ\ga)(t),O)^2$ satisfies $g_\ga''(t)\geq 1$. 

From the generalised Gauss lemma for submanifolds \cite[lemma 2.11]{Gray} it follows that, for $p\in M$, if $d_{\rho_J}(p,S^1\cdot z)$ is small enough then
\[
d_{\rho_J}(p,S^1\cdot z)=\inf\{\de:\ p\in R(z,\de)\}.
\] 
So if $g$ is small enough then $g=g_\ga$. \end{proof}

We prove one last lemma which gives conditions for an $\ep$-perturbed gradient segment not to get in and out of a torus several times:
 
\begin{lemm}\label{lemm:hull} Let $\ep<s$. If $\de_0>0$ there is some $0<\de_1<\de_0$ such that if $\ga:A\rightarrow M$ is an $\ep$-gradient segment and $z\in Z$ then 
\[
\ga(CH(\ga^{-1}(R(z,\de_1))))\subseteq R(z,\de_0),
\] 
where $CH$ denotes the convex hull.\end{lemm}

\begin{proof}Let $\de_2<\de_0$ be such that for any $z\in Z$ we have $R(z,\de_2)\subseteq Z_{a_i}'$. If $\de_2$ is small enough, the infimum of the distances between the boundaries $\partial R(z,\de_2/2)$ and $\partial R(z,\de_2)$ for $z\in Z$ is a strictly positive number $d>0$. Let $\de_1<\de_2/2$ be such that for any $z\in Z$ we have
\[
\sup_{R(z,\de_1)}\mu-\inf_{R(z,\de_1)}\mu<2d\frac{s-\ep}{\sqrt{S+3\ep}}.
\] 
We will now prove that $\de_1$ satisfies the conditions of the lemma even replacing $\de_0$ by $\de_2$. To do so we will argue by contradiction. Suppose that $\ga:A\rightarrow M$ is an $\ep$-perturbed gradient segment such that for some $z\in Z$ there exists $t_0<t_1<t_2$ such that $\ga(t_0),\ga(t_2)\in R(z,\de_1)$ but $\ga(t_1)\not\in R(z,\de_2)$ as shown in the following picture:
\begin{center}
\includegraphics{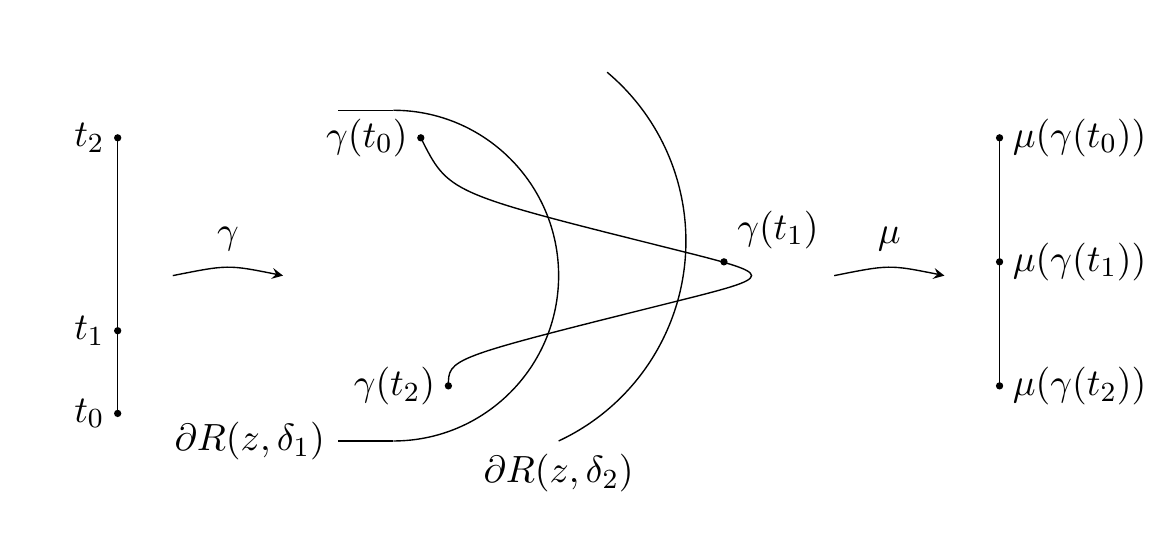}
\end{center}
Let $B=[t_0,t_2]$. Then 
\[
L(\ga(B))=\int_{t_0}^{t_2}\|\ga'(t)\|dt=\int_{t_0}^{t_1}\|\ga'(t)\|dt+\int_{t_1}^{t_2}\|\ga'(t)\|dt\geq 2d.
\]
On the other hand, using the relation between length and energy and both inequalities of Lemma \ref{lemm:inequalities} we get
\[
L(\ga(B))^2\leq|B|\cdot E(\ga(B))<|B|^2(S+3\ep)<\frac{(\mu(\ga(t_0))-\mu(\ga(t_2)))^2}{(s-\ep)^2}(S+3\ep).
\]
Combining both expressions we conclude
\[
2d<\frac{\sqrt{S+3\ep}}{s-\ep}(\mu(\ga(t_0))-\mu(\ga(t_2))),
\]
which is a contradiction with the choice of $\de_1$.
\end{proof}

The previous lemmas are the tools needed to prove the following result, which characterises the intersection of an $\ep$-perturbed gradient segment with the torus $R(z,\de)$ when $\ep$ and $\de$ are small enough:

\begin{prop}\label{prop:goodintersec} There exist positive $\ep,\de>0$ such that if $\ga:A\rightarrow M$ is an $\ep$-perturbed gradient segment and $z\in Z$, then
\begin{enumerate}
\item $\ga^{-1}(R(z,\de))\subseteq A$ is connected.
\item If $\emptyset\neq\ga(A)\cap R(z,\de)\subseteq\partial R(z,\de)$, then $\ga^{-1} (R(z,\de))$ consists of a unique point.
\end{enumerate}
\end{prop}

\begin{proof}
Take $\ep<s$ and $\de_0$ satisfying Lemma \ref{lemm:sqconvex} and let $\de_1$ be the corresponding value given by Lemma \ref{lemm:hull}. Take $\de=\de_1$. By Lemma \ref{lemm:hull}, $B:=CH(\ga^{-1}(R(z,\de)))$ satisfies $\ga(B)\subseteq R(z,\de_0)$. Therefore, according to Lemma \ref{lemm:sqconvex}, the function $g:B\rightarrow\RR$ defined by
\[
g(t)=\{\inf\ze^2:\ \ga(t)\in R(z,\ze)\}
\]
is strictly convex. As a consequence the maximum of $g$ is attained at one of the extremes of $B$. That is
\[
g(t)\leq\max(g(\inf B),g(\sup B)),\ \forall t\in B.
\] 
Since both $\ga(\inf B),\ga(\sup B)\in R(z,\de)$ we have that 
\[
g(\inf B), g(\sup B)\leq\de^2.
\]
Combining this with the inequality above we get that $g(t)\leq\de^2$ for all $t\in B$. Hence $\ga(t)\in R(z,\de)$ for all $t\in B$. In other words, $B\subseteq \ga^{-1}(R(z,\de))$. This means that $\ga^{-1}(R(z,\de))$ coincides with its convex hull, so it is connected. This proves the first part of the proposition.

In the particular case in which the intersection is contained in the boundary, we have that $g\equiv\de^2$.  This is a contradiction with $g$ being strictly convex on $B$ unless $B$ consists of a single point. Since the convex hull of $\ga^{-1}(R(z,\de))$ is a single point, this set must be a single point itself.
\end{proof}

Suppose that the gradient line $\ga:=(\vp_q^{JX})_{|[0,\tau]}:[0,\tau]\rightarrow M$ intersects the interior of $R(z,\de_z)$. Take $U\subseteq R(z,\de_z)$ a connected and simply connected neighbourhood of $\im\ga\cap R(z,\de_z)$, which exists thanks to the deduction before Lemma \ref{lemm:monodecreasing} and Proposition \ref{prop:goodintersec}. Choose a lift $\sig:U\rightarrow R^\#(z,\de_z)$. If $E\in W_z$ is sufficiently small we can define $\ga_E:[0,\tau]\rightarrow M$ by
\[
\left\{
\begin{array}{l}
\ga_E(0)=q\\
\ga_E'(t)=-(J+\tilde{\pi}E_{\sig(\ga_E(t))})X_{\ga_E(t)}
\end{array}
\right.,
\]
where $\pi^*{\cal D}\stackrel{\tilde{\pi}}{\rightarrow}{\cal D}$ is the natural projection: the composition 
\[
U\stackrel{\sig}{\rightarrow}R^\#(z,\de_z)\stackrel{E}{\rightarrow}\pi^*{\cal D}\stackrel{\tilde{\pi}}{\rightarrow}{\cal D}
\] 
is an element of $C^\infty(U,{\cal D})$. We need $E$ to be small in order that $\ga_E(t)$ stays in $U$ for all $t\in\ga^{-1}(R(z,\de_z))$. Then, for $E\in W_z^0$, a small neighbourhood of $0$ in $W_z$, we can define the map 
\[
\begin{array}{rccc}
\Phi_{q,\tau}: & W_z^0 & \longrightarrow & M\\
 &E&\longmapsto&\ga_E(t)
\end{array}.
\]
The same arguments as in the proof of Proposition \ref{prop:phisurj} show that $d_0\Phi_{q,\tau}$ is surjective. Hence, shrinking $W_z^0$ if necessary, we can assume that $\Phi_{q,\tau}$ is a submersion.

If $E$ is sufficiently small, $\ga_E$ is an $\ep$-perturbed gradient segment. Namely we need $\sup_t|d_{\ga_E(t)}\mu(\tilde{\pi}E_{\sig(\ga_E(t))}X_{\ga_E(t)})|<\ep$ and $\sup_t \|\tilde{\pi}E_{\sig(\ga_E(t))}X_{\ga_E(t)}\|^2<\ep$. We shall then define the following norm on $W_z$:
\[
\|E\|_z=\sup_{(\te,p)\in R(z,\de_z)}(|d_p\mu(\tilde{\pi}E_{(\te,p)}X_p)|+\|\tilde{\pi}E_{(\te,p)}X_p\|^2).
\]

The union of the sets $\stackrel{\circ}{R}(z,\de_z)\cap Z$ as $z$ runs over all points of $Z$ is an open covering of $Z$. Since $Z$ is compact there exists a finite number of points $z_1,\ldots,z_\ell\subseteq Z$ such that $Z$ is contained in the union of the interiors of the sets $R(z_i,\de_{z_i})$. To simplify the notation let $R_i=R(z_i,\de_{z_i})$, $R^\#_i=R^\#(z_i,\de_{z_i})$ and $\pi_i:R_i^\#\rightarrow R_i$ be the projection. 

Take $\ep$ small enough so that all the results in this section hold and such that the ball $B(0,\ep)_{\|\cdot\|_z}$ is contained in $W_z^0$. Define
\[
\PP:=\{(E_1,\ldots,E_\ell)\in W_{z_1}\oplus\cdots\oplus W_{z_\ell}:\ \|E_i\|_{z_i}<\ep/\ell\ \forall i\}.
\]
This is the space of perturbations we will work with. In particular, we make the following definition:

\begin{defi}\label{defi:pperturbed} Let $P=(E_1,\ldots,E_\ell)\in\PP$. A $P${\it -perturbed gradient segment} is a tuple $(\ga,\ga_1^\#,\ldots,\ga_\ell^\#)$ where
\begin{enumerate}
\item $\ga:A\rightarrow M$ is a curve on $M$. 
\item $\ga_i^\#:\ga^{-1}R_i\rightarrow R_i^\#$ is a lift of the restriction of $\ga$, i.e. $\pi_i\circ\ga_i^\#=\ga$ on $\ga^{-1}R_i$.
\item The equation
\[
\ga'=-\left(J+\sum_{i=1}^\ell\tilde{\pi}_i(E_i)_{\ga^\#_i}\right)X_\ga.
\]
is satisfied.
\end{enumerate}
\end{defi}

\begin{rema}
We have in particular that $\ga$ is an $\ep$-perturbed gradient segment.
\end{rema}

\begin{rema}\label{rema:actionpgs}
$P$-perturbed gradient segments admit an $S^1$-action given by
\[
\te\cdot(\ga,\ga_1^\#,\ldots,\ga_\ell^\#)=(\te\cdot\ga,\te\cdot\ga_1^\#,\ldots,\te\cdot\ga_\ell^\#)
\]
because the invariance of $\pi_i$ makes $\te\cdot\ga_i^\#$ into a lift of $\te\cdot\ga$ and the equation is preserved by the equivariance of $J$ and $E_i$.
\end{rema}

\section{The moduli space of perturbed broken gradient lines}\label{sec:modpbgd}

In this section we generalise the results of Section \ref{subsec:bgd} where we studied moduli spaces of the form $\MM^{\mu,(J+E)X}$ where $E$ was a suitable perturbation of the almost complex structure $J$. We define analogues of these spaces and we study their structure. Similarly to what happened in Chapter \ref{ch:standard} these analogous moduli spaces admit a stratification indexed by the number of breaking points, but they have the structure of a branched manifold as a consequence of the different possible choices of a lift.

\begin{defi}\label{defi:pbgd}
Let $P\in\PP$. An {\it oriented $P$-perturbed broken gradient line} of $(\mu,JX)$ is a tuple $G=(K;K_1,\ldots,K_\ell;b)$ where $K\subseteq M$ is a compact subset, each $K_j\subseteq R^\#_j$ is a compact (possibly empty) subset and $b\in K$ such that either
\[
K=\{b\}\ \mbox{and}\ K_j=
\left\{
\begin{array}{lr}
\emptyset &\mbox{if}\ b\not\in R_j\\
\{b_j^\#\}\ \mbox{s.t.}\ \pi_j(b_j^\#)=b\ &\mbox{if}\ b\in R_j
\end{array}
\right.
\]
or, if not, then:
\begin{enumerate}
\item $K\cap F$ is finite: let $K\cap F=\{c_1,\ldots,c_r\}$ with $\mu(c_1)>\cdots>\mu(c_r)$ (note that $K\cap F$ may be empty).
\item There exists a homeomorphism $h_K:[0,1]\rightarrow K$ and $P$-perturbed gradient segments
\[
(\ga_k:A_k\rightarrow M,\ga_{k,1}^\#,\ldots,\ga_{k,\ell}^\#),\ k=0,\ldots,r
\]
such that
\begin{itemize}
\item $A_0$ is of the form $[t_0,+\infty)$ and $h_K^{-1}\circ\ga_0:A_0\rightarrow[0,h_K^{-1}(c_1))$ is an increasing homeomorphism.
\item If $0<k<r$, $A_k=\RR$ and $h_K^{-1}\circ\ga_k:A_k\rightarrow(h_K^{-1}(c_k),h_K^{-1}(c_{k+1}))$ is an increasing homeomorphism.
\item $A_r$ is of the form $(-\infty,t_1]$ and $h_K^{-1}\circ\ga_r:A_r\rightarrow(h_K^{-1}(c_r),1]$ is an increasing homeomorphism.
\end{itemize}
and such that
\[
K_j=
\left\{
\begin{array}{lcl}
\im\ga_{0,j}^\# & \mbox{if} & \inf\mu(R_j)>\mu(c_1)\\
\im\ga_{k,j}^\# & \mbox{if} & \mu(c_k)>\sup\mu(R_j)>\inf\mu(R_j)>\mu(c_{k+1}),\ 0<k<r.\\
\im\ga_{r,j}^\# & \mbox{if} & \mu(c_r)>\sup\mu(R_j)\\
\end{array}
\right.
\]
\item Either $b=h_K(0)$ or $b=h_K(1)$.
\end{enumerate}

The point $b$ is called the {\it beginning} of $G$. Similarly, we define the {\it end} of $G$ to be the point
\[
e=\left\{
\begin{array}{llccr}
     b &\mbox{if} & K=\{b\}    &            &  \\
h_K(1) &\mbox{if} & K\neq\{b\} & \mbox{and} & b=h_K(0)\\
h_K(0) &\mbox{if} & K\neq\{b\} & \mbox{and} & b=h_K(1)\\
\end{array}
\right..
\]
In the first case we say that $(K,b)$ is {\it point-like}, in the second case that it is {\it descending} and in the last case that it is {\it ascending}.
\end{defi}

\begin{rema} The picture one should have in mind is the same as in Figure \ref{fig:obroken}, but now with the pieces of $K$ between two breaking points being the image of a $P$-perturbed gradient segment that lifts to the sets $K_j$.
\end{rema}
\begin{rema}\label{rema:singleparameter}
If $K\cap F=\emptyset$ then the parametrisation is given by a single $P$-perturbed gradient segment with $\ga_0:[t_0,t_1]\rightarrow M$ such that the composition $h_K^{-1}\circ\ga_0:[t_0,t_1]\rightarrow[0,1]$ is an increasing homeomorphism.
\end{rema}

Denote by $d_J$ the distance on $M$ induced by the Riemanninan metric $\rho_J$. We can pull $\rho_J$ back to $R^\#_j$ by means of $\pi_j$, and this induces a distance $d_{J,j}$ on $R_j^\#$. Let $d(K_j):=\sup\{d_{J,j}(x,\partial R_j^\#):\ x\in K_j\}$. Given two $P$-perturbed broken gradient lines $G=(K;K_1,\ldots,K_\ell,b)$ and $G'=(K';K_1',\ldots,K_\ell';b')$ define
\[
d^*(G,G'):=d_J^H(K,K')+d_J(b,b')+\sum_{j=1}^\ell d_{J,j}^H(K_j,K_j')(d(K_j)+d(K_j')),
\]
where the superscript $H$ denotes de Hausdorff distance between subsets. This defines a pseudodistance in the set of $P$-perturbed broken gradient lines because there exist $G\neq G'$ such that $d^*(G,G')=0$. This happens when $K=K'$ and for each $j$ such that $K_j\neq K_j'$ both $K_j$ and $K_j'$ are contained in $\partial R_j^\#$, which according to the second part of Proposition \ref{prop:goodintersec}, means that each of them consists of a unique point. We denote by $(\GG^P,d)$ the metric space obtained from this pseudodistance, i.e. $\GG^P$ is obtained by identifying $G$ and $G'$ if $d^*(G,G')=0$ and for $[G],[G']\in\GG^P$ we take $d([G],[G']):=d^*(G,G')$. The bracket notation will be dropped from now on, and by abuse of notation $G$ will represent any element of its class.

We shall now define a stratification of $\GG^P$. The following definitions resemble those in Section \ref{subsec:bgd}: given $I\subseteq\{1,\ldots,n\}$, define 
\[
\GG^P_I:=\{(K;\cdots)\in\GG^P:\ i\in I\Leftrightarrow K\cap C_i\neq\emptyset\}. 
\]
and split this set into
\[
\GG_I^P=\GG_{I\downarrow}^P\sqcup\GG_{I\uparrow}^P\sqcup\GG_{I-}^P
\]
according to its elements being descending, ascending or point-like. Note that if $\#I>1$ there are no point-like elements while there is a homeomorphism
\[
\GG_{\{i\}-}^P=\{(\{b\};\emptyset,\cdots,\emptyset;b):\ b\in C_i\}\simeq C_i
\]
because if $b\in C_i$ it cannot belong to any $R_j$.

\begin{rema}\label{rema:genflip} 
(cf. Remark \ref{rema:flip}) $\GG_{I\downarrow}^P$ and $\GG_{I\uparrow}^P$ are homeomorphic via 
\[
G=(K;K_1,\ldots,K_\ell;b)\mapsto(K;K_1,\ldots,K_\ell;e),
\] 
where $e$ is the end of $G$. The inverse is given by the same expression. These transformations consist in flipping the orientation of the $P$-perturbed gradient line.
\end{rema}

\begin{rema}\label{rema:actionpbgl}
In Remark \ref{rema:actionpgs} we defined an $S^1$-action on $P$-perturbed gradient segments which makes
\[
\te\cdot(K;K_1,\ldots,K_\ell;b)=(\te\cdot K;\te\cdot K_1,\ldots,\te\cdot K_\ell;\te\cdot b)
\]
into an $S^1$-action on $\GG^P$ (if $K\neq\{b\}$ we only need to take $h_{\te\cdot K}=\te\cdot h_K$). The sets $\GG_{I\downarrow}^P$, $\GG_{I\uparrow}^P$ and $\GG_{I-}^P$ are invariant under this action.
\end{rema}

We shall begin studying $\GG_\emptyset^P$. If $G=(K;\ldots;b)\in\GG_\emptyset^P$ is such that $K\neq\{b\}$, then, according to Definition \ref{defi:pbgd}, $G$ is parametrised by a single $P$-perburbed gradient segment with first component $\ga:=\ga_0:[t_0,t_1]\rightarrow M$ such that $h_K(0)=\ga(t_0)$ and $h_K(1)=\ga(t_1)$. With this notation consider the map (cf. Lemma \ref{lemm:densestratum})
\[
\begin{array}{rccc}
\al_\emptyset^P:&\GG_\emptyset^P &\longrightarrow & \RR\times(M\setminus F)\\
 &G=(K;\cdots;b)&\longmapsto&(\ga^{-1}(e)-\ga^{-1}(b),b)
\end{array},
\]
where it is implicitly understood that if $G$ is point-like then $\ga^{-1}(e)-\ga^{-1}(b)=0$, although $\ga$ is not defined. Note that the image of a descending element is $(t_1-t_0,b)$ while the image of an ascending element is $(t_0-t_1,b)$. This is well defined because the difference $t_1-t_0$ does not depend on the chosen parametrisation $\ga$. Another obvious observation is that this map does not depend on $K_1,\ldots,K_\ell$, so it is not injective in general: loosely speaking, given an element of $\RR\times(M\setminus F)$, it has as many preimages as choices of lifts we can make. Therefore it is important to identify the points where the number of lifts we can take changes: 
\begin{defi} 
An oriented $P$-perturbed gradient line $G=(K;\ldots;b)$ is said to be {\it tangent to $R_j$} if $\emptyset\neq K\cap R_j\subseteq\partial R_j$, which by Proposition \ref{prop:goodintersec} is a unique point. 
\end{defi}
It is precisely a these tangent points that the number of lifts varies. We illustrate this in the following figure, where $o_j$ denotes the number of elements of the stabiliser of $z_j$:

\begin{figure}[H]
	\centering
		\includegraphics{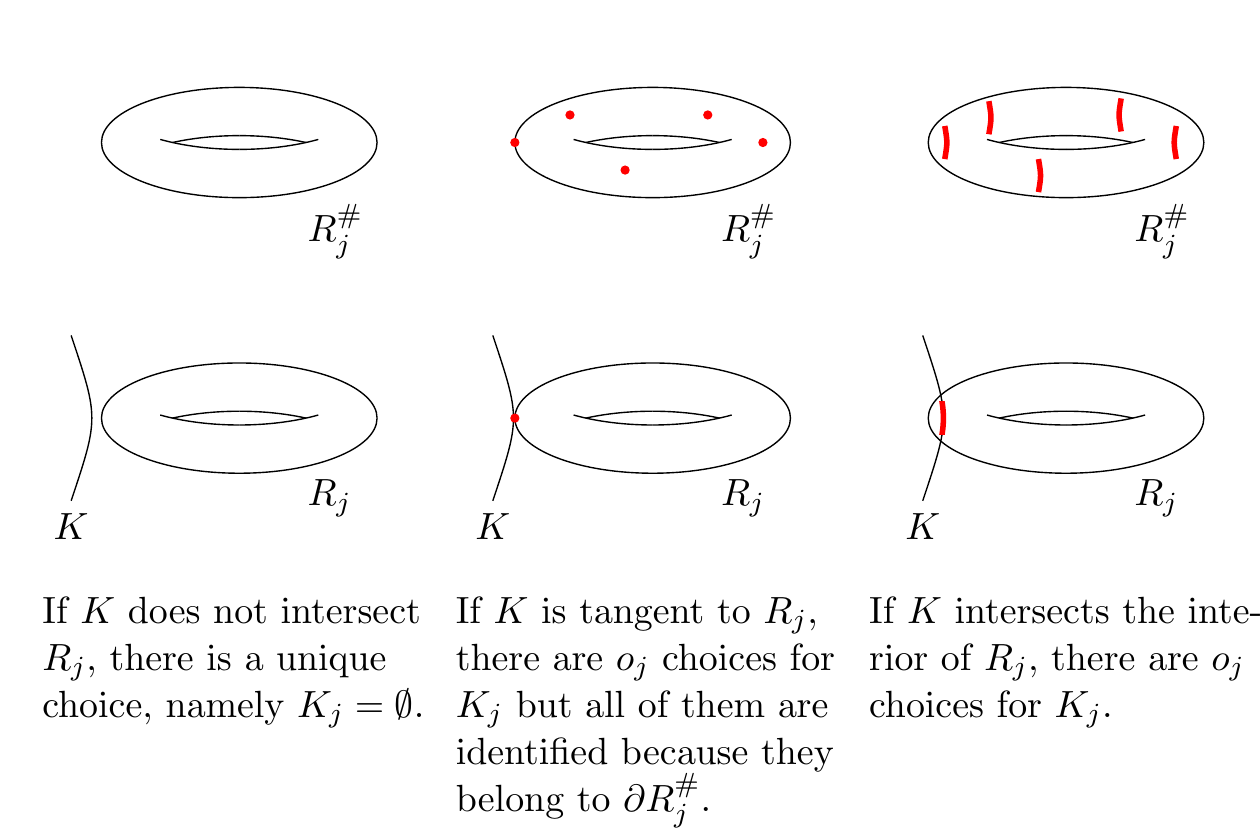}
	\caption{Lifts of a perturbed gradient line.}
	\label{fig:branching}
\end{figure}

The figure shows that if $G$ is tangent to $R_j$, in a neighbourhood of $G$ in $\GG_\emptyset^P$ there are elements with only one choice for $K_j$ and elements with $o_j$ choices for $K_j$. Not surprisingly $G$ will turn out to be a branching point as explained in Section \ref{sec:branched}. The following proposition describes the local structure of $\GG_\emptyset^P$:

\begin{prop}\label{prop:gzstruc}
Let $G=(K;\cdots;b)\in\GG_\emptyset^P$ and let
\begin{itemize}
\item $\Lm_0$ be the set of $j$ such that $K$ intersects the interior of $R_j$ 
\item $\Lm_1$ be the set of $j$ such that $K$ is tangent to $R_j$
\end{itemize}
If $\Lm_1=\emptyset$ --if $G$ is not tangent to any $R_j$-- let $N=1$. Otherwise, if $\Lm_1=\{j_1,\ldots,j_s\}$, let $N=o_{j_1}\cdots o_{j_s}$. There exist open subsets $W_1,\ldots,W_N\subseteq\RR\times(M\setminus F)$ and continuous maps $\vp_k:W_k\rightarrow\GG_\emptyset^P$ such that
\begin{enumerate}
\item $\al_\emptyset^P\circ\vp_k:W_k\longrightarrow\RR\times(M\setminus F)$ is the identity on $W_k$.
\item $\vp_1(W_1)\cup\cdots\cup\vp_N(W_N)$ is a neighbourhood of $G$ in $\GG_\emptyset^P$.
\end{enumerate}
\end{prop}

\begin{proof}
For each $j\in\Lm_0$, let $U_j^\#\subseteq R_j^\#$ be a small open neighbourhood of $K_j$ such that $\pi_j:U_j^\#\rightarrow U_j:=\pi_j(U_j^\#)$ is a diffeomorphism of open manifolds with boundary and let $\sig_j$ be the inverse.

Similarly, for each $j\in\Lm_1$, choose $q_j\in\pi^{-1}(K\cap R_j)$ --there are $o_j$ choices-- and let $U_j^\#\subseteq R_j^\#$ be a small open neighbourhood of $q_j$ such that the projection $\pi_j:U_j^\#\rightarrow U_j:=\pi_j(U_j^\#)$ is a diffeomorphism of open manifolds with boundary. Let $\sig_j$ be the inverse of this diffeomorphism.

Now let $B\subseteq M$ be an open neighbourhood of $K$ such that $B\cap R_j\subseteq U_j$ for every $j\in \Lm_0\cup \Lm_1$. On $B$ we define the vector field
\[
V=\left(J+\sum_{j\in \Lm_0\cup \Lm_1}\tilde{\pi}_j(E_j)_{\sig_j}\right)X
\]
where, remember, $P=(E_1,\ldots,E_\ell)$. Then all the integral curves of $-V$ are $\ep$-perturbed gradient segments.

Assume $G=(K;K_1,\ldots,K_\ell; b)$ is descending. Then it is parametrised by an $\ep$-perturbed gradient segment $(\ga:[t_0,t_1]\rightarrow M,\ga_1^\#,\ldots,\ga_\ell^\#)$ with $\ga(t_0)=b$. Note that $\ga$ is an integral curve of $-V$.

Now, for every $(t,p)$ in a sufficiently small neighbourhood $W$ of $(t_1-t_0,p)$ in $\RR\times(M\setminus F)$ there is a unique integral curve $\ga_{t,p}:[0,t]\rightarrow M$ of $-V$ such that $\ga_{t,p}(0)=p$. Define $K_{t,p}=\im\ga_{t,p}$. Taking $W$ small enough we can assume that $K_{t,p}\cap R_j$ is non-empty if and only if $j\in \Lm_0\cup \Lm_1$. In this case let $K_{t,p,j}=\sig_j(K_{t,p}\cap R_j)$. Otherwise let $K_{t,p,j}=\emptyset$. Define
\[
\vp(t,p)=(K_{t,p};K_{t,p,1},\ldots,K_{t,p,\ell};p),
\]
which is an element of $\GG_{\emptyset\downarrow}^P$ because it is parametrised by the $P$-perturbed gradient segment $(\ga_{t,p},\ga_{t,p,1}^\#,\ldots,\ga_{t,p,\ell}^\#)$ where, if $j\in \Lm_0\cup \Lm_1$ --so $K_{t,p}\cap R_j$ is non-empty--, $\ga_{t,p,j}^\#:\ga_{t,p}^{-1}(R_j)\rightarrow R_j^\#$ equals $\sig_j\circ\ga_{t,p}$.

Since $\al_\emptyset^P(K_{t,p};\ldots;p)=(t,p)$ we have that $\al_\emptyset^P\circ\vp$ is the identity on $W$.

If $\Lm_1=\emptyset$ this already proves the first claim (for descending elements) because there are no choices involved. Otherwise for each $j\in\Lm_1=\{j_1,\ldots,j_s\}$ we can make $o_j$ choices for a total of $o_{j_1}\cdots o_{j_s}=N$ choices. Namely, the set $Q_i:=\pi^{-1}(K\cap R_{j_i})$ has $o_{j_i}$ points and for each of the $N$ elements of the form $q:=(q_1,\ldots,q_s)\in Q_1\times\cdots\times Q_s=:Q$ we can construct a neighbourhood $B_q$, a vector field $V_q$, an open subset $W_q$ and a map $\vp_q$.

Analogous arguments prove the first claim for the cases in which $G$ is either ascending or point-like.

In order to prove the second claim observe that thanks to Proposition \ref{prop:goodintersec}, if $G'=(K';K_1',\ldots,K_\ell';b')\in\GG_\emptyset^P$, each intersection $K'\cap R_j$ is connected, so if $G'$ lies near $G$, the compact $K'$ must be the image of an integral curve of $V_q$ for some $q\in Q$.
\end{proof}

From this proposition we deduce that $\GG_\emptyset^P$ is a $2m+1$-dimensional branched manifold (this concept was explained in Section \ref{sec:branched}) with $G$ being a branching point if and only if $\Lm_1\neq\emptyset$. If $G$ is not a branching point, i.e. it lies on a branch, we make the following definition: 

\begin{defi}\label{defi:weight}
The {\it weight} of $G\in\GG_\emptyset^P\setminus(\GG_\emptyset^P)^\prec$ is the rational number $\prod_{j\in\Lm_0}\frac{1}{o_j}$.
\end{defi}

\begin{rema}
The intuitive idea behind this definition is that when a branch reaches a branching point and it splits into $o_j$ branches its weight is equally distributed among the new branches. This notion will be used in Chapter \ref{ch:biinvariant} to define weightings on certain train tracks (see Definition \ref{defi:weighting}).
\end{rema}

We move forward on to studying the structure of $\GG_{\{i\}}^P$. The techniques we use are very similar to those in Proposition \ref{prop:gzstruc}. Compare also with Lemma \ref{lemm:codimonestratum}.

\begin{prop}\label{prop:gostruc} There is homeomorphism $\GG_{\{i\}-}^P\simeq C_i$. Otherwise, given $G=(K;\cdots;b)\in\GG_{\{i\}}^P$ not point-like define $\Lm_0,\Lm_1$ and $N$ as in Proposition \ref{prop:gzstruc}. Then there exist $2m$-dimensional submanifolds $W_1,\ldots,W_N\subseteq\MTM$ and continuous maps $\vp_k:W_k\rightarrow\GG_{\{i\}\downarrow}^P$ (resp. $\GG_{\{i\}\uparrow}^P$) such that
\begin{enumerate}
\item If $b_{(p_0,p_1)}, e_{(p_0,p_1)}$ are the beginning and the end of $\vp_k(p_0,p_1)$ then
\[
\begin{array}{ccc}
W_k &\longrightarrow&\MTM\\ 
(p_0,p_1)&\longmapsto & (b_{(p_0,p_1)},e_{(p_0,p_1)})
\end{array}
\] 
is the identity on $W_k$.
\item $\vp_1(W_1)\cup\cdots\cup\vp_N(W_N)$ is a neighbourhood of $G$ in $\GG_{\{i\}\downarrow}^P$ (resp. $\GG_{\{i\}\uparrow}^P$).
\end{enumerate}
\end{prop}

\begin{proof}
We already proved the existence of a homeomorphism $\GG_{\{i\}-}^P\simeq C_i$. 

Suppose now that $G=(K;\ldots;b)$ is descending. Then $G$ is parametrised by 
\[
\left\{
\begin{array}{lcl}
(\ga_0:[t_0,+\infty)\rightarrow M,\ga_{0,1}^\#,\ldots,\ga_{0,\ell}^\#)& \mbox{with} & \ga_0(t_0)=b\\
(\ga_1:(-\infty,t_1]\rightarrow M,\ga_{1,1}^\#,\ldots,\ga_{1,\ell}^\#)& \mbox{with} & \ga_1(t_1)=e\\
\end{array}
\right.
\]
such that $\lim_{t\rightarrow+\infty}\ga_0(t)=\lim_{t\rightarrow-\infty}\ga_1(t)\in C_i$.

Define $Q$, $B_q$ and $V_q$ as in Proposition \ref{prop:gzstruc}. Then $\ga_0,\ga_1$ are gradient lines of the vector field $-V_q$. Let $\xi_q$ be the flow of $-V_q$, which is defined on $B_q$, and take the corresponding (un)stable manifolds of $C_i$:
\[
S_q=\{p\in B_q:\ \lim_{t\rightarrow +\infty}(\xi_q)_t(p)\in C_i\},
\]
\[
U_q=\{p\in B_q:\ \lim_{t\rightarrow -\infty}(\xi_q)_t(p)\in C_i\}.
\]

For every $(p_0,p_1)$ in a small enough neighbourhood $W_q$ of $(h_K(0),h_K(1))=(b,e)$ in $S_q\times_{C_i}U_q$ there exist unique integral curves $\ga_{p_0}:[0,+\infty)\rightarrow M$, $\ga_{p_1}:(-\infty,0]\rightarrow M$ of $-V_q$ such that $\ga_{p_0}(0)=p_0$, $\ga_{p_1}(0)=p_1$. They satisfy
\[
c_i:=\lim_{t\rightarrow+\infty}\ga_{p_0}(t)=\lim_{t\rightarrow-\infty}\ga_{p_1}(t)\in C_i.
\]
Define 
\[
K_{(p_0,p_1)}=\im\ga_{p_0}\sqcup\{c_i\}\sqcup\im\ga_{p_1}.
\]
Taking $W_q$ small enough we can assume that $K_{(p_0,p_1)}\cap R_j$ is non-empty if and only if $j\in \Lm_0\cup \Lm_1$. In this case let $K_{(p_0,p_1),j}=\sig_j(K_{(p_0,p_1)}\cap R_j)$. Otherwise let $K_{(p_0,p_1),j}=\emptyset$. Define
\[
\vp_q(p_0,p_1)=(K_{(p_0,p_1)};K_{(p_0,p_1),1},\ldots,K_{(p_0,p_1),\ell};p).
\]
Using the homeomorphism 
\[
\begin{array}{rccc}
h_{K_{(p_0,p_1)}}:&[0,1]&\longrightarrow & K_{(p_0,p_1)}\\
 & t &\longmapsto &\left\{
\begin{array}{lcl}
\ga_{p_0}(\tan(\pi t)) &\mbox{if} & 0\leq t<1/2\\
c_i &\mbox{if} &t=1/2\\
\ga_{p_1}(-\tan(\pi(1-t))) &\mbox{if} & 1/2 <t\leq 1
\end{array}
\right.
\end{array}
\]
and the same arguments than in Proposition \ref{prop:gzstruc} we can see that $\vp_q(p_0,p_1)$ is an element of $\GG_{\{i\}\downarrow}^P$. The definition of $h_{K_{(p_0,p_1)}}$ already shows that the first claim of the proposition is satisfied. The second part is again proved by the same arguments as in Proposition \ref{prop:gzstruc}.

Reversing the proof or by means of Remark \ref{rema:genflip} we get the result for ascending $G$.
\end{proof}

Finally we study the case in which $\#I>1$. The result we want to prove is the analogue of Lemma \ref{lemm:codimtwostrata} and again we will face some transversality issues that will need to be addressed. To do this we need to take auxiliary regular levels $Y_{a_1}^-,\cdots,Y_{a_{\nu-1}}^-$ and $Y_{a_2}^+,\ldots,Y_{a_\nu}^+$ of $\mu$ such that, for $i=2,\ldots,\nu$ (cf. Figure \ref{fig:auxlevelspertzone})
\[
a_{i-1}>\mu(Y_{a_{i-1}})>\sup_{Z_{a_i}'}\mu>\inf_{Z_{a_i}'}\mu>\mu(Y_{a_i}^+)>a_i.
\]
Assuming some transversality conditions on $P$ that will be discussed later, the proposition is

\begin{prop}Let $I=\{i_1,\ldots,i_r\}$ with $r>1$ and $\mu(C_{i_1})>\cdots>\mu(C_{i_r})$. Given $G=(K;\cdots;b)\in\GG_I^P$ define $\Lm_0,\Lm_1$ and $N$ as in Proposition \ref{prop:gzstruc}. Then there exist $(2m-r+1)$-dimensional submanifolds $W_1,\ldots,W_N\subseteq M\times\stackrel{r+1}{\cdots}\times M$ and continuous maps $\vp_k:W_k\rightarrow\GG_{I\downarrow}^P$ (resp. $\GG_{I\uparrow}^P$) such that
\begin{enumerate}
\item If $\vp_k(p_0,\cdots,p_r)=(K_{(p_0,\cdots,p_r)};\ldots)$ and $K_{(p_0,\ldots,p_r)}\cap Y_{a_i}^+=:\{y_i\}$, then
\[
\begin{array}{ccc}
W_k &\longrightarrow& M\times\stackrel{r+1}{\cdots}\times M\\ 
(p_0,\ldots,p_r)&\longmapsto & (h_{K_{(p,\ldots,p_r)}}(0),y_2,\ldots,y_r,h_{K_{(p_0,\ldots,p_r)}}(1))
\end{array}
\]
is the identity on $W_k$.
\item $\vp_1(W_1)\cup\cdots\cup\vp_N(W_N)$ is a neighbourhood of $G$ in $\GG_{I\downarrow}^P$ (resp. $\GG_{I\uparrow}^P$).
\end{enumerate}
\end{prop}

\begin{proof}
The proof is almost identical to that of Proposition \ref{prop:gostruc}:

Suppose that $G$ is descending. Then it is parametrised by
\[
\left\{
\begin{array}{lcl}
(\ga_0:[t_0,+\infty)\rightarrow M,\ga_{0,1}^\#,\ldots,\ga_{0,\ell}^\#)& \mbox{with} & \ga_0(t_0)=h_K(0)=b\\
(\ga_k:(-\infty,+\infty)\rightarrow M,\ga_{k,1}^\#,\ldots,\ga_{k,\ell}^\#) &\mbox{for}& 1\leq k\leq r-1\\
(\ga_r:(-\infty,t_1]\rightarrow M,\ga_{r,1}^\#,\ldots,\ga_{r,\ell}^\#)& \mbox{with} & \ga_r(t_1)=h_K(1)\\
\end{array}
\right.
\]
such that $\lim_{t\rightarrow+\infty}\ga_{k-1}(t)=\lim_{t\rightarrow-\infty}\ga_k(t)\in C_{i_k}$, for $k=1,\ldots,r$.

Define $Q$, $B_q$ and $V_q$ as in Proposition \ref{prop:gzstruc} and $\xi_q$ as in Proposition \ref{prop:gostruc}. Consider the (un)stable manifolds
\[
S_{i_k,q}=\{p\in B_q:\ \lim_{t\rightarrow +\infty}(\xi^q)_t(p)\in C_{i_k}\},
\]
\[
U_{i_k,q}=\{p\in B_q:\ \lim_{t\rightarrow -\infty}(\xi^q)_t(p)\in C_{i_k}\}
\]
and let
\[
\begin{array}{l}
\WW_{I,q}:=S_{i_1,q}\times_{C_{i_1}}(U_{i_1,q}\cap S_{i_2,q}\cap Y_{f(C_{a_{i_2}})}^+)\times_{C_{i_2}}\cdots\\
\hspace{1.5cm}\cdots\times_{C_{i_{r-1}}}(U_{i_{r-1},q}\cap S_{i_r,q}\cap Y_{f(C_{i_r})}^+)\times_{C_{i_r}}U_{i_r,q},
\end{array}
\]
which is a $(2m-r+1)$-dimensional submanifold of $B_q\times\stackrel{r+1}{\cdots}\times B_q$ provided that some transversality conditions are satisfied.

For $k=2,\ldots,r$, let $y_k^K$ be the unique point in $K\cap Y_{a_{i_k}}^+$. Then for every $(p_0,\ldots,p_r)$ in a small enough neighbourhood $W_q$ of $(h_K(0),y_2^K,\ldots,y_r^K,h_K(1))$ in $\WW_{I,q}$ there are unique integral curves 
$\ga_{p_0}:\RR_{\geq 0}\rightarrow M$, $\ga_{p_r}:\RR_{\leq 0}\rightarrow M$, $\ga_{p_k}:\RR\rightarrow M$ ($0<k<r$) of $-V_q$ such that $\ga_{p_0}(0)=p_0$, $\ga_{p_r}(0)=p_r$ and $(\im\ga_{p_k})\cap Y_{i_{k+1}}^+=\{p_k\}$. For $k=1,\ldots,r$ they satisfy 
\[
\lim_{t\rightarrow+\infty}\ga_{p_{k-1}}(t)=\lim_{t\rightarrow-\infty}\ga_{p_k}(t)=:c_k\in C_{i_k}.\]
Define 
\[
K_{(p_0,\ldots,p_r)}=\im\ga_{p_0}\sqcup\{c_1\}\sqcup\cdots\sqcup\im\ga_{p_{r-1}}\sqcup\{c_r\}\sqcup\im\ga_{p_r}.
\]
Taking $W_q$ small enough we can assume that $K_{(p_0,\ldots,p_r)}\cap R_j$ is non-empty if and only if $j\in \Lm_0\cup \Lm_1$. In this case let $K_{(p_0,\ldots,p_r),j}=\sig_j(K_{(p_0,\ldots,p_r)}\cap R_j)$. Otherwise let $K_{(p_0,\ldots,p_r),j}=\emptyset$. Let
\[
\vp_q(p_0,\ldots,p_r)=(K_{(p_0,\ldots,p_r)};K_{(p_0,\ldots,p_r),1},\ldots,K_{(p_0,\ldots,p_r),\ell};p_0),
\]
and consider the homeomorphism 
\[
\begin{array}{rccc}
h_{K_{(p_0,\ldots,p_r)}}:&[0,1]&\longrightarrow & K_{(p_0,\ldots,p_r)}\\
 & t &\longmapsto &\left\{
\begin{array}{lcl}
\ga_{p_0}(\tan(\frac{\pi}{2}(r+1)t)) &\mbox{if} & 0\leq t<\frac{1}{r+1}\\
\ga_{p_k}(\tan(\pi(r+1)(t-\frac{2k+1}{2(r+1)})) &\mbox{if} & \frac{k}{r+1}<t<\frac{k+1}{r+1}\\
\ga_{p_r}(-\tan(\frac{\pi}{2}(r+1)(1-t))) &\mbox{if} &\frac{r}{r+1}<t\leq 1\\
c_k &\mbox{if} &t=\frac{k}{r+1}\\
\end{array}
\right.
\end{array}
\]
where $0<k<r$ in the definition at the intervals $(\frac{k}{r+1},\frac{k+1}{r+1})$ and $0<k\leq r$ in the definition at the points $\frac{k}{r+1}$.

The same arguments of propositions \ref{prop:gzstruc}, \ref{prop:gostruc} conclude the proof for descending $G$. The result also applies for an ascending $G$ by means of Remark \ref{rema:genflip}.
\end{proof}

The validity of this proposition is subject to the fact that the fibred products $\WW_{I,q}$ are indeed submanifolds. Note that $B_q$ does not depend on $P$, but the vector field $V_q$ does and hence so does $\WW_{I,q}$. For that reason let us write $V^P_q$ and let $\xi^P_q$ be the flow of $-V^P_q$. We define (un)stable manifolds
\[
S^P_{i_k,q}=\{p\in B_q:\ \lim_{t\rightarrow +\infty}(\xi^P_q)_t(p)\in C_{i_k}\},
\]
\[
U^P_{i_k,q}=\{p\in B_q:\ \lim_{t\rightarrow -\infty}(\xi^P_q)_t(p)\in C_{i_k}\}
\]
and the fibred product
\[
\begin{array}{l}
\WW^P_{I,q}:=S_{i_1,q}^P\times_{C_{i_1}}(U_{i_1,q}^P\cap S_{i_2,q}^P\cap Y_{f(C_{i_2})}^+)\times_{C_{i_2}}\cdots\\
\hspace{1.5cm}\cdots\times_{C_{i_{r-1}}}(U_{i_{r-1},q}^P\cap S_{i_r,q}^P\cap Y_{f(C_{i_r})}^+)\times_{C_{i_r}}U_{i_r,q}^P,
\end{array}
\]

Similarly to Definition \ref{defi:regularity}, let
\[
\begin{array}{l}
{\cal S}_{I,q}^P:=S_{i_1,q}^P\times(U_{i_1,q}^P\cap Y_{f(C_{i_1})f(C_{i_2}),q}^P)\times (S_{i_2,q}^P\cap Y_{f(C_{i_2}),q}^+)\times\cdots\\
\hspace{1.5cm}\cdots\times(U_{i_{r-1},q}^P\cap Y_{f(C_{i_{r-1}})f(C_{i_r}),q}^P)\times(S_{i_r,q}^P\cap Y_{f(C_{i_r}),q}^+)\times U_{i_r,q}^P,
\end{array}
\]
where $Y_{a_ia_j,q}^P=\{p\in B_q\cap Y_{a_i}^-:\ \exists t\ \mbox{s.t.}\ (\xi^P_q)_t(p)\in Y_{a_j}^+\}$ and also let
\[
{\cal T}_I:=(C_{i_1}\times C_{i_1})\times\cdots\times(C_{i_r}\times C_{i_r})\times(Y_{a_{i_2}}^+\times Y_{a_{i_2}}^+)\times\cdots\times(Y_{a_{i_r}}^+\times Y_{a_{i_r}}^+)
\]
and
\[
{\bf\De}_I:=\De_{C_{i_1}\times C_{i_1}}\times\cdots\times\De_{C_{i_r}\times C_{i_r}}\times\De_{Y_{a_{i_2}}^+\times Y_{a_{i_2}}^+}\times\cdots\times\De_{Y_{a_{i_r}}^+\times Y_{a_{i_r}}^+}.
\]
Define the map $H_{I,q}^P:{\cal S}_{I,q}^P\rightarrow{\cal T}_I$ in the same fashion as in Definition \ref{defi:regularity}. If this map is transverse to ${\bf\De}_I$, then the preimage $(H_{I,q}^P)^{-1}({\bf\De}_I)$ is a smooth manifold homeomorphic to $\WW_{I,q}^P$ (see Lemma \ref{lemm:homeo}).

We can apply arguments analogous to those of Section \ref{subsec:perturb} to prove that for a generic choice of $P\in\PP$ the transversality condition above is satisfied: let
\[
S_{i_k,q}=\{(p,P)\in V_q\times\PP:\ p\in S_{i_k,q}^P\},
\]
\[
U_{i_k,q}=\{(p,P)\in V_q\times\PP:\ p\in U_{i_k,q}^P\}
\]
and also let $Y_{a_ia_j,q}:=\{(p,P)\in (V_q\cap Y_i^-)\times\PP:\ p\in Y_{a_ia_j,q}^P\}.$
Finally we consider the fibre bundle over $\PP$ defined by
\[
\begin{array}{l}
{\cal S}_{I,q}:=S_{i_1,q}\times_\PP(U_{i_1,q}\cap Y_{f(C_{i_1})f(C_{i_2}),q})\times_\PP(S_{i_2,q}\cap (Y_{f(C_{i_2})}^+\times\PP))\times_\PP\cdots\\
\hspace{1.5cm}\cdots\times_\PP(U_{i_{r-1},q}\cap Y_{f(C_{i_{r-1}})f(C_{i_r}),q})\times_\PP(S_{i_r,q}\cap (Y_{f(C_{i_r})}^+\times\PP))\times_\PP U_{i_r,q}.
\end{array}
\]
Let $H_{I,q}:{\cal S}_{I,q}\rightarrow T_I$ be the maps whose restriction to the fibre over $P$ is the map $H_{I,q}^P$. Then Theorem \ref{theo:genericity} shows that $H_{I,q}$ is transverse to ${\bf \De}_I$ and, as a consequence, for a generic choice of $P\in\PP$, the map $H_{I,q}^P$ is also transverse to ${\bf \De}_I$.

%% file: Biinvariant.tex
\chapter{Biinvariant diagonal classes}\label{ch:biinvariant}


Following the notation used in previous chapters let $(M,\om)$ be a compact connected symplectic manifold of dimension $2m$, endowed with an effective Hamiltonian $S^1$-action with moment map $\mu:M\rightarrow\RR$. Let $J$ be an invariant almost complex structure compatible with $\om$ and $\rho_J:=\om(\cdot,J(\cdot))$ be the corresponding Riemmanian metric. Finally let $X$ be the vector field generated by the infinitesimal action and let $\xi_t:=\xi_t^{JX}$ denote the flow of $-JX$ at time $t$.

Consider the $\STS$ actions on $\MTM$ and $\CC P^1$ defined respectively by
\[
(\te,\ze)(p,q)=(\te\cdot p,\ze\cdot q),
\]
\[
(\te,\ze)[z:w]=[\te z:\ze w].
\]
where we are identifying $S^1$ with the complex numbers of norm $1$.

In the first part of this chapter is we construct an $H^*(B(\STS);\QQ)$-module homomorphism
\[
\lm:H_{\STS}^*(\CC P^1)\longrightarrow H_{\STS}^{*+2m-2}(\MTM).
\]
that we will call the {\it lambda-map}. We will deal first with the easier case in which the $S^1$-action on $M$ is semi-free (as in Chapter \ref{ch:standard}) to provide a model for the general case in which the action has non-trivial finite stabilisers (as in Chapter \ref{ch:multivalued}). For semi-free actions the lambda-map can be defined with integer coefficients while in the general case it will be defined with rational coefficients. 

Later we introduce some objects called {\it global biinvariant diagonal classes}, which allow to recover Kirwan surjectivity \cite{Kir} in our context. The lambda-map will be used to construct such objects: $\lm(1)$ turns out to be a biinvariant diagonal class. By further studying the lambda-map we can get some extra results, notably we can describe a global biinvariant diagonal class of a product manifold in terms of the lambda-maps of its components. Finally, we address the question of uniqueness of global biinvariant diagonal classes; we will show that they are not unique in general by means of very explicit computations.

\section{The lambda-map for semi-free actions}\label{sec:lambdasimple}

In this section we assume that the action of $S^1$ on $M$ is semi-free, and so we can apply the results of Chapter \ref{ch:standard}. In particular there exists a space $\PP$ of $S^1$-invariant perturbations of $J$ and a residual subset $R\subseteq\PP$ such that for all $E\in R$, the Morse-Bott pair $(\mu,(J+E)X)$ is regular (Theorem \ref{theo:generictranssemifree}). Then for $E\in R$ the moduli space of broken gradient lines $\MM^E:=\MM^{\mu,(J+E)X}$ satisfies the conditions of Theorem \ref{theo:stratification}. According to this theorem $\MM^E$ admits a stratification indexed by the critical components of $\mu$. In this section we denote by $\xi_t^E$ the flow of the vector field $-(J+E)X$ at time $t$.

Endow $D:=M\times\RR\times S^1$ with the $\STS$-action
\[
(\te,\ze)(p,t,\al)=(\te\cdot p,t,\ze\al\bar{\te}).
\]

\begin{lemm}\label{lemm:fgareequivsimple} 
The maps 
\[
\begin{array}{rccc}
f^E:&D&\longrightarrow&\MTM\\
 &(p,t,\al)&\longmapsto&(p,\al\cdot\xi^E_t(p))
\end{array},
\]
\[
\begin{array}{rccc}
g:&D&\longrightarrow&\CC P^1\\
 &(p,t,\al)&\longmapsto&[1:2^t\al]
\end{array}
\]
are $\STS$-equivariant.
\end{lemm}

\begin{proof} First we have
\[
\begin{array}{rcl}
f^E((\te,\ze)(t,p,\al)) & = & f^E(\te\cdot p,t,\ze\al\bar{\te})\\
                        & = & (\te\cdot p,\ze\al\bar{\te}\cdot\xi^E_t(\te\cdot p))\\
                        & = & (\te\cdot p,\ze\al\bar{\te}\te\cdot\xi^E_t(p))\\
                        & = & (\te\cdot p,\ze\al\cdot\xi^E_t(p))\\
                        & = & (\te,\ze)(p,\al\cdot\xi^E_t(p))\\
                        & = & (\te,\ze)f(p,t,\al)
\end{array}.
\]
On the other hand,
\[
\begin{array}{rcl}
g((\te,\ze)(t,p,\al)) & = & g(\te\cdot p,t,\ze\al\bar{\te})\\
 & = & \left[1:2^t\ze\al\bar{\te}\right]\\
 & = & \left[\te:2^t\ze\al\right]\\
 & = & (\te,\ze)\left[1:2^t\al\right]\\
 & = & (\te,\ze)g(p,t,\al)\\
\end{array}.
\]
\end{proof}

The $\STS$-principal bundles $S^{2k+1}\times S^{2k+1}\rightarrow\CC P^k\times\CC P^k$ approximate the classifying bundle of $\STS$. Consider the spaces
\[D_k:=D\times_{\STS}(S^{2k+1}\times S^{2k+1}),\]
\[(\MTM)_k:=(\MTM)\times_{\STS}(S^{2k+1}\times S^{2k+1}),\]
\[\CC P^1_k:=\CC P^1\times_{\STS}(S^{2k+1}\times S^{2k+1}).\]
Since the maps $f^E,g$ are $\STS$-equivariant they induce maps 
\[
\begin{diagram}
\node{}\node{D_k}\arrow{sw,t}{g_k}\arrow{se,t}{f_k^E}\node{}\\
\node{\CC P^1_k}\node{}\node{(\MTM)_k}
\end{diagram}.
\]

What we want to do is to construct a map in equivariant cohomology by stabilisation of the maps $(f_k^E)^!g_k^*$. Since $D_k$ is not compact we cannot define $(f_k^E)^!$ directly and we will need to rely on pseudocycles. We shall first study how to define an $\STS$-action on broken gradient lines which have at least one breaking point:

Let $I=\{i_1,\ldots,i_r\}\neq\emptyset$ with $\mu(C_{i_1})>\cdots>\mu(C_{i_r})$. For $(K,b)\in\MM_I^E$ we define
\begin{itemize}
\item $K^+:=\{p\in K:\ \mu(p)\geq C_{i_1}\}$
\item $K^-:=\{p\in K:\ \mu(p)\leq C_{i_r}\}$
\item $K^0:=\{p\in K:\ \mu(C_{i_1})\geq\mu(p)\geq\mu(C_{i_r})\}$
\end{itemize}
so that $K^+$ consists of the points of $K$ above the top breaking point and $K^-$ consist of the points below the bottom breaking point, while $K^0$ consists of the points in between (see Figure \ref{fig:obroken}). If $(K,b)$ is descending let $K_b=K^+$ and $K_e=K^-$ and viceversa if $(K,b)$ is ascending. If $(K,b)$ is point-like, just take $K_b=K_e=\{b\}$.

Given $(\te,\ze)\in\STS$ consider the set $K_{(\te,\ze)}=\te K_b\cup K^0\cup \ze K_e$. Then $(K_{(\te,\ze)},\te\cdot b)$ is also an element of $\MM_I^E$ (compare with Definition \ref{defi:brokenline}):

If $(K,b)$ is point-like, then $K=\{b\}$ and $b$ is a fixed point of the action, so $(K_{(\te,\ze)},\te\cdot b)=(K,b)$. Otherwise we have that $K_{(\te,\ze)}\cap F=K\cap F$ and that the map
\[
\begin{array}{rccc}
h_{K_{(\te,\ze)}}: & [0,1] & \longrightarrow & K_{(\te,\ze)}\\
 &t&\longmapsto &\left\{
\begin{array}{ll}
\te\cdot h_K(t) & \mbox{if}\ h_K(t)\in K_b\\
         h_K(t) & \mbox{if}\ h_K(t)\in K^0\\
\ze\cdot h_K(t) & \mbox{if}\ h_K(t)\in K_e
\end{array}
\right.
\end{array}
\]
satisfies all the required conditions because the flow $\xi^E$ is $S^1$-equivariant. Finally we have that $\te\cdot b=\te\cdot h_K(0)=h_{K_{(\te,\ze)}}(0)$ if $(K,b)$ is descending and that $\te\cdot b=\te\cdot h_K(1)=h_{K_{(\te,\ze)}}(1)$ if $(K,b)$ is descending.

In this way we have defined an action of $\STS$ on $\MM_I^E$ which moreover preserves descending, ascending and point-like elements. Then the map
defined by
\[
\begin{array}{rccc}
\Ga_{I\downarrow}^E: &\MM_{I\downarrow}^E &\longrightarrow &\MTM\\
 & (K,b) &\longmapsto & (b,e)
\end{array}
\]
is $\STS$-equivariant and the same holds for the analogous maps $\Ga_{I\uparrow}^E,\Ga_{I-}^E$.

We use these maps to prove

\begin{lemm}\label{lemm:omlimf} For $E\in R$, the union of the maps $\Ga_{I\downarrow}^E$, $\Ga_{I\uparrow}^E$, $\Ga_{I-}^E$ is an omega-map of $f^E:D\rightarrow\MTM$ and as a consequence $f^E$ is a pseudocycle.
\end{lemm}

\begin{proof}
Let $(p,q)\in\Om_{f^E}$. If $p\in C_i$ is a fixed point, then $f^E$ maps the whole $D$ to $(p,p)\in\De_{C_i\times C_i}=\im\Ga_{\{i\}-}^E$. Otherwise, $p$ and $q$ are the extremal points of a broken gradient line $(K,b)$, so either $h_K(0)=p$ and $h_K(1)=q$ --in which case $(p,q)\in\im\Ga_{I\downarrow}^E$ for some $I$-- or viceversa --in which case $(p,q)\in\im\Ga_{I\uparrow}^E$ for some $I$.

We have seen that $\Om_{f^E}$ is contained in the union of the images of the smooth (because $E\in R$) manifolds $\MM_{I\downarrow}^E$, $\MM_{I\uparrow}^E$, $\MM_{I-}^E$ all of which have at most dimension $2m$. Since $\dim D=2m+2$, we have that $f^E$ is a pseudocycle.
\end{proof}

\begin{lemm}\label{lemm:omlimg} The omega-limit-set of the map $g:D\rightarrow\CC P^1$ is contained in $\{[1:0],[0:1]\}$. As a consequence the union of the two maps $c^\pm:\{pt\}\rightarrow\CC P^1$ defined by $c^+(pt)=\{[1:0]\}$ and $c^-(pt)=\{[0:1]\}$ is an omega-map for $g$ and $g$ is a pseudocycle.
\end{lemm}

\begin{proof} Let $(p_j,t_j,\al_j)\in D$ with $j\geq 1$ be a sequence without convergent subsequences but such that 
\[
g(p_j,t_j,\al_j)=[1:2^{t_j}\al_j]\mathop{\longrightarrow}\limits_{j}[z:w]\in\CC P^1.
\] 
Suppose that $[z:w]\neq[1:0],[0:1]$. Then $[z:w]=[1:w/z]$ with $w/z\neq 0$ and $2^{t_j}\al_j\rightarrow w/z$. Hence $2^{t_j}=|2^{t_j}\al_j|\rightarrow|w/z|$. This implies that $t_j\rightarrow\log_2|w/z|$. Since both $M$ and $S^1$ are compact, the sequences $p_j$ and $\al_j$ have subsequences converging to some $p\in M$ and $\al\in S^1$ respectively. Then $(p_j,t_j,\al_j)$ has a subsequence converging to $(p,\log_2|w/z|,\al)\in D$, a contradiction.
\end{proof}

\begin{coro}
For $E\in R$, the map $f^E\times g$ is a pseudocycle. Also the maps $f_k^E$, $g_k$ and $f_k^E\times g_k$ are pseudocycles for every $k$.
\end{coro}

\begin{proof}
Proposition \ref{prop:productpseudo} applies to $f^E\times g$ because since $\Om_g$ is contained in the image of $0$-dimensional manifolds, $\Om_{f^E\times g}$ is contained in the image of manifolds of the same dimensions as those that contain $\Om_{f^E}$.

Since $\Om_{f^E}$ is contained in the images of $\STS$-equivariant maps and $\STS$ acts freely on the compact manifold $S^{2k+1}\times S^{2k+1}$, by virtue of Proposition \ref{prop:equivpseudo}, we have that $f_k^E$ is a pseudocycle. The same arguments apply to assert that $g_k$ is a pseudocycle because the maps $c^+,c^-$ are trivially $(\STS)$-equivariant. 

Finally, $f_k^E\times g_k$ is a pseudocycle thanks to Proposition \ref{prop:pseudocombi}, which applies for the same dimensional reasons we used to prove that $f^E\times g$ is a pseudocycle.
\end{proof}

We prove one last result, that will imply independence on the choice of a perturbation:

\begin{lemm}\label{lemm:bordclassindep}
There exists a residual subset $R'\subseteq\PP$ such that the bordism class of $f_k^E\times g_k$ does not depend on the choice of $E\in R'\cap R$.
\end{lemm}

\begin{proof} Let $E\in\PP$. Since the sets $\MM_{I\downarrow}^E$, $\MM_{I\uparrow}^E$, $\MM_{I-}^E$ admit $\STS$-actions we can consider the set
\[
\MM_{I\downarrow,k}^E:=\MM_{I\downarrow}^E\times_{\STS}(S^{2k+1}\times S^{2k+1})
\] 
and similarly we get $\MM_{I\uparrow,k}^E$ and $\MM_{I-,k}^E$.

According to Lemmas \ref{lemm:omlimf}, \ref{lemm:omlimg} above and Proposition \ref{prop:pseudocombi}, the omega-limit set of $f_k^E\times g_k$ is contained in the union of the images of the maps
\[
\begin{array}{rccc}
\Ga_{I\downarrow,k}^{E+}:&\MM_{I\downarrow,k}^E&\longrightarrow &(\MTM)_k\times\CC P^1_k\\
 &\left[x,s\right]&\longmapsto &\left(\left[\Ga_{I\downarrow}^E(x),s\right],[[1:0],s]\right)
\end{array}
\]
and the images of the analogous maps $\Ga_{I\downarrow,k}^{E-}$, $\Ga_{I\uparrow,k}^{E+}$, $\Ga_{I\uparrow,k}^{E-}$, $\Ga_{I-,k}^{E+}$, $\Ga_{I-,k}^{E-}$. Here $s$ denotes an element of $S^{2k+1}\times S^{2k+1}$ and the superscript $\pm$ denotes if we are taking either $[1:0]$ or $[0:1]$ in the second component, following the notation $c^\pm$ of Lemma \ref{lemm:omlimg}.

Consider now the manifold
\[
\MM_{I\downarrow,k}=\{([x,s],E):\ E\in\PP,\ [x,s]\in\MM_{I\downarrow,k}^E\}
\]
and the map
\[
\begin{array}{rccc}
\Ga_{I\downarrow,k}^\pm:&\MM_{I\downarrow,k}&\longrightarrow &(\MTM)_k\times\CC P^1_k\\
 &(\left[x,s\right],E)&\longmapsto &\Ga_{I\downarrow,k}^{E\pm}(\left[x,s\right])
\end{array}.
\]
In an analogous way, define manifolds $\MM_{I\uparrow,k},\MM_{I-,k}$ and maps $\Ga_{I\uparrow,k}^\pm,\Ga_{I-,k}^\pm$. All these manifolds project on $\PP$. Not surprisingly we denote these projections by $\pi_{I\downarrow,k},\pi_{I\uparrow,k},\pi_{I-,k}$. We also take $\pi_k:D_k\times\PP\rightarrow\PP$ to be the projection to the second factor. By Sard's theorem there exists a residual subset $R'\subseteq\PP$ of regular values of all these projections. Given $E_0,E_1\in R$ take a path $\psi:[0,1]\rightarrow\PP$ transverse to all the projections and such that $\psi(0)=E_0$ and $\psi(1)=E_1$. We claim that the map
\[
\begin{array}{rccc}
B:&CS(\pi_k,\psi)&\longrightarrow &(\MTM)_k\times\CC P^1_k\\
 &([x,s],\psi(t),t)&\longmapsto &(f^{\psi(t)}_k\times g_k)([x,s])
\end{array}
\]
is a bordism between $f_k^{E_0}\times g_k$ and $f^{E_1}_k\times g_k$: under the diffeomorphism
\[
\begin{array}{ccc}
CS(\pi_k,\psi) &\longrightarrow & D_k\times[0,1]\\
([x,s],\psi(t),t) &\longmapsto & ([x,s],t)
\end{array}
\] 
the map $B$ is identified with the map $([x,s],t)\mapsto (f^{\psi(t)}_k\times g_k)([x,s])$, which when applied to $([x,s],0)$ gives $(f^{E_0}_k\times g_k)([x,s])$ and when applied to $([x,s],1)$ gives $(f^{E_1}_k\times g_k)([x,s])$.

Finally $\Om_B$ is contained in the union of the images of the maps
\[
\begin{array}{ccc}
CS(\pi_{I\downarrow,k},\psi)&\longrightarrow &(\MTM)_k\times\CC P^1_k\\
([x,s],\psi(t),t)&\longmapsto &\Ga_{I\downarrow,k}^{\psi(t)\pm}(x)
\end{array}
\]
and the images of the analogous maps defined on $CS(\pi_{I\uparrow,k},\psi)$ and $CS(\pi_{I-,k},\psi)$.
\end{proof}

Consider, for $E\in R\cap R'$, the diagram
\[
\begin{diagram}
\node{}\node{D_k}\arrow{sw,t}{f^E_k}\arrow{s,r}{f^E_k\times g_k}\arrow{se,t}{g_k}\node{}\\
\node{(\MTM)_k}\node{(\MTM)_k\times\CC P^1_k}\arrow{w,b}{\pi_{(\MTM)_k}}\arrow{e,b}{\pi_{\CC P^1_k}}\node{\CC P^1_k}
\end{diagram}
\]
where $\pi_{(\MTM)_k}$ and $\pi_{\CC P^1_k}$ are the projections. Since $f^E_k\times g_k$ is a pseudocycle, according to Proposition \ref{prop:exerciseeight} the map
\[
\lm_k:H^*(\CC P^1_k;\ZZ)\longrightarrow H^{*+2m-2}((\MTM)_k;\ZZ)
\]
defined by
\[
\lm_k(\al):=PD^{-1}_{(\MTM)_k}(\pi_{(\MTM)_k})_*(\pi_{\CC P^1_k}^*\al\frown\Phi_*^{(\MTM)_k\times\CC P^1_k}[f_k^E\times g_k])
\]
provides a substitute for $(f^E_k)^!g_k^*$. This map does not depend on the choice of $E$ thanks to Lemma \ref{lemm:bordclassindep}. Finally using the stabilisation of equivariant cohomology --Lemma \ref{lemm:stabilisation}-- on the maps $\lm_k$ we get a map in equivariant cohomology
\[
\lm:H_{\STS}^*(\CC P^1;\ZZ)\longrightarrow H_{\STS}^{*+2m-2}(\MTM;\ZZ)
\]
that we call the {\it lambda-map}.

\section{The lambda-map for general actions}\label{sec:lambdageneral}

In this section we drop the condition that the action of $S^1$ on $M$ is semi-free so the right framework to work with is that of Chapter \ref{ch:multivalued}: there we constructed a space $\PP$ of multivalued perturbations of $J$ and for $P\in\PP$ we defined the space $\GG^P$ of oriented $P$-perturbed gradient lines (Definition \ref{defi:pbgd}). The rest of Chapter \ref{ch:multivalued} was devoted to proving that, for a generic choice of $P\in\PP$, $\GG^P$ admits a stratification parametrised by the critical set of $\mu$ and that the strata are branched manifolds 

We will follow the same steps as in Section \ref{sec:lambdasimple}, where the role played by $D=M\times\RR\times S^1$ will now be taken by $\GG_\emptyset^P\times S^1$. Since this is not a manifold, but a branched manifold, we cannot apply the theory of pseudocycles directly and we will need to be careful when doing so. In particular the lambda-map will be defined in equivariant cohomology with rational coefficients.

Let $\STS$ act on $\GG_\emptyset^P\times S^1$ via
\[
(\te,\ze)(G,\al)=(\te\cdot G,\ze\al\bar{\te}),
\]
where the $S^1$-action on $\GG^P_\emptyset$ is the one described in Remark \ref{rema:actionpbgl}.

Remember that if $G\in\GG^P_\emptyset$ is not point-like, it is parametrised by a unique gradient segment with first component $\ga:[a_0,a_1]\rightarrow M$ such that $h_K(0)=\ga(a_0)$ and $h_K(1)=\ga(a_1)$ (see Remark \ref{rema:singleparameter}). In these terms define the maps
\[
\begin{array}{rccc}
f^P:&\GG_\emptyset^P\times S^1&\longrightarrow&\MTM\\
 &(G,\al)&\longmapsto& (b,\al\cdot e)
\end{array},
\]
where $e$ is the end of $G$ and
\[
\begin{array}{rccc}
g^P:&\GG_\emptyset^P\times S^1&\longrightarrow&\CC P^1\\
 &(G,\al)&\longmapsto&[1:2^t\al]
\end{array},
\]
where $t=0$ if $G$ is point-like, $t=a_1-a_0$ if $G$ is descending and $t=a_0-a_1$ if $G$ is ascending.

\begin{lemm}\label{lemm:fgareequivgeneral} 
Both $f^P$ and $g^P$ are $\STS$-equivariant maps.
\end{lemm}

\begin{proof} If $b,e$ are the beginning and end of $G$, then the beginning and end of $\te\cdot G$ are $\te\cdot b,\te\cdot e$. The equivariance of $f^P$ follows:
\[
\begin{array}{rcl}
f^P(\te\cdot G,\ze\al\bar{\te}) & = & (\te\cdot b,\ze\al\bar{\te}\te\cdot e)\\
 & = & (\te,\ze)(b,\al\cdot e)\\
 & = & (\te,\ze)f^P(G,\al).
\end{array}
\]

Note that if $G$ is not point-like then $\te\cdot\ga$ parametrises $\te\cdot G$, so in any case we have
\[
\begin{array}{rcl}
g^P(\te\cdot G,\ze\al\bar{\te}) & = & [1:2^t\ze\al\bar{\te}]\\
 & = & [\te:\ze2^t\al]\\
 & = & (\te,\ze)[1:2^t\al]\\
 & = & (\te,\ze)g^P(G,\al).
\end{array}
\]
\end{proof}

Let $(\GG_\emptyset^P\times S^1)_k:=(\GG_\emptyset^P\times S^1)\times_{\STS} S^{2k+1}\times S^{2k+1}$. From the lemma it follows that we have well-defined maps
\[
\begin{diagram}
\node{}\node{(\GG^P_{\emptyset}\times S^1)_k}\arrow{sw,t}{g_k^P}\arrow{se,t}{f_k^P}\node{}\\
\node{\CC P^1_k}\node{}\node{(\MTM)_k}
\end{diagram}.
\]

If we keep following the pattern of Section \ref{sec:lambdasimple} we see that our next goal is to define an $\STS$ action on $\GG^P\setminus\GG^P_{\emptyset}$, which consists of those perturbed gradient lines with a least one breaking point. Let $I=\{i_1,\ldots,i_r\}\neq\emptyset$ with $\mu(C_{i_1})>\cdots>\mu(C_{i_r})$. If $G\in\GG^P_{I-}$ it must be the case that $G=(\{b\};\emptyset,\ldots,\emptyset;b)$ for some fixed point $b$, so we will take the action to be trivial on point-like elements. Otherwise let $G=(K;K_1,\ldots,K_\ell;b)$ and for $(\te,\ze)\in\STS$ define $K_{(\te,\ze)}$ as in the Section \ref{sec:lambdasimple}. If $G$ is descending we let
\[
K_{j,(\te,\ze)}=\left\{
\begin{array}{ll}
\te\cdot K_j &\mbox{if}\ \inf\mu(R_j)>\mu(C_{i_1})\\
\ze\cdot K_j &\mbox{if}\ \sup\mu(R_j)<\mu(C_{i_r})\\
K_j &\mbox{otherwise}
\end{array}
\right..
\]
while if $G$ is ascending we make the same definition interchanging $\te$ and $\ze$. Then we have that $(K_{(\te,\ze)};K_{1,(\te,\ze)},\ldots,K_{\ell,(\te,\ze)};\te\cdot b)$ is also an element of $\GG_I^P$. We have thus defined an $\STS$ action on $\GG_I^P$ that preserves point-like, ascending and descending elements that makes the map
\[
\begin{array}{rccc}
\Ga_{I\downarrow}^P: &\GG_{I\downarrow}^P &\longrightarrow &\MTM\\
 & (K;\ldots;b) &\longmapsto & (b,e)
\end{array}
\]
and the analogous $\Ga_{I\uparrow}^P,\Ga_{I-}^P$ into $\STS$-equivariant maps.

The two lemmas that follow study the omega-limit sets of $f^P$ and $g^P$ and are analogous to Lemmas \ref{lemm:omlimf} and \ref{lemm:omlimg}:

\begin{lemm}\label{lemm:omlimfp}
The omega-limit set of $f^P:\GG^P_\emptyset\times S^1\rightarrow\MTM$ is covered by the images of the maps $\Ga_{I\downarrow}^P$, $\Ga_{I\uparrow}^P$ and $\Ga_{I-}^P$.
\end{lemm}

\begin{proof} 
The proof is totally analogous to that of Lemma \ref{lemm:omlimf}.
\end{proof}

\begin{lemm}\label{lemm:omlimgp}
The omega-limit set of $g^P:\GG^P_{\emptyset}\times S^1\rightarrow\CC P^1$ is contained in $\{[1:0],[0:1]\}$.
\end{lemm}

\begin{proof} Let $[z:w]\in\CC P^1$. Then, there exists $(G_j,\al_j)\in \GG_\emptyset^P\times S^1$, a sequence with no convergent subsequences such that $g^P(G_j,\al_j)\mathop{\rightarrow}\limits_{j}[z:w]$. The only way such a sequence cannot converge is that $G_j$ is not point-like for an infinite number of $j$ and that, for these terms, the parametrising curve $[a_{0j},a_{1j}]\rightarrow M$ is such that $\lim_ja_{1j}-a_{0j}=+\infty$. Then $\lim_j[1:2^{a_{1j}-a_{0j}}\al_j]=[0:1]$ and $\lim_j[1:2^{a_{0j}-a_{1j}}\al_j]=[1:0]$.
\end{proof}

At this point we cannot follow the steps of Section \ref{sec:lambdasimple} anymore. The reason is that although we have control over the omega-limit sets of the maps $f^P,g^P$ we cannot say that they are pseudocycles because they are defined on branched manifolds and we have not developed a theory of pseudocycles in that case. All the necessary ingredients to define the lambda-map are provided by the following theorem:

\begin{theo}\label{theo:bigone} Let $\ell:W\rightarrow(\MTM)_k$ be a $(q+2m-2)$-pseudocycle and let $h:V\rightarrow \CC P^1_k$ be a $(4k+2-q)$-pseudocycle.

Take also $D\subseteq C^\infty(T\CC P^1_k)$ and $E\subseteq C^\infty(T(\MTM)_k)$ as in Lemma \ref{lemm:surjev} and let $\DC=\exp D$ and $\EE=\exp E$. For $\eta\in\DC$, $\ga\in\EE$ consider the diagram of Cartesian squares
\[
\begin{diagram}
\node{CS(f_k^P\circ\pi_{\eta\circ h}^P,\ga\circ\ell)}\arrow{s}\arrow[2]{e}\node[2]{W}\arrow{s,r}{\ga\circ\ell}\\
\node{CS(g_k^P,\eta\circ h)}\arrow{e,t}{\pi^P_{\eta\circ h}}\arrow{s}\node{(\GG^P_\emptyset\times S^1)_k}\arrow{e,t}{f^P_k}\arrow{s,r}{g^P_k}\node{(\MTM)_k}\\
\node{V}\arrow{e,b}{\eta \circ h}\node{\CC P^1_k}
\end{diagram}.
\]
There exists a residual subset $R\subseteq \PP\times\DC\times\EE$ such that for every $(P,\eta,\ga)\in R$

\begin{enumerate}[(1)]
\item $CS(g_k^P,\eta\circ h)$ is a branched manifold of dimension $2m+4k+2-q$. 

\item $\sres:=CS(f_k^P\circ\pi^P_{\eta\circ h},\ga\circ\ell)$ is a finite set.

\item For any $x=([(G,\al),s],v,w)\in\sres$, $G\in\GG^P_\emptyset$ is not a branching point. Then we let $w(x)$ be the weight of $G$ (Definition \ref{defi:weight}). We also assign a sign to $x$ as follows: if $y:=(f_k^P\circ\pi_{\eta\circ h}^P)([(G,\al),s],v)=(\ga\circ\ell)(w)$, the differential at $x$ of the map
\[
(f_k^P\circ\pi_{\eta\circ h}^P)\times(\ga\circ\ell):CS(g_k^P,\eta\circ h)\times W\longrightarrow (\MTM)_k\times(\MTM)_k
\]
induces an isomorphism between $T_x(CS(g_k^P,\eta\circ h)\times W)$ and its image $T_{(y,y)}\De_{(\MTM)_k\times(\MTM)_k}$. We define $\sig(x)=\pm 1$ depending on whether this isomorphism preserves ($+1$) or reverses ($-1$) orientations.

\item The linear map
\[
\begin{array}{rccc}
\LL_{[h],k}:&\BB_{q+2m-2}((\MTM)_k)&\longrightarrow&\QQ\\
      &[\ell]&\longmapsto& \sum_{x\in\sres}\sig(x)w(x)
\end{array}
\] 
is well defined on bordism classes of pseudocycles, i.e. it does not depend on the choice of $\ell$. Moreover it only depends on the bordism class $[h]$, but not on the particular choice of $h$ (hence the notation). It neither depends on the choices of $P,\eta,\ga,D,E$.
\end{enumerate}
\end{theo}

\begin{proof}
The proof of this theorem relies on strong transversality techniques (Section \ref{sec:intersectionpairing}) but applied to branched manifolds instead of usual manifolds. We will also use train tracks (Section \ref{subsec:traintracks}) to prove most of point (4):

Applying Lemma \ref{lemm:generictransverse} to $g_k^P$ and $h$ we deduce that for a generic $\eta\in\DC$, the maps $g_k^P$ and $\eta\circ h$ are transverse. Then $CS(g_k^P,\eta\circ h)$ is a branched manifold and its dimension is 
\[
\begin{array}{l}
\dim(\GG_\emptyset^P\times S^1)_k+\dim V-\dim\CC P^1_k=\\
\hspace{3cm}=(2m+4k+2)+(4k+2-q)-(4k+2)\\
\hspace{3cm}=2m+4k+2-q
\end{array}.
\]
This proves (1).

To prove (2) we apply \ref{lemm:genericstronglytransverse} to the maps $f_k^P\circ\pi_{\eta\circ h}^P$ and $\ell$ to assert that for a generic $\ga\in\EE$, the maps $f_k^P\circ\pi_{\eta\circ h}^P$ and $\ga\circ\ell$ are strongly transverse. Note that to do so we need some dimensional relations to hold: from Lemmas \ref{lemm:omlimfp}, \ref{lemm:omlimgp} and equivariance we can construct omega-maps for $f_k^P\circ\pi_{\eta\circ h}^P$ of codimension 2 and since $\ell$ is a pseudocycle it also admits an omega-map of codimension 2. Now, thanks to Lemma \ref{lemm:stcsfinite}, $CS(f_k^P\circ\pi^P_\eta,\ga\circ\ell)$ is a compact smooth branched manifold of dimension
\[
\begin{array}{l}
\dim CS(g_k^P,\eta\circ h)+\dim W-\dim (\MTM)_k=\\
\hspace{3cm}=(2m+4k+2-q)+(q+2m-2)-(4m+4k)=0.
\end{array}
\] 
Being compact and $0$-dimensional means that it is a finite set.

The claim in (3) follows from the fact that the set of branching points $(\GG_\emptyset^P)^\prec$ has zero measure.

Now we shall see that $\LL_{[h],k}$ does not depend on the choice of a representative in $[\ell]$. The proof of this fact is very similar to that of Proposition \ref{prop:intersectionproduct}: let $(\ell',W')$ be a pseudocycle bordant with $\ell$ and let $(\tilde{\ell},\widetilde{W})$ be a bordism between $\ell$ and $\ell'$. By Lemma \ref{lemm:genericstronglytransverse} $f_k^P\circ\pi_\eta^P$ and $\ga\circ\tilde{\ell}$ are strongly transverse maps for a generic choice of $\ga\in\EE$. Then $CS(f_k^P\circ\pi_\eta^P,\ga\circ\tilde{\ell})=:T$ is a compact branched manifold of dimension 1. Since it is oriented, $T$ is a train track. Let $S:=\sres$ and $S':=\mathcal{S}_{(h,\ell',P,\eta,\ga)}$. Then $\partial T=S'-S$ because $\partial\widetilde{W}=W'-W$. Reversing the orientation of $T$ if necessary we get that
\[
\partial T^+=\{x\in S:\ \sig(x)=1\}\cup\{x\in S':\ \sig(x)=-1\},
\]
\[
\partial T^-=\{x\in S:\ \sig(x)=-1\}\cup\{x\in S':\ \sig(x)=1\}.
\]
Finally note that $T$ carries a weighting (see Definition \ref{defi:weighting}) induced by the weights on $\GG_\emptyset^P$ as in (3). Applying Proposition \ref{prop:ttboundary} we deduce that
\[
\sum_{x\in S\atop\sig(x)=1}w(x)+\sum_{x\in S'\atop\sig(x)=-1}w(x)=\sum_{x\in S\atop\sig(x)=-1}w(x)+\sum_{x\in S'\atop\sig(x)=1}w(x),
\]
which collecting in terms of $S$ and $S'$ gives
\[
\sum_{x\in S\atop\sig(x)=1}w(x)-\sum_{x\in S\atop\sig(x)=-1}w(x)=\sum_{x\in S'\atop\sig(x)=1}w(x)-\sum_{x\in S'\atop\sig(x)=-1}w(x).
\]
From this we conclude that
\[
\sum_{x\in S}\sig(x)w(x)=\sum_{x\in S'}\sig(x)w(x),
\]
which is what we wanted to prove.

To show independence from the choice of a rsepresentative in $[h]$ we do something very similar: if $h'$ is a pseudocycle and $\tilde{h}$ is a cobordism between $h$ and $h'$, for a generic $\eta\in\DC$ we have that $CS(g_k^P,\eta\circ\tilde{h})$ is an oriented compact branched manifold and that $CS(f_k^P\circ\pi_{\eta\circ\tilde{h}}^P,\ga\circ\ell)$ is a train track with boundary ${\cal S}_{(h',\ell,P,\eta,\ga)}-{\cal S}_{(h,\ell,P,\eta,\ga)}$. Applying the same arguments as above to this train track we get the desired result.

We also use train tracks to prove independence from $\eta$, $\ga$ and $P$:

Let $\ga_0,\ga_1\in\EE$ be generic and satisfy (2). Take a path $\psi:[0,1]\rightarrow\EE$ such that $\psi(0)=\ga_0,\psi(1)=\ga_1$ and such that $f_k^P\circ\pi_{\eta\circ h}^P$ and
\[
\begin{array}{rccc}
L: & W\times[0,1]&\longrightarrow &(\MTM)_k\\
   & (w,t)&\longmapsto &(\psi(t)\circ\ell)(w)
\end{array}
\]
are strongly transverse maps: we can do so because the dimensions of the omega-limit sets of $L$ and $f_k^P\circ\pi^P_{\eta\circ h}$ are controlled (in the fashion of Lemma \ref{lemm:genericstronglytransverse}), thanks to $\ell$ being a pseudocycle in the former and from Lemmas \ref{lemm:omlimfp}, \ref{lemm:omlimgp} in the latter. Then $CS(f_k^P\circ\pi_{\eta\circ h},L)$ is a train track with boundary $\mathcal{S}_{(h,\ell,P,\eta,\ga_1)}-\mathcal{S}_{(h,\ell,P,\eta,\ga_0)}$. Using the same arguments on train tracks as above we get independence of $\ga\in\EE$.

Independence of $\eta$ is proved very similarly: let $\eta_0,\eta_1\in\DC$ be generic and satisfy (1). Take a path $\psi:[0,1]\rightarrow\DC$ such that $\psi(0)=\eta_0,\psi(1)=\eta_1$ making $g_k^P$ and
\[
\begin{array}{rccc}
H: & V\times[0,1]&\longrightarrow &\CC P^1_k\\
   & (v,t)&\longmapsto &(\psi(t)\circ h)(w)
\end{array}
\]
into transverse maps. Then $CS(g_k^P,H)$ is a smooth branched manifold and the projection $\pi^P_H:CS(g_k^P,H)\rightarrow(\GG_\emptyset^P\times S^1)_k$ is smooth. If we take $\psi$ so that $f_k^P\circ\pi^P_H$ and $\ga\circ\ell$ are strongly transverse maps (which we can, controlling dimensions of omega-limit sets as above), then $CS(f_k^P\circ\pi_H^P,\ga\circ\ell)$ is a train track with boundary $\mathcal{S}_{(h,\ell,P,\eta_1,\ga)}-\mathcal{S}_{(h,\ell,P,\eta_1,\ga)}$. This shows independence of $\eta$.

Finally, once again independence of $P$ follows in a very similar way. Take $P_0,P_1\in\PP$ and let $\psi:[0,1]\rightarrow\PP$ be a a path such $\psi(0)=P_0$ and $\psi(1)=P_1$. Define 
\[
\GG=\{(x,t):t\in[0,1],x\in(\GG_\emptyset^{\psi(t)}\times S^1)_k\}
\]
and consider the maps
\[
\begin{array}{rccc}
g: & \GG&\longrightarrow &\CC P^1_k\\
   & (x,t)&\longmapsto &g_k^{\psi(t)}(x)
\end{array},
\begin{array}{rccc}
f: & \GG&\longrightarrow &(\MTM)_k\\
   & (x,t)&\longmapsto &f_k^{\psi(t)}(x)
\end{array}
\]
Taking $\psi$ properly we get that $CS(g,\eta\circ h)$ is a smooth branched manifold, and we have a projection $\pi:CS(g,\eta\circ h)\rightarrow\GG$, and also we have that $CS(f\circ\pi,\ga\circ\ell)$ is a train track with boundary $\mathcal{S}_{(h,\ell,P_1,\eta,\ga)}-\mathcal{S}_{(h,\ell,P_0,\eta,\ga)}$. This shows independence of $P$.

The only remaining thing to do is to prove that $\LL_{[h],k}$ is independent of the choice of $D$ and $E$: if $D'$ also satisfies the conditions of Lemma \ref{lemm:surjev}, then so does $D'':=D+D'$. Applying the result to $\DC''=\exp(D'')$ instead of $\DC$ we see that $L_{[h],k}$ does not depend on the choice of elements within (a residual subset of) $\DC''$. Since $\DC''$ contains both $\DC$ and $\DC':=\exp D'$, we can assert that $L_{[h],k}$ does not depend of any choice within these two sets. The same applies to $E$.
\end{proof}

In this theorem, given a bordism class $[h]\in\BB_{4k+2-q}(\CC P^1_k)$ we have constructed a linear map $\LL_{[h],k}:\BB_{q+2m-2}((\MTM)_k)\rightarrow\QQ$. Recall from Section \ref{sec:pseudocycles} on pseudocycles the equivalences 
\[
\begin{array}{rl}
\Psi_*^{(\MTM)_k}:& H_*((\MTM)_k;\ZZ)\longrightarrow\BB_*((\MTM)_k)\\
\Phi_*^{\CC P^1_k}:&\BB_*(\CC P^1_k)\longrightarrow H_*(\CC P^1_k;\ZZ)
\end{array}.
\]

Then $\LL_{[h],k}\circ\Psi_*^{(\MTM)_k}: H_{q+2m-2}((\MTM)_k;\ZZ)\rightarrow\QQ$ is linear and we can view it as
\[
\LL_{[h],k}\circ\Psi_*^{(\MTM)_k}\in H^{q+2m-2}((\MTM)_k;\QQ).
\] 
Also $\Phi_*^{\CC P^1_k}[h]\in H_{4k+2-q}(\CC P^1_k;\ZZ)$ and via intersection pairing of pseudocycles we get a cohomology class 
\[
I_{\Phi_*^{\CC P^1_k}[h]}\in H^q(\CC P^1_k;\QQ).
\]
Since the elements of the form $I_{\Phi_*^{\CC P^1_k}(\be)}$ with $\be\in\BB_*(\CC P^1_k)$ generate $H^*(\CC P^1_k;\QQ)$ we get a correspondence
\[
\begin{array}{rccc}
\lm_k:H^*(\CC P^1_k;\QQ) & \longrightarrow & H^{*+2m-2}((\MTM)_k;\QQ)\\
I_{\Phi_*^{\CC P^1_k}(\be)} &\longmapsto & \LL_{\be,k}\circ\Psi_*^{(\MTM)_k}
\end{array},
\]
Finally, by means of stabilisation in equivariant cohomology --Lemma \ref{lemm:stabilisation}-- we get a map 
\[
\lm:H^*_{\STS}(\CC P^1;\QQ)\longrightarrow H^{*+2m-2}_{\STS}(\MTM;\QQ),
\]
which we call the {\it lambda-map}.

\begin{rema} Before going on we shall justify why we used different techniques to define the lambda-map for the semi-free and for the general case: for the semi-free case we could have used ideas related to the intersection pairing as well, but that would have made us lose information contained in the torsion. Moreover, being able to use the correspondence between bordisms of pseudocycles and integral homology we were able to make a more direct construction. Trying this more direct construction in the general case would be possible if there was a theory relating bordisms of pseudocycles defined on branched manifolds and rational homology. Such a result may be achieved if we define branched manifolds always with a weighting as in \cite{Sal} and try to make an analogous construction to that in \cite{Zin} for such more general pseudocycles. However, this is far from the scope of our work.
\end{rema}


\section{The Kirwan map}

From now on will consider a general action of $S^1$ (i.e. not necessarily semi-free) and all cohomologies will be taken with {\bf rational coefficients}.

Let $c\in\RR$ be a regular value of the moment map $\mu$. The $S^1$-action on $\mu^{-1}(c)$ has finite stabilisers and by the Cartan isomorphism (Proposition \ref{prop:cartaniso}) we have that $H^*_{S^1}(\mu^{-1}(c))\simeq H^*(M\sq_cS^1)$, where $M\sq_cS^1:=\mu^{-1}(c)/S^1$ is a common notation for the symplectic quotient. The inclusion $\mu^{-1}(c)\hookrightarrow M$ induces a map $H_{S^1}^*(M)\rightarrow H_{S^1}^*(\mu^{-1}(c))$, which composed with the Cartan isomorphism gives rise to the {\it Kirwan map}
\[
\ka_c:H^*_{S^1}(M)\longrightarrow H^*(M\sq_cS^1).
\]

This construction is valid for any Hamiltonian action of an abelian group. In particular for the $\STS$ action on $\MTM$ defined by $(\te,\ze)(p,q)=(\te\cdot p,\ze\cdot q)$, whose moment map is $(p,q)\mapsto(\mu(p),\mu(q))$. If $c\in\RR$ is a regular value of $\mu$, then $(c,c)$ is a regular value of this action and there is a natural identification $(\MTM)\sq_{(c,c)}(\STS)\simeq M\sq_cS^1\times M\sq_cS^1$. The corresponding Kirwan map is of the form
\[
\ka_{(c,c)}:H^*_{\STS}(\MTM)\longrightarrow H^*(M\sq_cS^1\times M\sq_cS^1).
\]

\begin{defi}
Let $c$ be a regular value of $\mu$. A cohomology class $\be\in H^{2m-2}_{\STS}(\MTM)$ is a {\it biinvariant diagonal class} for $c$ if $\ka_{(c,c)}(\be)$ is the Poincaré dual of $[\De_{M\ssq_cS^1\times M\ssq_cS^1}]$. We say that $\be$ is a {\it global} biinvariant diagonal class if it is a biinvariant diagonal class for all regular values of $\mu$.
\end{defi}

\begin{rema}
Since the action of $S^1$ on $\mu^{-1}(c)$ may contain non-trivial finite stabilisers, $M\sq_c S^1$ is not a smooth manifold in general, but only an orbifold. Nevertheless, Poincaré duality still holds for orbifolds. The details can be found in \cite{Sat} (note that Satake refers to orbifolds as ``V-manifolds'').
\end{rema}

Using the isomorphism $H_*(-;\QQ)\simeq H^*(-;\QQ)^\vee$ we can view Poincaré duality as taking values in the dual of the cohomology, i.e. 
\[
PD_{M\ssq_cS^1}:H^*(M\sq_cS^1)\longrightarrow H^{2m-2-*}(M\sq_cS^1)^\vee.
\]
Using the natural identification
\[
(\MTM)\times_{\STS}B(\STS)\simeq (M\times_{S^1}BS^1)\times (M\times_{S^1}BS^1)
\] 
and Künneth we have
\[
H_{\STS}^{2m-2}(\MTM)\simeq\bigoplus_{p=0}^{2m-2}H_{S^1}^{2m-2-p}(M)\otimes H_{S^1}^p(M).
\]
Consider the map
\[
\bigoplus_{p=0}^{2m-2}H_{S^1}^{2m-2-p}(M)\otimes H_{S^1}^p(M)\longrightarrow\bigoplus_{p=0}^{2m-2}H^p(M\sq_cS^1)^\vee\otimes H_{S^1}^p(M),
\] 
defined by $(PD_{M\ssq_cS^1}\circ\ka_c)\ot\mbox{id}$. In this way for any regular value $c$ of the moment map $\mu$ and for any degree $2m-2$ equivariant cohomology class $\be$ of $\MTM$ we get a degree preserving linear map $l_c^\be:H^*(M\sq_cS^1)\longrightarrow H_{S^1}^*(M)$.

\begin{lemm}\label{lemm:rightinverse}
If $\be$ is a biinvariant diagonal class for $c$, then $l_c^\be$ is a right inverse of $\ka_c$.
\end{lemm}

\begin{proof}
Let $\be=\sum\ep_i\ot\eta_i\in H_{S^1}^*(M)\ot H_{S^1}^*(M)$ be any decomposition. Then we get $\ka_{(c,c)}(\be)=\sum\ka_c(\ep_i)\otimes\ka_c(\eta_i)$. Since $\ka_{(c,c)}(\be)$ is the Poincaré dual of $[\De_{M\ssq_cS^1\times M\ssq_cS^1}]$, for any $a\in H^*(M\sq_cS^1)$ we have
\[
\sum\left(\int_{M\ssq_cS^1}a\smile\ka_c(\ep_i)\right)\ka_c(\eta_i)=a.
\]
On the other hand
\[
l_c^\be(a)=\sum\left(\int_{M\ssq_cS^1}a\smile\ka_c(\ep_i)\right)\eta_i,
\] 
and applying $\ka_c$ we get 
\[
(\ka_c\circ l_c^\be)(a)=\sum\left(\int_{M\ssq_cS^1}a\smile\ka_c(\ep_i)\right)\ka_c(\eta_i).
\] 
Therefore $(\ka_c\circ l_c^\be)(a)=a$.
\end{proof}

We have defined biinvariant diagonal classes and we have seen that they provide right inverses of the Kirwan map. However we still do not know if biinvariant diagonal classes exist. The lambda-map now comes into play to prove that they do exist:

\begin{theo}\label{theo:lambdagivesbdc}
$\lm(1)$ is a global biinvariant diagonal class.
\end{theo}

\begin{proof}
To simplify notation we will write $\De:=\De_{M\sq_cS^1\times M\sq_cS^1}$ and we will omit subscripts and superscripts for $PD$, $\Phi_*$ and $\Psi_*$ which should be clear from the context.

Take $k$ large enough so that $\lm(1)=\lm_k(1)$ and so that we have isomorphisms
\[
H^{2m-2}_{\STS}(\MTM)\simeq H^{2m-2}((\MTM)_k),
\]
\[
H^{2m-2}_{\STS}(\mu^{-1}(c)\times\mu^{-1}(c))\simeq H^{2m-2}((\mu^{-1}(c)\times\mu^{-1}(c))_k).
\]

Consider the maps
\[
\begin{diagram}
\node{(\mu^{-1}(c)\times\mu^{-1}(c))_k}\arrow{e,t,J}{\io_k}\arrow{s,l}{v_k}\node{(\MTM)_k}\\
\node{M\sq_cS^1\times M\sq_cS^1}
\end{diagram}
\]
where the vertical arrow, defined by $[(p,q),s]\mapsto([p],[q])$, induces the Cartan isomorphism in degree $2m-2$ cohomology. Then 
\[
\ka_{(c,c)}(\lm(1))=((v_k^*)^{-1}\circ\io_k^*\circ\lm_k)(1).
\]
and therefore the result is equivalent to proving that
\[
(v_k)^*PD^{-1}[\De_{M\ssq_cS^1\times M\ssq_cS^1}]=(\io_k^*\circ\lm_k)(1).
\]

Let us first compute the right hand side of this equation: the cohomology class $\lm_k(1)$ can be seen as a morphism
\[
\vp_k:H_{2m-2}((\MTM)_k;\ZZ)\longrightarrow\QQ
\] 
and then the cohomology class $(\io_k^*\circ\lm_k)(1)$ is identified with the map $\vp_k\circ(\io_k)_*$. Given $a\in H_{2m-2}((\mu^{-1}(c)\times\mu^{-1}(c))_k;\ZZ)$ we shall compute $(\vp_k\circ(\io_k)_*)a$.

If $id$ denotes the identity on $\CC P^1_k$, then its bordism class is identified with the fundamental class of $\CC P^1_k$: $\Phi_*[id]=[\CC P^1_k]$. Then $I_{\Phi_*[id]}=PD^{-1}[\CC P^1_k]=1$ and from the definition of the lambda-map it follows that $\vp_k=\LL_{[id],k}\circ\Psi_*$.

Let us compute $\LL_{[id],k}$ using Theorem \ref{theo:bigone}: $CS(g_k^P,id)$ is the graph of $g_k^P$ and therefore the projection $\pi_{id}^P:CS(g_k^P,id)\rightarrow(G_\emptyset^P\times S^1)_k$ is a diffeomorphism, so for a $(2m-2)$-pseudocycle $\ell$ of $(\MTM)_k$ we can identify $CS(f_k^P\circ\pi_{id}^P,\ell)$ with $CS(f_k^P,\ell)$. Hence $\LL_{[id],k}([\ell])$ is obtained by counting points of $CS(f_k^P,\ell)$ with signs and weights. In other words, $\LL_{[id],k}([\ell])=\#_wCS(f_k^P,\ell)$ where $\#_w$ denotes counting with weights.

If $h_a$ is a $(2m-2)$-pseudocycle of $(\mu^{-1}(c)\times\mu^{-1}(c))_k$ such that $\Phi_*[h_a]=a$, then
\[
\begin{array}{rcl}
(\vp_k\circ(\io_k)_*)a & = & (\vp_k\circ(\io_k)_*)\Phi_*[h_a]\\
 & \stackrel{(1)}{=} & \vp_k(\Phi_*[\io_k\circ h_a])\\
 & \stackrel{(2)}{=} & (\LL_{[id],k}\circ\Psi_*\circ\Phi_*)[\io_k\circ h_a]\\
 &                =  & \LL_{[id],k}[\io_k\circ h_a]\\
 & \stackrel{(3)}{=} & \#_wCS(f_k^P,\io_k\circ h_a)
\end{array}
\]
where $(1)$ is given by the fact that $\io_k$ is a pseudocycle of $(\MTM)_k$ and the functoriality of $\Phi_*$, $(2)$ is the expression we found out for $\vp_k$ and $(3)$ follows from the computation in the last paragraph.

In conclusion, the cohomology class $(\io_k)^*\lm_k(1)$ is identified with the map
\[
\begin{array}{rccc}
\vp_k\circ(\io_k)_*:& H_{2m-2}((\mu^{-1}(c)\times\mu^{-1}(c))_k;\ZZ)&\longrightarrow &\QQ\\
 & a & \longmapsto & \#_wCS(f_k^P,\io_k\circ h_a)
\end{array},
\]
where $h_a$ is any pseudocycle such that $\Phi_*[h_a]=a$.

On the other hand the cohomology class $(v_k)^*PD^{-1}[\De_{M\ssq_cS^1\times M\ssq_cS^1}]$ seen as a map from integer homology to the rationals is defined by sending $a$ to the intersection number of $\im (v_k\circ h_a)$ and $\De_{M\ssq_cS^1\times M\ssq_cS^1}$. That this value coincides with $\#_wCS(f_k^P,\io_k\circ h_a)$ follows from the fact that
\[
v_k(\im f_k^P\cap(\mu^{-1}(c)\times\mu^{-1}(c))_k)=\De_{M\ssq_cS^1\times M\ssq_cS^1},
\]
which we shall now prove:

By equivariance we only need to see that if $(p,q)\in\im f^P\cap(\mu^{-1}(c)\times\mu^{-1}(c))$ then $p$ and $q$ are in the same $S^1$-orbit: since $(p,q)\in\im f^P$ there exists $\al\in S^1$ such that $p$ and $\al\cdot q$ are the beginning and the end of a $P$-perturbed gradient line, but since $\mu(\al\cdot q)=\mu(q)=\mu(p)$ the perturbed gradient line must be point-like and $p=\al\cdot q$, so $p$ and $q$ are in the same $S^1$-orbit.
\end{proof}

\begin{coro}[Kirwan surjectivity]\label{coro:kirwansurj}
The Kirwan map $\ka_c$ is surjective for every regular value $c$ of $\mu$.
\end{coro}

\begin{proof}
Given any regular value $c$ of $\mu$ the theorem asserts that $\lm(1)$ is a biinvariant diagonal class for $c$. Then, according to Lemma \ref{lemm:rightinverse}, $l_c^{\lm(1)}$ is a right-inverse of $\ka_c$ and hence $\ka_c$ is surjective.
\end{proof}

\begin{rema}
Kirwan surjectivity was originally proved by Kirwan in \cite{Kir} in a more general context; namely for the action of any compact and connected Lie group. The techniques we used are rather different from hers but have the advantage that a right-inverse of the Kirwan map is made explicit. Our techniques are much more similar to the ones in \cite{Mun}, where Kirwan surjectivity is recovered in the same context as ours by means of the construction of a global biinvariant diagonal class. The upshot here is that the global biinvariant diagonal class we obtained is just $\lm(1)$ and we can get extra results by further studying the lambda-map.
\end{rema}


\section{Biinvariant diagonal classes of product manifolds}\label{sec:biidiagprod}

We have seen that the lambda-map provides a global biinvariant diagonal class and, as a consequence, Kirwan's surjectivity at every regular level. To do so we just needed to compute $\lm(1)$. In this section we compute the global biinvariant diagonal class of a product Hamiltonian space in terms of the lambda-maps of its components. To do so it is convenient to study in detail the source space of the lambda-map, which is the equivariant cohomology ring $H^*_{\STS}(\CC P^1)$. 

Let $(M,\om_M,S^1,\mu_M)$ and $(N,\om_N,S^1,\mu_N)$ be Hamiltonian spaces and let 
\[
(\MTN,\pi_M^*\om_M+\pi_N^*\om_N,S^1,\mu_M+\mu_N)
\] 
be their product Hamiltonian space (here $\pi_M,\pi_N$ are the projections). We would like to relate the lambda map $\lm^{\MTN}$ of the product space with the lambda maps $\lm^M,\lm^N$ of its components.

Let $\io_\De:\CC P^1\rightarrow\CC P^1\times\CC P^1$ denote the diagonal map. In Section \ref{subsec:edc} we constructed an equivariant version of the shriek map it induces: 
\[
\io_{\De,\STS}^!:H^*_{\STS}(\CC P^1)\longrightarrow H_{\STS}^{*+2}(\CC P^1\times\CC P^1).
\]
By means of equivariant Künneth we see that the target space of $\io_{\De,\STS}^!$ can be identified with
\[
\bigoplus_{q=0}^{*+2} H^q_{\STS}(\CC P^1)\ot_{H^*(\CC P^\infty\times\CC P^\infty)}H_{\STS}^{*+2-q}(\CC P^1).
\] 
If we apply $\lm^M\ot\lm^N$ to an element of this space we obtain an element of
\[
\bigoplus_{q=0}^{*+2} H^{q+2m-2}_{\STS}(\MTM)\ot_{H^*(\CC P^\infty\times\CC P^\infty)}H_{\STS}^{*+2n-q}(\NTN)
\]
which again in terms of equivariant Künneth can be seen  as a subspace of $H^{*+2(m+n)-2}_{\STS}(\MTM\times\NTN)$. In these terms we can think
\[
(\lm^M\ot\lm^N)\circ\io^!_{\De,{\STS}}:H^*_{\STS}(\CC P^1)\longrightarrow H^{*+2(m+n)-2}_{\STS}(\MTN\times\MTN).
\]
The result we want to prove is

\begin{theo}\label{theo:bigtwo}
We have that $(\lm^M\ot\lm^N)\circ\io^!_{\De,{\STS}}=\lm^{\MTN}$.
\end{theo}

\begin{rema}
Before doing anything else, we need to be careful on how we take perturbations: if $J_M,J_N$ are almost complex structures compatible with $\om_M,\om_N$ respectively, then $J_M\times J_N$ is an almost complex structure compatible with $\pi_M^*\om_M+\pi_N^*\om_N$. Moreover if ${\cal D}_M,{\cal D}_N$ and ${\cal D}_{\MTN}$ are the bundles of infinitesimal perturbations of $J_M,J_N$ and $J_M\times J_N$, the bundle $\pi_M^*{\cal D}_M\oplus\pi_N^*{\cal D}_N$ is a subbundle of ${\cal D}_{\MTN}$. We shall take a perturbation zone $Z'$ around a union of regular values $Z$ in $\MTN$ such that its projections onto $M,N$ are valid perturbation zones too: if $(p,q)\in Z$ we need that $p$ and $q$ are not both fixed points. Then if $\de$ is small enough, the projections of the solid torus $R((p,q),\de)$ on $M,N$ are also valid tori. Recall that perturbations are supported on a finite number of tori the union of which contains $Z$. From there and the relation of bundles above we can construct a space of perturbations $\PP_{\MTN}$ on $\MTN$ that gives rise to spaces of perturbations $\PP_M,\PP_N$ and such that each $P_{\MTN}\in\PP_{\MTN}$ gives $P_M\in\PP_M$ and $P_N\in\PP_N$. It is in this framework that we shall prove the theorem.
\end{rema}

\begin{proof}
As in the remark take a perturbation $P_{\MTN}\in\PP_{\MTN}$ and the associated perturbations $P_M,P_N$. In Section \ref{sec:modpbgd} we defined the moduli space of perturbed broken gradient lines $\GG^{P_M}$ associated to $M$. To simplify notation we will write $\GG^M:=\GG^{P_M}_\emptyset\times S^1$. At the beginning of Section \ref{sec:lambdageneral} we defined $\STS$-equivariant maps $g^{P_M},f^{P_M}$ from $\GG^M$ to $\CC P^1$ and $\MTM$ respectively. Again to simplify notation let us write $g^M,f^M$ for these maps. The bundles
\[
E_k:=S^{2k+1}\times S^{2k+1}\longrightarrow B_k:=\CC P^k\times\CC P^k
\] 
stabilise to the universal bundle of $\STS$. Hence we can form the associated bundles $\GG^M_k$, $\CC P^1_k$, $(\MTM)_k$ over $B_k$ and we the corresponding associated maps $g_k^M,f_k^M$ (these notations are taken again from Section \ref{sec:lambdageneral}). The very same concepts applied to $N$ and $\MTN$ respectively let us define $\GG^N_k$,$g^N_k$,$f^N_k$ and $\GG^{\MTN}_k$, $g^{\MTN}_k$,$f^{\MTN}_k$.

Consider the two diagrams
\[
\begin{diagram}
\node{CS(g_k^{\MTN},h)}\arrow{e,t}{\pi_h}\arrow{s}\node{\GG^{\MTN}_k}\arrow{e,t}{f^{\MTN}_k}\arrow{s,r}{g^{\MTN}_k}\node{((\MTN)\times(\MTN))_k}\\
\node{V}\arrow{e,b}{h}\node{\CC P^1_k}
\end{diagram},
\]
\[
\begin{diagram}
\node{CS(g_k^M\times g_k^N,\io_{\De,k}\circ h)}\arrow{e,t}{\pi_{\io_{\De,k}\circ h}}\arrow{s}\node{\GG^M_k\times_{B_k}\GG^N_k}\arrow{e,t}{f^M_k\times f^N_k}\arrow{s,r}{g^M_k\times g^N_k}\node{(\MTM)_k\times_{B_k}(\NTN)_k}\\
\node{V}\arrow{e,b}{\io_{\De,k}\circ h}\node{\CC P^1_k\times_{B_k}\CC P^1_k}
\end{diagram},
\]
where $h:V\rightarrow\CC P^1_k$ is a pseudocycle. According to Theorem \ref{theo:bigone} the first diagram is the one used to define the successive approximations $\lm^{\MTN}_k$ of the lambda map $\lm^{\MTN}$. The same construction of Theorem \ref{theo:bigone} applied to the second diagram gives rise to $(\lm_k^M\ot\lm_k^N)\circ\io_{\De,k}^!$. 

Once again according to Theorem \ref{theo:bigone} the spaces we need to study are the images of $f_k^{\MTN}\circ\pi_h$ and $(f_k^M\circ g_k^M)\circ\pi_{\io_{\De,k}\circ h}$. Their respective target spaces are identified under the diffeomorphism
\[
\begin{array}{rccc}
\chi:&((\MTN)\times(\MTN))_k  & \longrightarrow & (\MTM)_k\times_{B_k}(\NTN)_k\\
 &\left[((p,q),(p',q')),s\right] & \longmapsto & ([(p,p'),s],[(q,q'),s])
\end{array}.
\]
Therefore the result we are now proving will be completed if we are able to see that the images of $f_k^{\MTN}\circ\pi_h$ and $(f_k^M\circ g_k^M)\circ\pi_{\io_{\De,k}\circ h}$ are identified under $\chi$. To see this we will construct a diffeomorphism 
\[
\tilde{\chi}:CS(g_k^{\MTN},h)\rightarrow CS(g_k^M\times g_k^N,\io_{\De,k}\circ h)
\] 
between their source spaces such that
\[
\begin{diagram}
\node{CS(g_k^{\MTN},h)}\arrow[4]{e,t}{f^{\MTN}_k\circ\pi_h}\arrow{s,lr}{\tilde{\chi}}{\simeq}\node{}\node{}\node{}\node{((\MTN)\times(\MTN))_k}\arrow{s,lr}{\simeq}{\chi}\\
\node{CS(g_k^M\times g_k^N,\io_{\De,k}\circ h)}\arrow[4]{e,b}{(f^M_k\times f^N_k)\circ\pi_{\io_{\De,k}\circ h}}\node{}\node{}\node{}\node{(\MTM)_k\times_{B_k}(\NTN)_k}
\end{diagram}
\]
commutes.

Let $(G,\al)\in\GG^{\MTN}$. If $G=(K;\ldots;b)$ then $\pi_MG:=(\pi_M(K);\ldots;\pi_M(b))$ with the suitable lifts is an element of $\GG^M$. The same construction can be done for the projection on $N$. If $G$ is point-like define $t_G:=0$. Otherwise $G$ is parametrised by a curve $\ga:[t_0,t_1]\rightarrow\MTN$. Define $t_G:=t_1-t_0$ if $G$ is descending and define $t_G:=t_0-t_1$ if $G$ is ascending. Then $([(G,\al),s],v)\in CS(g_k^{\MTN},h)$ if and only if $h(v)=[[1:2^{t_G}\al],s]$. We define
\[
\tilde{\chi}([(G,\al),s],v)=([(\pi_MG,\al),s],[(\pi_NG,\al),s],v),
\]
which is an element of $CS(f_k^P\times g_k^P,\io_{\De,k}\circ h)$ because $t_{\pi_MG}=t_{\pi_NG}=t_G$. We shall now construct the inverse of $\tilde{\chi}$:

The elements of $CS(g_k^M\times g_k^N,\io_{\De,k}\circ h)$ are those of the form
\[
([(G,\al),s'],[(G',\al'),s'],v)\in\GG^M_k\times\GG^N_k\times V
\] 
such that $[s]=[s']$ and $h(v)=[[1:2^{t_G}\al],s]=[[1:2^{t_{G'}}\al'],s']$.

Since $s,s'\in S^{2k+1}\times S^{2k+1}$ are in the same orbit and the action of $\STS$ on this space is free, there exists a unique $(\te,\ze)\in\STS$ such that $(\te,\ze)s=s'$. Then
\[
[[1:2^{t_{G'}}\al'],s']=[[1: 2^{t_{G'}}\al'],(\te,\ze)s]=[(\bar{\te},\bar{\ze})[1:2^{t_{G'}}\al'],s]
\]
and hence
\[
[1:2^{t_G}\al]=(\bar{\te},\bar{\ze})[1:2^{t_{G'}}\al']=[\bar{\te},2^{t_{G'}}\bar{\ze}\al']=[1:2^{t_{G'}}\te\bar{\ze}\al'].
\]
Therefore we have that $2^{t_G}\al=2^{t_{G'}}\te\bar{\ze}\al'$ and in particular $|2^{t_G}|=|2^{t_{G'}}|$, from where we deduce that $t_G=t_{G'}$ and then that $\al=\te\bar{\ze}\al'$, so $\al'=\ze\al\bar{\te}$. Now we have that $(G',\al')=(G',\ze\al\bar{\te})=(\te,\ze)\cdot(\bar{\te}G',\al)$ according to how the action of $\STS$ on $\GG^N$ is defined. 

Finally we have that $[(G',\al'),s']=[(\te,\ze)\cdot(\bar{\te}G',\al),(\te,\ze)s]=[(\bar{\te}G',\al),s]$. Since $\bar{\te}$ can be any element of $S^1$ and $t_{\bar{\te}G'}=t_{G'}$ we end up finding out that $CS(g_k^M\times g_k^N,\io_{\De,k}\circ h)$ is the set of elements of the form
\[
([(G,\al),s],[(G',\al),s],v)
\]
where $\al\in S^1$, $s\in S^{2k+1}\times S^{2k+1}$, $v\in V$ and $G,G'$ are perturbed curves of $M,N$ with no breaking points such that $t_G=t_{G'}$ and such that $h(v)=[[1:2^{t_G}\al],s]$. Given an element of this set let $G=(K;\ldots;b)$ and let $G'=(K';\ldots;b')$. Let $G\times G'=(K\times K';\ldots;(b,b'))$ with the lifts chosen according to those of $G$ and $G'$. Since $t_G=t_{G'}$, if $G,G'$ are not point-like then $h_{K\times K'}$ gives to $G\times G'$ the structure of a perturbed gradient line of $\MTN$. The construction implies that the projections of $G\times G'$ are precisely $G$ and $G'$. Moreover, $t_{G\times G'}=t_G=t_{G'}$ so $([(G\times G',\al),s],v)$ is an element of $CS(g_k^{\MTN},h)$. This finishes the construction of the inverse of $\tilde{\chi}$.

Now if $G\in\GG^{\MTN}$ has beginning and end $(b,b'),(e,e')\in\MTN$ we have that
\[
\begin{array}{rcl}
(\chi\circ f_k^{\MTN}\circ\pi_h)([(G,\al),s],v) & = & \chi([(b,b'),(\al\cdot e,\al\cdot e'),s])\\
&=&([(b,\al\cdot e),s],[(b',\al\cdot e'),s])
\end{array}
\]
and that
\[
\begin{array}{l}
((f_k^M\circ f_k^N)\circ\pi_{\io_{\De,k}\circ h}\circ\tilde{\chi})([(G,\al),s],v)=\\ 
\hspace{3cm}=(f_k^M\times f_k^N)([(\pi_MG,\al),s],[(\pi_NG,\al),s])\\
\hspace{3cm}=([(b,\al\cdot e),s],[(b',\al\cdot e'),s])
\end{array}.
\]
This shows that $\chi\circ f_k^{\MTN}\circ\pi_h=(f_k^M\circ f_k^N)\circ\pi_{\io_{\De,k}\circ h}\circ\tilde{\chi}$ and the proof is completed.
\end{proof}

\subsection{The cohomology ring $H^*_{\STS}(\CC P^1)$}

The source space of both the lambda-map and $\io^!_{\De,\STS}$ is the cohomology ring $H^*_{\STS}(\CC P^1)$. In this section we compute it in detail --also giving a geometric interpretation of the classes that generate it-- and we study its image under $\io^!_{\De,\STS}$. Combining this with the last theorem we will get formulas for a global biinvariant class of a product Hamiltonian space.

From equivariant cohomology we know that 
\[
H_{\STS}^*(\CC P^1)=H^*(\CC P^1\times_{\STS}(S^\infty\times S^\infty)),
\]
where
\[
\pi:\CC P^1\times_{\STS}(S^\infty\times S^\infty)\longrightarrow\CC P^\infty\times\CC P^\infty
\]
is the fibre bundle induced by the $\STS$ actions on $\CC P^1$ and on its universal bundle.

Let $\rho_1,\rho_2:\CC P^\infty\times\CC P^\infty\rightarrow\CC P^\infty$ denote the projections to each factor and let $\OO_j(-1)=\rho_j^*\OO_{\CC P^\infty}(-1)$ be the pull-back of the tautological bundle $\OO_{\CC P^\infty}(-1)$. If $s_1,s_2\in S^\infty\subseteq\CC^\infty$, the fibre of the bundle
\[
\CP(\OO_1(-1)\oplus\OO_2(-1))\longrightarrow\CC P^\infty\times\CC P^\infty
\]
over the point $([s_1],[s_2])$ is $\CP(\langle s_1\rangle\oplus\langle s_2\rangle)$, where $\langle s_j\rangle$ is the $1$-dimensional complex subspace of $\CC^\infty$ generated by $s_j$.

We can identify $\CP(\OO_1(-1)\oplus\OO_2(-1))$ with $\CC P^1\times_{\STS}(S^\infty\times S^\infty)$ via
\[
[as_1,bs_2]\longmapsto \left[[a:b],(s_1,s_2)\right].
\]
From this identification we get that
\[
H_{\STS}^*(\CC P^1)\simeq H^*(\CP(\OO_1(-1)\oplus\OO_2(-1)))
\]
and we can compute this cohomology in terms of Chern classes: the total Chern class of $\OO_1(-1)\oplus\OO_2(-1)$ is 
\[
\begin{array}{rcl}
c(\OO_1(-1)\oplus\OO_2(-1)) & = & c(\OO_1(-1))c(\OO_2(-1))\\
 & = & (1+\tau_1)(1+\tau_2)\\
 & = & 1+(\tau_1+\tau_2)+\tau_1\tau_2,
\end{array}
\]
where $\tau_1,\tau_2$ are the pull-backs, under $\rho_1,\rho_2$ respectively, of the Euler class of $\OO_{\CC P^\infty}(-1)$. Then we have
\[
c_1(\OO_1(-1)\oplus\OO_2(-1))=\tau_1+\tau_2,\ \ \ c_2(\OO_1(-1)\oplus\OO_2(-1))=\tau_1\tau_2.
\]
From the Leray-Hirsch theorem we deduce that
\[
\begin{array}{rcl}
H^*(\CP(\OO_1(-1)\oplus\OO_2(-1))) & \simeq & \QQ[t_1,t_2,u]/(u^2-(t_1+t_2)u+t_1t_2)\\
 & \simeq & \QQ[t_1,t_2,u]/(u-t_1)(u-t_2)
\end{array},
\]
where $u$ is the Euler class of the bundle $\OO_{\mathbb{P}(\OO_1(-1)\oplus\OO_2(-1))}(-1)$ and $t_1=\pi^*\tau_1,t_2=\pi^*\tau_2$. In conclusion we have that 
\[
H^*_{\STS}(\CC P^1)\simeq\QQ[t_1,t_2,u]/(u-t_1)(u-t_2).
\]

To study the image of $\io_{\De,\STS}^!$ we will use its approximations $\io_{\De,k}^!$ induced by the diagonal inclusion $\io_{\De,k}:\CC P^1_k\rightarrow(\CC P^1\times\CC P^1)_k$ as we did in Section \ref{subsec:edc}. As usual let $\CC P^1_k=\CC P^1\times_{\STS}(S^{2k+1}\times S^{2k+1})$. We can carry out the exact same construction as above to compute the approximations $H^*(\CC P^1_k)$ of $H^*_{\STS}(\CC P^1)$:
 
Let $\rho_1,\rho_2:\CC P^k\times\CC P^k\rightarrow\CC P^k$ be the projections to each factor and let $\OO_j(-1)=\rho_j^*\OO_{\CC P^k}(-1)$. There is an identification $\CP(\OO_1(-1)\oplus\OO_2(-1))\simeq\CC P^1_k$ and 
\[
H^*(\CC P^1_k)=\QQ[u,t_1,t_2]/(t_1^{k+1},t_2^{k+1},(u-t_1)(u-t_2))
\]
where $u$ is the Euler class of $\OO_{\CP(\OO_1(-1)\oplus\OO_2(-1)))}(-1)$ and $t_1,t_2$ are the pull-backs under $\rho_1,\rho_2$ of the Euler class of $\OO_{\CC P^k}(-1)$.

Let $s_1,s_2\in S^{2k+1}\subseteq\CC^{k+1}$ and let $H^0=\{(s^0,\ldots,s^k)\in\CC^{k+1}:\ H^0=0\}$. Then
\[
\begin{array}{rccc}
\sig: & \CP(\OO_1(-1)\oplus\OO_2(-1))) & \longrightarrow & \OO_{\CP(\OO_1(-1)\oplus\OO_2(-1)))}(-1)\\
 &[as_1,bs_2] &\longmapsto &\left(\displaystyle\frac{s_1}{s_1^0},\frac{bs_2}{as_1^0}\right)
\end{array}
\]
is  a well-defined meromorphic section with no zeroes and poles in the union of
\[
P=\{[0:1]\}\times_{\STS}(S^{2k+1}\times S^{2k+1})\ \mbox{and}\ Q=\CC P^1\times_{\STS}((S^{2k+1}\cap H^0)\times S^{2k+1}).
\]
From this we deduce that $PD(u)=-[P]-[Q]$.

If $s\in S^{2k+1}$, then
\[
\begin{array}{ccc}
\CC P^k & \longrightarrow &\OO_{\CC P^k}(-1)\\
\left[s\right] & \longmapsto & \displaystyle\frac{s}{s^0}
\end{array}
\] 
is a meromorphic section with no zeroes and pole locus $\CP(H^0)$. Hence the Poincaré dual of the Euler class is $PD(c_1(\OO(\CC P^k)))=-[\CP(H_0)]$. Now observe that$Q=\pi^{-1}(\rho_1^{-1}(\CP(H_0)))$, and since $\pi$ and $\rho_1$ are submersions we have that
\[
\begin{array}{rcl}
-PD^{-1}[Q] & = & -PD^{-1}[\pi^{-1}(\rho_1^{-1}(\CP(H_0)))]\\
            & = & -\pi^*\rho_1^*PD^{-1}[\CP(H_0)]\\
            & = & \pi^*\rho_1^*c_1(\OO_{\CC P^k}(-1))\\
            & = & t_1.
\end{array}
\]
From this we obtain that $u=-PD^{-1}[P]+t_1$.

Under the diffeomorphism
\[
\begin{array}{ccc}
(\CC P^1\times\CC P^1)_k &\longrightarrow & \CC P^1_k\times_{\CC P^k\times\CC P^k}\CC P^1_k\\
\left[(x,y),(s_1,s_2)\right] & \longmapsto & [[x,(s_1,s_2)],[y,(s_1,s_2)]]
\end{array}
\]
we can think of $\io_{\De,k}$ as the map
\[
\begin{array}{rccc}
\io_{\De,k}: &\CC P^1_k &\longrightarrow & \CC P^1_k\times_{\CC P^k\times\CC P^k}\CC P^1_k\\
 & \left[x,(s_1,s_2)\right] & \longmapsto & [[x,(s_1,s_2)],[x,(s_1,s_2)]]
\end{array}
\]
By Leray-Hirsch and Künneth, there is an isomorphism
\[
\begin{array}{ccc}
H^*(\CC P^k)\ot_{H^*(\CC P^k\times\CC P^k)}H^*(\CC P^k) & \longrightarrow & H^*( \CC P^1_k\times_{\CC P^k\times\CC P^k}\CC P^1_k)\\
\al\ot \be & \longmapsto & \pi_1^*\al\smile\pi_2^*\be
\end{array}
\]
where $\pi_1,\pi_2:\CC P^1_k\times_{\CC P^k\times\CC P^k}\CC P^1_k\rightarrow\CC P^1_k$ are the projections.

In these terms we have

\begin{lemm}\label{lemm:iot}
For $j=1,2$, $\io_{\De,k}^!(t_j)=t_j\io^!_{\De,k}(1)$.
\end{lemm}

\begin{proof}
Let $\tilde{\pi}:\CC P^1_k\times_{\CC P^k\times\CC P^k}\CC P^1_k\rightarrow\CC P^k\times\CC P^k$ be the projection. Then
\[
\begin{array}{rcl}
\io_{\De,k}^!(t_j) & = & \io_{\De,k}^!(\pi^*\rho_j^*c_1(\OO_{\CC P^k}(-1)))\\
                   & \stackrel{(1)}{=} & \io_{\De,k}^!(\io_{\De,k}^*\tilde{\pi}^*\rho_j^*c_1(\OO_{\CC P^k}(-1))\\
                   & = & \tilde{\pi}^*\rho_j^*c_1(\OO_{\CC P^k}(-1))\smile\io^!_{\De,k}(1)\\
                   & \stackrel{(2)}{=} & \pi_j^*\pi^*\rho_j^*c_1(\OO_{\CC P^k}(-1))\smile\io^!_{\De,k}(1)\\
                   & = & \pi_j^*t_j\smile\io^!_{\De,k}(1)\\
                   & = & t_j\io^!_{\De,k}(1)
\end{array}
\]
where we used $\pi=\tilde{\pi}\circ\io_{\De,k}$ in $(1)$, $\tilde{\pi}=\pi\circ\pi_j$ in $(2)$ and Künneth in the last equality.
\end{proof}

In view of this result we are also interested in computing:

\begin{lemm}\label{lemm:ioone}
$\io_{\De,k}^!(1)=(t_1+t_2)(1\ot 1)-u\ot 1-1\ot u$.
\end{lemm}

\begin{proof}
Note that
\[
\begin{array}{rcl}
\io_{\De,k}^!(1) & = & (PD^{-1}\circ(\io_{\De,k})_*\circ PD)(1)\\
                 & = & PD^{-1}\circ(\io_{\De,k})_*[\CC P^1_k]\\
                 & = & PD^{-1}[\io_{\De,k}(\CC P^1_k)]\\
                 & = & PD^{-1}[\De_{\CC P^1_k\times\CC P^1_k}]
\end{array}
\]
Consider the bundle $\pi:\CC P^1_k\rightarrow\CC P^k\times\CC P^k$. In \cite[theo. 4.1]{Grah} there is a description of how to compute the Poincaré dual of $[\De_{\CC P^1_k\times\CC P^1_k}]$ as an element of $H^*(\CC P^1_k\times_{\CC P^k\times\CC P^k}\CC P^1_k)$ in terms of this bundle:

The cohomology of the base is $H^*(\CC P^k\times\CC P^k)\simeq\QQ[\tau_1,\tau_2]/(\tau_1^{k+1},\tau_2^{k+1})$ and the cohomology of the total space is 
\[
H^*(\CC P^1_k)\simeq\QQ[t_1,t_2,u]/((u-t_1)(u-t_2),t_1^{k+1},t_2^{k+1}).
\]

Consider $x_1=y_2=u$ and $x_2=y_1=1$ as elements of $H^*(\CC P^1_k)$. Then both sets $\{x_1,x_2\}$ and $\{y_1,y_2\}$ restrict to a basis of the cohomology of the fibre. Let $z_{ij}=\pi^!(x_iy_j)$ and consider the matrix $Z=(z_{ij})$. According to the aforementioned theorem, $Z$ is invertible and, if $(Z^{-1})^t=(a_{ij})$, the class we are looking for is $\sum a_{ij}x_i\ot y_j$.  

For dimensional reasons $\pi^!(1)=0$. Since $\pi^!$ is integration along the fibre and $u$ is the Euler class, we have  $\pi^!(u)=-1$. Finally, 
\[
\pi^!(u^2)=\pi^!((t_1+t_2)u-t_1t_2)=(t_1+t_2)\pi^!(u)+t_1t_2\pi^!(1)=-(t_1+t_2).
\]
Then
\[
Z=
\begin{pmatrix}
-1 & -(t_1+t_2)\\
0  & -1
\end{pmatrix},
\]
and therefore
\[
(Z^{-1})^t=
\begin{pmatrix}
-1 & 0\\
 t_1+t_2  & -1
\end{pmatrix}.
\]
Then the class $PD^{-1}[\De_{\CC P^1_k\times\CC P^1_k}]$ coincides with
\[
(t_1+t_2)(x_2\ot y_1)+(-1)x_1\ot y_1+ (-1)x_2\ot y_2=(t_1+t_2)(1\ot 1)-u\ot 1-1\ot u
\] 
and the result is proved.
\end{proof}

\begin{rema}
By the usual stabilisation argument in equivariant cohomology  --Lemma \ref{lemm:stabilisation}-- we see that 
$$
\io_{\De,\STS}^!(1)=(t_1+t_2)(1\ot 1)-u\ot 1-1\ot u
$$
in terms of equivariant Künneth. This is the equivariant diagonal class --see Section \ref{subsec:edc}-- for the $\STS$-action on $\CC P^1$.
\end{rema}

The last generator of the cohomology of $H^*(\CC P^1_k)$ for which we have to compute its image under $\io_{\De,k}^!$ is $u$. Let us do it:

\begin{lemm}\label{lemm:iou}
$\io_{\De,k}^!(u)=t_1t_2(1\ot 1)-u\ot u$.
\end{lemm}

\begin{proof}
Recall that $u=t_1-PD^{-1}[P]$ where
\[
P=\{[0:1]\}\times_{\STS}(S^{2k+1}\times S^{2k+1})\subseteq\CC P^1_k. 
\]
From the two previous lemmas we have that
\[
\io^!_{\De,k}(t_1)=t_1\io^!_{\De,k}(1)=(t_1^2+t_1t_2)(1\ot 1)-t_1(u\ot 1)-t_1(1\ot u).
\]
Let us now compute $\io_{\De,k}^!(PD^{-1}[P])$. We have
\[
\begin{array}{rcl}
\io_{\De,k}^!(PD^{-1}[P]) & = & (PD^{-1}\circ(\io_{\De,k})_*\circ PD)(PD^{-1}[P])\\
                 & = & (PD^{-1}\circ(\io_{\De,k}))_*[P]\\
                 & = & PD^{-1}[\io_{\De,k}(P)]\\
                 & = & PD^{-1}[P\times_{\CC P^k\times\CC P^k}P]\\
                 & = & PD^{-1}[(P\times_{\CC P^k\times\CC P^k}\CC P^1_k)\cap(\CC P^1_k\times_{\CC P^k\times\CC P^k}P)]\\
                 & = & PD^{-1}[\pi_1^{-1}(P)\cap\pi_2^{-1}(P)]\\
                 & = & PD^{-1}[\pi_1^{-1}(P)]\smile PD^{-1}[\pi_2^{-1}(P)]\\
                 & = & \pi_1^*PD^{-1}[P]\smile\pi_2^*PD^{-1}[P]\\
                 & = & \pi_1^*(t_1-u)\smile\pi_2^*(t_1-u)\\
                 & = & (t_1-u)\ot(t_1-u)\\
                 & = & t_1^2(1\ot 1)-t_1(u\ot 1)-t_1(1\ot u)+u\ot u
\end{array}
\]
The difference of the two expressions gives $\io^!_{\De,k}(u)=t_1t_2(1\ot 1)-u\ot u$.
\end{proof}

From the three previous lemmas and by means of stabilisation in equivariant cohomology --Lemma \ref{lemm:stabilisation}-- we have, in terms of equivariant Künneth, that
\begin{itemize}
\item $\io^!_{\De,\STS}(1)=(t_1+t_2)(1\ot 1)-u\ot 1-1\ot u$
\item $\io^!_{\De,\STS}(t_j)=(t_j^2+t_1t_2)-t_j(u\ot 1)-t_j(1\ot u)$ for $j=1,2$
\item $\io^!_{\De,\STS}(u)=t_1t_2(1\ot 1)-u\ot u$
\end{itemize}

Now, using these formulas we get the following immediate corollary of Theorem \ref{theo:bigtwo}:

\begin{coro}
Let $(M,\om_M,S^1,\mu_M)$ and $(N,\om_N,S^1,\mu_N)$ be Hamiltonian spaces and let $\lm^M,\lm^N$ be its corresponding lambda-maps. If $\lm^{\MTN}$ is the lambda-map of its corresponding product Hamiltonian space, we have the following formulas:
\begin{enumerate}
\item $\lm^{\MTN}(1)=(t_1+t_2)\lm^M(1)\ot\lm^N(1)-\lm^M(u)\ot\lm^N(1)-\lm^M(1)\ot\lm^N(u)$
\item $\lm^{\MTN}(u)=t_1t_2\lm^M(1)\ot\lm^N(1)-\lm^M(u)\ot\lm^N(u)$
\end{enumerate}
\end{coro}

\begin{rema}
In particular we get a formula to compute $\lm^{\MTN}(1)$, which is a global biinvariant diagonal class of $\MTN$, in terms of $\lm^M$ and $\lm^N$.
\end{rema}

\section[Non-uniqueness of global biinvariant diagonal classes]{Non-uniqueness of global biinvariant diagonal\\ classes}

In Theorem \ref{theo:lambdagivesbdc} we proved that given a Hamiltonian space $(M,\om,S^1,\mu)$ there exists an element $\be\in H^{2m-2}_{\STS}(\MTM)$ which is a global biinvariant diagonal class. It is a natural question to ask whether such a class is unique or not. In this section we prove that this is not true in general. To do so, we carry out some hands-on computations of biinvariant diagonal classes that might be illustrative.

Fix a basis $\la e^q_1,\ldots,e^q_{d_q}\ra$ of $H^q_{S^1}(M)$ for $q=0,\ldots,2m-2$. Given a regular value $c$ of $\mu$, by means of integration on $M\sq_cS^1$ we get a pairing
\[
\begin{array}{rccc}
I_s^q:&H^q_{S^1}(M)\ot H^{2m-2-q}_{S^1}(M)&\longrightarrow&\QQ\\
      &e^q_i\ot e^{2m-2-q}_j&\longmapsto&\int_{M\ssq_cS^1}\ka_c(e^q_i)\smile\ka_c(e^{2m-2-q}_j)
\end{array}.
\]

Let 
\[
\be\in H^{2m-2}_{\STS}(\MTM)\simeq\bigoplus_{q=0}^{2m-2}H_{S^1}^q(M)\ot H_{S^1}^{2m-2-q}
\] 
be a biinvariant diagonal class for the regular value $c$ and let $b_{ij}^q\in\QQ$ denote the coefficient of $e^q_i\ot e^{2m-2-q}_j$ in the expression of $\be$ in terms of the chosen bases. Since $\ka_{(c,c)}(\be)$ is Poincaré dual to the diagonal class, for any $a\in H^{2m-2-q}(M\sq_cS^1)$ we have the relation
\begin{equation}\label{eq:one}
\sum_{i,j}b_{ij}^q\left(\int_{M\ssq_cS^1}\ka_c(e_i^q)\smile a\right)\ka_c(e_j^{2m-2-q})=a.
\end{equation}
In particular, for $a=\ka_c(e_l^{2m-2-q})$ we obtain
\begin{equation}\label{eq:two}
\sum_{i,j}b_{ij}^qI_c^q(e_i^q\ot e_l^{2m-2-q})\ka_c(e_j^{2m-2s-q})=\ka_c(e_l^{2m-2-q}).
\end{equation}
Applying $\int_{M\ssq_cS^1}\ka_c(e_k^q)\smile\cdot$ --i.e. Poincaré duality w.r.t. $\ka_c(e_k^q)$-- to both sides of the equality we get
\begin{equation}\label{eq:three}
\sum_{i,j}b_{ij}^qI_c^q(e_i^q\ot e_l^{2m-2-q})I^q_c(e_k^q\ot e_j^{2m-2-q})=I^q_c(e^q_k\ot e_l^{2m-2-q}).
\end{equation}
These equations can be expressed in a much more compact way in terms of matrices: define the $d_q\times d_{2m-2-q}$ matrices $A_c^q$ with entries $I_c^q(e_i^q\ot e_j^{2m-2-q})$ and $B^q$ with entries $b_{ij}^q$. Equation \ref{eq:three} is equivalent to $A_c^q(B^q)^t A_c^q=A_c^q$. In other words, $(B^q)^t$ is a {\it pseudoinverse}\footnote{In the literature the notion of `pseudoinverse' varies. To us a pseudoinverse of a $r\times c$ matrix $A$ is a $c\times r$ matrix $B$ satisfying the equation $ABA=A.$ This coincides with the notion of $\{1\}$-pseudoinverse in \cite[\S 1.1]{BeGr}.} of $A_c^q$. Conversely, if equation \ref{eq:three} is satisfied, we get equation \ref{eq:two} by applying the inverse of Poincaré duality w.r.t. $\ka_c(e_k^q)$. Then equation \ref{eq:one} is satisfied for classes of the form $a=\ka_c(\al)$, but from Kirwan's surjectivity --Corollary \ref{coro:kirwansurj}-- these are all the possible classes. In summary, we have proved the following result:

\begin{prop} 
A class $\be$ is biinvariant diagonal for a regular value $c$ if and only if $A_c^q(B^q)^tA_c^q=A_c^q$, $q=0,\cdots,2m-2$.
\end{prop}

\begin{rema}\label{rema:symmetry}
Since $I_c^q$ is symmetric we get $A_c^{2m-2-q}=(A_c^q)^t$. Transposing $A_c^{2m-2-q}=A_c^{2m-2-q}(B^{2m-2-q})^t A_c^{2m-2-q}$ results into $A_c^q=A_c^qB^{2m-2-q}A_c^q$, so both $(B^q)^t$ and $B^{2m-2-q}$ have to be pseudoinverses of $A_c^q$.
\end{rema}

The set $P(A)$ of pseudoinverses of a matrix $A$ is an affine space with underlying vector space $Z(A)=\{B:\ ABA=0\}$. We call the elements of $Z(A)$ {\it null-pseudoinverses} of $A$. More in general, if $A_1,\ldots,A_n$ is a collection of matrices with the same number of rows and columns and $P(A_1,\ldots,A_n)$ denotes the space of common pseudoinverses of all these matrices, we have that --if it is not empty-- it is an affine space with underlying vector space $Z(A_1,\ldots,A_n)=Z(A_1)\cap\cdots\cap Z(A_n)$, the space of common null-pseudoinverses.

Going back to our particular setting, we first observe that if $c,c'$ are two regular values of $\mu$ with no critical values in between, $M\sq_cS^1$ and $M\sq_{c'}S^1$ are diffeomorphic and $I^q_c=I^q_{c'}$. Therefore to prove that $(B^q)^t$ is a pseudoinverse of $A_c^q$ for all regular values $c$ it suffices to prove that it is a pseudoinverse of $A_c^q$ only for a finite number of regular values $c_1,\ldots, c_n$. We fix such regular values and to make the notation shorter we write $A_i^q$ for $A_{c_i}^q$. 

In Theorem \ref{theo:lambdagivesbdc} we proved that the image of $1$ under the lambda-map is a global biinvariant diagonal class, so if $L^q$ are the matrices representing $\lm(1)$ in the chosen bases we have $L^q\in P(A_1^q,\ldots,A_n^q)$ and hence $P(A_1^q,\ldots,A_n^q)$ is an affine space with underlying vector space $Z(A_1^q)\cap\cdots\cap Z(A_n^q)$.

We need to find a way to compute the matrices $A_i^q$. We will use the main result in \cite{Kal}, which asserts that if the fixed point set $F$ of the action is finite --or, equivalently, if the fixed points are isolated--, then
\[
\int_{M\ssq_cS^1}\ka_c(\al)=\sum_{\mu(p)>0}\frac{\al_0(p)}{w(p)}
\]
where the sum is taken over the points of $F$. Here $w(p)$ denotes the product of weights of the infinitesimal action on $T_pM$, while $\al_0$ is a certain smooth function associated to the class $\al$. This formula is obtained using a version of equivariant cohomology that we have not explained: similarly to how de Rham cohomology recovers the usual cohomology (with real coefficients) there is a complex which uses differential forms, called the Cartan complex, that recovers equivariant cohomology. The function $\al_0$ turns out to be the $0$-form associated to $\al$ in this complex. An excellent reference for these topics is \cite{GuSt}. 

Let $S^1$ act on $\CC P^n$ by
\[
\te[z_0:\cdots:z_n]=[\te^{j_0}z_0:\cdots:\te^{j_n}z_n]
\]
where $j_0<\ldots<j_n$ are integers. To have isolated fixed points we need these numbers to be pairwise different; we take them in increasing order only for convenience when making computations. In Example \ref{exam:eqcoprojexam} we showed that there exist degree 2 classes $x,t$ such that
\[
H_{S^1}^*(\CC P^n)\simeq \QQ[t,x]/(x+j_0t)\cdots(x+j_nt).
\]
Hence odd degree cohomology vanishes. The function associated to $t$ is the constant function $1$, while the function associated\footnote{In the Cartan complex, $x$ is the class of the equivariant symplectic form $1\ot\om-\tau\ot\mu$ and $t$ is the class of $\tau\ot 1$. See \cite[\S 5]{Kal} for the details.} to $x$ is $\mu$. The function associated to a product of these classes is the product of their associated functions. Since the polynomial giving the relation has degree $2n+2$, if $0\leq q\leq n-1$ there are no relations between generators of $H^{2q}_{S^1}(\CC P^n)$, so $\{t^q,t^{q-1}x,\ldots,tx^{q-1},x^q\}$ is a basis of $H^{2q}_{S^1}(\CC P^n)$. The fixed points are
\[
p_0=(1:0:\cdots:0),\ldots,p_n=(0:\cdots:0:1),
\] 
which satisfy $\mu(p_i)=j_i$, $w(p_i)=\prod_{k\neq i}(j_k-j_i)=:w_i$. Finally take regular values $c_i$ such that $j_0<c_1<j_1<\cdots<j_{n-1}<c_n<j_n$. Kalkman's formula gives
\[
\int_{M\ssq_{c_i}S^1}\ka_{c_i}(t^{n-1-q}x^q)=\sum_{k=i}^n\frac{\mu(p_k)^q}{w(p_k)}=\sum_{k=i}^n\frac{j_k^q}{w_k}.
\]
Using this formula, we observe that if $t^{q-a}x^a$ is an element of the basis of $H^{2q}_{S^1}(\CC P^n)$ and $t^{n-1-q-b}x^b$ is an element of the basis of $H^{2(n-1-q)}_{S^1}(\CC P^n)$ we have that
\[
I_{c_i}^{2q}(t^{q-a}x^a\ot t^{n-1-q-b}x^b)=\int_{M\ssq_{c_i}S^1}\ka_{c_i}(t^{n-1-a-b}x^{a+b})=\sum_{k=i}^n\frac{j_k^{a+b}}{w_k},
\]
so we can compute explicitely the matrices $A_i^{2q}$ and look for their common pseudoinverses.

\begin{exam} Let us make all the computations for the case $\CC P^2$. We will show that in this case there is a unique global biinvariant diagonal class:

For $\CC P^2$ we have bases $\{1\}$ of $H^0_{S^1}(\CC P^2)$ and $\{t,x\}$ of $H_{S^1}^2(\CC P^2)$. The matrices $A_1^0,A_2^0$ are of dimension $1\times 2$. Concretely 
\[
A_1^0=\begin{bmatrix}I_{c_1}^0(1\ot t) & I_{c_1}^0(1\ot x)\end{bmatrix}\ \mbox{and}\ A_2^0=\begin{bmatrix}I_{c_2}^0(1\ot t) & I_{c_2}^0(1\ot x)\end{bmatrix}.
\] 
Let us compute them using Kalkman's formula:
\[
\left.
\begin{array}{r}
\displaystyle I_{c_1}^0(1\ot t)=\frac{1}{w_1}+\frac{1}{w_2}=-\frac{1}{w_0}\\
\displaystyle I_{c_1}^0(1\ot x)=\frac{j_1}{w_1}+\frac{j_2}{w_2}=-\frac{j_0}{w_0}
\end{array}
\right\}
\Rightarrow A_1^0=-\frac{1}{w_0}\begin{bmatrix}1 & j_0\end{bmatrix},
\]
\[
\left.
\begin{array}{r}
\displaystyle I_{c_2}^0(1\ot t)=\frac{1}{w_2}\\
\displaystyle I_{c_2}^0(1\ot x)=\frac{j_2}{w_2}
\end{array}
\right\}
\Rightarrow A_2^0=\frac{1}{w_2}\begin{bmatrix}1 & j_2\end{bmatrix}.
\]

We now want to find a $1\times 2$ matrix $B^0=\begin{bmatrix}b_{11}^0 & b_{12}^0\end{bmatrix}$ satisfying both $A_1^0(B^0)^tA_1^0=A_1^0$ and $A_2^0(B^0)^tA_2^0=A_2^0$. Since $A_0^0(B^0)^t$ and $A_1^0(B^0)^t$ are $1\times 1$ matrices, these equations are equivalent to $A_0^0(B^0)^t=A_1^0(B^0)^\top=\begin{bmatrix}1\end{bmatrix}$. These can be wrapped up in the system
\[
\begin{bmatrix}
1& j_0 \\
1& j_2 
\end{bmatrix}
\begin{bmatrix}
b_{11}^0\\ b_{12}^0
\end{bmatrix}=
\begin{bmatrix}
-w_0\\ w_2
\end{bmatrix},
\]
which has as a unique solution $B^0=(j_2-j_0)\begin{bmatrix}-j_1 & 1\end{bmatrix}$. Using the symmetry explained in Remark \ref{rema:symmetry} we get that there is also a unique matrix $B^2$ satisfying the required equations and that it must be $B^2=(B^0)^t$. As a conclusion we find out that
\[
(j_2-j_0)(1\ot x+x\ot 1-j_1(1\ot t+t\ot 1))
\]
is the unique global biinvariant diagonal class.
\end{exam}

We can carry out similar computations also for the $S^1$-action on a product of projective spaces. Let us study in more detail a very particular case, which is very suited for computations and eventually provides an example where global biinvariant diagonal classes are not unique:

Let $S^1$ act on $\CC P^1\times\stackrel{n}{\cdots}\times\CC P^1$ by
\[
\te([z_0:w_0],\ldots,[z_{n-1}:w_{n-1}])=([\te z_0:w_0],\ldots,[\te^{2^{n-1}}z_{n-1}:w_{n-1}]),
\]
i.e. the action on the $i$-th component is given by $[\te^{2^i}z_i:w_i]$. The equivariant cohomology ring is
\[
H^*_{S^1}(\CC P^1\times\stackrel{n}{\cdots}\times\CC P^1)\simeq\QQ[t,x_0,\ldots,x_{n-1}]/((x_0+t)x_0,\ldots,(x_{n-1}+2^{n-1}t)x_{n-1})
\]
so a basis of $H^{2q}_{S^1}(\CC P^1\times\stackrel{n}{\cdots}\times\CC P^1)$ is given by $t^q$ and the elements of the form $x_{k_1}\cdots x_{k_q}$ with 
$0\leq k_1\leq\cdots\leq k_q\leq n-1$. We order these elements lexicographically with respect to the alphabet $t,x_0,\ldots,x_{n-1}$.

The critical points are those of the form $(p_0,\ldots,p_{n-1})$ with $p_i$ being either $[1:0]$ or $[0:1]$. So we have $2^n$ critical points, which can be encoded using $n$-bit strings:
\[
\bb=b_0\cdots b_{n-1}\ \ \mbox{with}\ \ b_i=\left\{\begin{array}{ccc} 0 &\mbox{if} &p_i=[0:1]\\ 1 &\mbox{if} &p_i=[1:0]\end{array}\right..
\]
Let $\sig(\bb)=(-1)^{b_0+\cdots+b_{n-1}}$ be the {\it signature} of $\bb$ and let $\bar{\bb}\in\ZZ$ be the integer represented by the string $b_{n-1}\cdots b_0$ as a binary number. Define also the value $e_n=\frac{1}{2}n(n-1)$. The moment map and the product of weights of the linearised action at the critical points are given by
\[
\mu(\bb)=\sum_{i=0}^{n-1}2^ib_i=\bar{\bb}
\] 
and 
\[
w(\bb)=\sig(\bb)\cdot 2^{e_n}
\] 
respectively. They can be computed from the moment map and the product of weights on each factor: $\mu$ is the sum of $\mu_i(p_i)=2^ib_i$ and $w$ is the product of $w_i(p_i)=(-1)^{b_i}\cdot 2^i$.

We have $2^n$ critical values whose images under $\mu$ are precisely $0,\ldots,2^n-1$. Consider the regular values $c_1,\ldots,c_{2^n-1}$ defined by $c_j=(2j-1)/2$. Then $\mu(\bb)>c_j$ if and only if $\bar{\bb}\geq j$. Now we take the elements of the chosen basis of $H^{2n-2}_{S^1}(\CC P^1\times\stackrel{n}{\cdots}\times\CC P^1)$ and compute the integrals of their images under $\ka_{c_j}$ by means of Kalkman's formula:
\[
\int_{M\ssq_{c_j}S^1}\ka_{c_j}(t^{n-1})=2^{-e_n}\sum_{\bar{\bb}\geq j}\sig(\bb),
\]
\[
\begin{array}{rcl}
\displaystyle \int_{M\ssq_{c_j}S^1}\ka_{c_j}(x_{k_1}\cdots x_{k_{n-1}}) & = & \displaystyle 2^{-e_n}\sum_{\bar{\bb}\geq j}\sig(\bb)\cdot\textstyle\prod_{\ell=1}^{n-1}\mu_{k_\ell}(b_{k_\ell})\\
 & = & \displaystyle 2^{-e_n}\sum_{\bar{\bb}\geq j}\sig(\bb)\cdot\textstyle\prod_{\ell=1}^{n-1}2^{k_\ell}b_{k_\ell}\\
 & = & \displaystyle 2^{-e_n+\sum_{\ell=1}^{n-1}k_\ell}\sum_{\bar{\bb}\geq j}\sig(\bb)\cdot\textstyle\prod_{\ell=1}^{n-1}b_{k_\ell}
\end{array}.
\]

With these formulas we can show an example where global biinvariant diagonal classes are not unique:

\begin{exam} Let us study the action we just described for the case $n=3$. The chosen basis of $H^4_{S^1}(\CC P^1\times\CC P^1\times\CC P^1)$ is $\{t^2,x_0^2,x_0x_1,x_0x_2,x_1^2,x_1x_2,x_2^2\}$. We start by computing the relevant values in the formulas above:

The first table gives the signatures of $\bar{\bb}$:
\[
\begin{array}{|c||c|c|c|c|c|c|c|c|}
\hline
\bar{\bb} & 0 & 1 & 2 & 3 & 4 & 5 & 6 & 7\\
\hline
\sig(\bb) &+1 &-1 &-1 &+1 &-1 &+1 &+1 &-1\\
\hline
\end{array}
\]

Since $n=3$ we have that $e_n=3$. The second table gives the values of the exponents $-e_n+k_1+k_2$ according to the values of $k_1,k_2$:
\[
\begin{array}{|c||c|c|c|c|c|c|}
\hline
k_1,k_2 & 0,0 & 0,1 & 0,2 & 1,1 & 1,2 & 2,2 \\
\hline
-e_n+k_1+k_2 &-3 &-2 &-1 &-1 & 0 & 1\\
\hline
\end{array}
\]

Now we write a table that gives the value of $b_{k_1}b_{k_2}$ in function of $\bar{\bb}$ (rows) and $k_1,k_2$ (columns):
\[
\begin{array}{|c||c|c|c|c|c|c|}
\hline
b_{k_1}b_{k_2} & 0,0 & 0,1 & 0,2 & 1,1 & 1,2 & 2,2 \\
\hhline{|=||=|=|=|=|=|=|}
0              &   0 &   0 &   0 &   0 &   0 &   0 \\
\hline
1              &   1 &   0 &   0 &   0 &   0 &   0 \\
\hline
2              &   0 &   0 &   0 &   1 &   0 &   0 \\
\hline
3              &   1 &   1 &   0 &   1 &   0 &   0 \\
\hline
4              &   0 &   0 &   0 &   0 &   0 &   1 \\
\hline
5              &   1 &   0 &   1 &   0 &   0 &   1 \\
\hline
6              &   0 &   0 &   0 &   1 &   1 &   1 \\
\hline
7              &   1 &   1 &   1 &   1 &   1 &   1 \\
\hline
\end{array}
\]

From these three tables we can extract all the necessary information to make the next table, which we call $T_1$, the rows of which are again labelled by $\bar{\bb}$, while the columns are labelled by $t^2$ and $x_{k_1}x_{k_2}$ according to the order given above. The column under $t^2$ contains the values $2^{-e_n}\sig(\bb)$ and the column under $x_{k_1}x_{k_2}$ contains the values $2^{-e_n+k_1+k_2}\sig(\bb)b_{k_1}b_{k_2}$:

\[
\begin{array}{|c||r|r|r|r|r|r|r|}
\hline
T_1 & t^2 & x_0^2 & x_0x_1 & x_0x_2 & x_1^2 & x_1x_2 & x_2^2 \\ \hhline{|=||=|=|=|=|=|=|=|}
  0 & 1/8 &     0 &      0 &      0 &     0 &      0 &   0   \\ \hline
  1 &-1/8 &  -1/8 &      0 &      0 &     0 &      0 &   0   \\ \hline
  2 &-1/8 &     0 &      0 &      0 &  -1/2 &      0 &   0   \\ \hline
  3 & 1/8 &   1/8 &    1/4 &      0 &   1/2 &      0 &   0   \\ \hline
  4 &-1/8 &     0 &      0 &      0 &     0 &      0 &  -2   \\ \hline
  5 & 1/8 &   1/8 &      0 &    1/2 &     0 &      0 &   2   \\ \hline
  6 & 1/8 &     0 &      0 &      0 &   1/2 &      1 &   2   \\ \hline
  7 &-1/8 &  -1/8 &   -1/4 &   -1/2 &  -1/2 &     -1 &  -2   \\ \hline
\end{array}
\]

We write down one more table, $T_2$, the columns of which are again labelled by the elements of the basis while the rows contain the values 1 to 7, corresponding to the regular values $c_1,\ldots,c_7$ (taken as in the discussion for the general case). Each entry of the table will contain the value of the integral over $M\sq_{c_j}S^1$ of $\ka_{c_j}$ applied to the element indicated by the column. From Kalkman's formulas, the only thing we have to do to obtain the row labelled by $j$ in $T_2$ is to add up the rows labelled $j,\ldots,7$ in $T_1$:
\[
\begin{array}{|c||r|r|r|r|r|r|r|}
\hline
T_2 & t^2 & x_0^2 & x_0x_1 & x_0x_2 & x_1^2 & x_1x_2 & x_2^2 \\ \hhline{|=||=|=|=|=|=|=|=|}
  1 &-1/8 &     0 &      0 &      0 &     0 &      0 &   0   \\ \hline
  2 &   0 &   1/8 &      0 &      0 &     0 &      0 &   0   \\ \hline
  3 & 1/8 &   1/8 &      0 &      0 &   1/2 &      0 &   0   \\ \hline
  4 &   0 &     0 &   -1/4 &      0 &     0 &      0 &   0   \\ \hline
  5 & 1/8 &     0 &   -1/4 &      0 &     0 &      0 &   2   \\ \hline
  6 &   0 &  -1/8 &   -1/4 &   -1/2 &     0 &      0 &   0   \\ \hline
  7 &-1/8 &  -1/8 &   -1/4 &   -1/2 &  -1/2 &     -1 &  -2   \\ \hline
\end{array}
\]
According to the general discussion the basis chosen for $H^2(\CC P^1\times\CC P^1\times\CC P^1)$ is $\{t,x_0,x_1,x_2\}$. The table $T_2$ contains all the necessary information to compute the matrices $A_j^2$ that represent
\[
I_{c_j}^2:H^2_{S^1}(\CC P^1\times\CC P^1\times\CC P^1)\ot H^2_{S^1}(\CC P^1\times\CC P^1\times\CC P^1)\longrightarrow\QQ
\]
in this basis. If $T_2(j,k)$ represents the $(j,k)$-th entry of $T_2$ we have that
\[
A_j^2=
\left[
\begin{array}{rrrr}
                  T_2(j,1) & -T_2(j,2) & -\frac{1}{2}\cdot T_2(j,5) & -\frac{1}{4}\cdot T_2(j,7)\\
                 -T_2(j,2) &  T_2(j,2) &                   T_2(j,3) &                   T_2(j,4)\\
-\frac{1}{2}\cdot T_2(j,5) &  T_2(j,3) &                   T_2(j,5) &                   T_2(j,6)\\
-\frac{1}{4}\cdot T_2(j,7) &  T_2(j,4) &                   T_2(j,6) &                   T_2(j,7)\\
\end{array}
\right],
\]
where the coefficients in the first row and column are determined by the relations that define the cohomology ring.

We get the matrices
\[
A_1^2=\frac{1}{8}
\left[\begin{array}{rrrr}
-1 & 0 & 0 & 0\\
0 & 0 & 0 & 0\\
0 & 0 & 0 & 0\\
0 & 0 & 0 & 0\\
\end{array}\right],\ 
A_2^2=\frac{1}{8}
\left[\begin{array}{rrrr}
0 & -1 & 0 & 0\\
-1 & 1 & 0 & 0\\
0 & 0 & 0 & 0\\
0 & 0 & 0 & 0\\
\end{array}\right],\
\]
\[ 
A_3^2=\frac{1}{8}
\left[\begin{array}{rrrr}
1 & -1 & -2 & 0\\
-1 & 1 & 0 & 0\\
-2 & 0 & 4 & 0\\
0 & 0 & 0 & 0\\
\end{array}\right],\ 
A_4^2=\frac{1}{8}
\left[\begin{array}{rrrr}
0 & 0 & 0 & 0\\
0 & 0 & -2 & 0\\
0 & -2 & 0 & 0\\
0 & 0 & 0 & 0\\
\end{array}\right],\ 
\]
\[
A_5^2=\frac{1}{8}
\left[\begin{array}{rrrr}
1 & 0 & 0 & -4\\
0 & 0 & -2 & 0\\
0 & -2 & 0 & 0\\
-4 & 0 & 0 & 16\\
\end{array}\right],\ 
A_6^2=\frac{1}{8}
\left[\begin{array}{rrrr}
0 & 1 & 0 & 0\\
1 & -1 & -2 & -4\\
0 & -2 & 0 & 0\\
0 & -4& 0 & 0\\
\end{array}\right],\
\]
\[ 
A_7^2=\frac{1}{8}
\left[\begin{array}{rrrr}
-1 & 1 & 2 & 4\\
1 & -1 & -2 & -4\\
2 & -2 & -4 & -8\\
4 & -4 & -8 & -16\\
\end{array}\right].
\]

Now one can prove that a matrix $(B^2)^t$ is a common pseudoinverse of $A_1^2,\ldots,A_7^2$ if and only if $B^2$ is of the form
\[
B^2=-
\left[\begin{array}{rrrc}
8 & 8 & 4 & a\\
8 & 0 & 4 & 2\\
4 & 4 & 0 & 1\\
b & 2 & 1 & \frac{a}{4}+\frac{b}{4}-1
\end{array}\right].
\]
The rows of $T_2$ are precisely the matrices $A_1^4,\ldots,A_7^4$. The only matrix $(B^4)^t$ that is pseudoinverse to all of them satisfies
\[
B^4=-
\left[\begin
{array}{ccccccc}
8 & 8 & 4 & 2 & 4 & 2 & 1
\end{array}
\right].
\]
Finally, the matrices $A_1^0,\ldots,A_7^0$ are given by transposing the rows of $T_2$ and hence the only matrix $(B^0)^t$ that is a common pseudoinverse to all of them is precisely $B^4$ (this is the symmetry explained in Remark \ref{rema:symmetry}).

Putting together all the information we obtained, we conclude that a class is a global biinvariant diagonal class if and only if it is represented by the matrices $B^0=(B^4)^t$, $B^2$, $B^4$ in the chosen bases. In particular the affine space of global biinvariant diagonal classes has dimension $2$, which proves the non-uniqueness of such classes in general.
\end{exam}

\begin{rema}
For the actions on complex spaces we described, having isolated fixed points maximizes the number of critical values, so it is the situation in which we have the most number of pairings $I_{c^i}^q$. The larger this number, the more difficult it is for a class to be a global biinvariant diagonal class, because the matrices representing it have to be common pseudoinverses to a larger number of matrices. Thinking the other way around, actions with isolated fixed points are the ones in which the space of global biinvariant diagonal classes is smallest. The example we computed shows that even in this optimal situation uniqueness is not achieved.
\end{rema}

%% file: Epilogue.tex
\chapter*{Epilogue}\label{ch:epilogue}

\addcontentsline{toc}{chapter}{Epilogue}

\fancyhead[RE,LO]{Epilogue}


This epilogue is a concession to the not-so-methodical part of the usual mathematical thought process. The classical definition-proposition-proof pattern is put aside --up to a certain point-- and a more descriptive tone is adopted.  The reader should expect the kind of instances that mathematicians sometimes think or verbalise more than those that they usually write down. 

The goal of these few final pages is to review some of the ideas that appeared throughout the thesis in a less formal but hopefully more intuitive way. These intuitions and thoughts give some clues on how our results could be extended to more general situations.

\subsection*{Biinvariant diagonal classes in Kirwan's original set-up}

Consider a Hamiltonian space $(M^{2m},\om,S^1,\mu)$ and denote by $X$ the vector field generated by the infinitesimal action. Take $J$ an invariant almost complex structure compatible with $\om$ and denote by $\xi^{JX}$ the flow of the vector field $-JX$. Already in the introduction we explained that the geometric idea behind the construction of a global biinvariant diagonal class was to consider the submanifold
\[
\De_{S^1}=\{(p,q)\in\MTM:\ \exists t\in\RR,\ \al\in S^1\ \mbox{s.t.}\ q=\al\cdot\xi^{JX}_t(p)\},
\]
which satisfies the properties
\begin{enumerate}
\item It is $\STS$-invariant.
\item It has dimension $2m+2$.
\item The intersection $\De_{S^1}\cap(\mu^{-1}(c)\times\mu^{-1}(c))$ is the preimage of the diagonal under the projection $\mu^{-1}(c)\times\mu^{-1}(c)\rightarrow M\sq_cS^1\times M\sq_cS^1.$
\end{enumerate}

We also explained that property 1 allows to define an equivariant cohomology --instead of ordinary cohomology-- class, that property 2 makes this class have the right dimension and that property 3 makes the class to be mapped to the Poincaré dual of the diagonal class. The big deal, however, was to find the right compactification for $\De_{S^1}$: we used broken gradient lines, which require certain transversality conditions. In order to preserve these conditions and get invariance at the same time we needed multivalued perturbations.

One natural question is how these techniques could be adapted to the case originally studied in \cite{Kir}, where the Hamiltonian space is of the form $(M^{2m},\om,G,\mu)$ with $G$ any compact and connected Lie group. Let $d=\dim G$. Recall also from the introduction of this thesis that if $0$ is a regular value of $\mu$, the action of $G$ on $\mu^{-1}(0)$ is locally free and $M\sq G=\mu^{-1}(0)/G$ is an orbifold. The action of $G$ in general does not restrict to any other regular level because $\mu$ is an equivariant map and not invariant as it happens when $G$ is abelian. Fix a $G$-invariant inner product on $\mg$ to identify $\mg$ and $\mg^*$. Then $\mu$ can be thought as taking values on $\mg$. Using the norm defined by the inner product we define a $G$-invariant smooth function $f:M\rightarrow\RR$ by $f(p):=\|\mu(p)\|^2$. Let $J$ be an invariant almost complex structure compatible with $\om$ and let $\rho_J=\om(\cdot,J(\cdot))$ be the corresponding Riemannian metric. Finally let $\nabla^Jf$ be the gradient vector field of $f$ with respect to $\rho_J$ and denote by $\xi^J$ the flow of $-\nabla^Jf$. In \cite{Ler} it is proved that for all $p\in M$, $\lim_{t\rightarrow +\infty}\xi^J_t(p)$ consists of unique point, so
\[
S_0=\{p\in M:\ \lim_{t\rightarrow +\infty}\xi^J_t(p)\in\mu^{-1}(0)\}
\]
is the stable set of the critical level $f^{-1}(0)=\mu^{-1}(0)$ of $f$.  Let $\pi:S_0\rightarrow\mu^{-1}(0)$ be the projection given by following the flow $\xi^J$.

\begin{epiprop}\label{prop:rightdelta} 
The set
\[
\De_G=\{(p,q)\in S_0\times S_0:\ \exists g\in G\ \mbox{s.t.}\ \pi(q)=g\cdot\pi(p)\}
\]
satisfies the following properties:
\begin{enumerate}
\item It is $\GTG$-invariant.
\item It is a $(2m+2d)$-dimensional submanifold of $\MTM$.
\item The intersection $\De_G\cap(\mu^{-1}(0)\times\mu^{-1}(0))$ is the preimage of the diagonal under the projection $\mu^{-1}(0)\times\mu^{-1}(0)\rightarrow M\sq G\times M\sq G$.
\end{enumerate}
\end{epiprop}

\begin{proof}
That $\De_G$ is $\GTG$-invariant is a consequence of the equivariance of $\pi$: if $\pi(q)=g\pi(p)$ and $g',g''\in G$, then
$\pi(g''q)=g''g(g')^{-1}\pi(g'p)$.

We will see later that $S_0$ is a smooth manifold and that $\pi$ is a submersion. Since $f^{-1}(0)$ is the minimum of the function $f$ we get $\dim S_0=\dim M=2m$. Moreover the smooth map
\[
\begin{array}{rccc}
F: &S_0\times S_0&\longrightarrow &M\sq G\times M\sq G\\
   &(p,q)&\longmapsto &(G\cdot\pi(p),G\cdot\pi(q))
\end{array}
\]
is transverse to the diagonal and thus $\De_G=F^{-1}(\De_{M\ssq G\times M\ssq G})$ is a smooth submanifold of dimension $4m-(2m-2d)=2m+2d$. This holds even with $M\sq G$ being an orbifold \cite{BoBr}.

The last property follows from the observation that if $p,q\in\mu^{-1}(0)$ then $\pi(p)=p$ and $\pi(q)=q$.
\end{proof}

This proposition suggests that $\De_G$ is a good candidate to construct a biinvariant diagonal class. Again compactification is the most difficult part. For $S^1$-actions we had the advantage that the moment map is a Morse-Bott function and therefore gradient lines converge to broken gradient lines (recall Theorem \ref{theo:stratification}). Unfortunately, the function $f$ is not Morse-Bott in general so we cannot apply this. The problem arises because we have limited control on how gradient lines converge to critical points: the connected components of the critical set are not smooth manifolds in general, although its stable sets are  \cite[10.16]{Kir} and following the flow gives a retraction from each stable manifold to its corresponding critical set \cite{Ler}. Similar results do not seem to exist for unstable sets, which is a serious drawback if one wants to apply a technique similar to broken gradient lines.
 
Another approach would be studying if the inclusion $\io:\De_G\hookrightarrow\MTM$ is a pseudocycle and then define the cohomology class by intersection pairings of pseudocycles. For $\io$ to be a pseudocycle we need to have control over its omega-limit set. In \cite{Kir} the stable manifolds of the flow of $f$ are well studied and it is proven that they define a stratification of $M$ in a sense \cite[2.11]{Kir}. Moreover all the strata are of even dimension, so the non-dense strata have all at least codimension $2$. From this stratification one could construct a stratification of the closure of $\De_G$. However, even the stratification on $M$ is not known to have any good convergence behaviour like e.g. being Whitney \cite{GoMa}. This was already pointed out by Kirwan in the footnote at the end of the introduction chapter of \cite{Kir}.

However we still have many good reasons to assert that $\De_G$ is a sensible choice. We shall prove some results to support this statement. We will explain our set-up first: let ${\bf G}=G\exp(i\mg)$ be a reductive complex Lie group with  $G$ a maximal compact subgroup and let $(M,\om,J)$ be a Kähler manifold. Suppose that ${\bf G}$ acts holomorphically on the complex manifold $(M,J)$ and that the induced action of $G$ on the symplectic manifold $(M,\om)$ is Hamiltonian. Let $\mu$ be the moment map for this Hamiltonian action and define $f=\|\mu\|^2$ as above by means of an invariant inner product on $\mg$. In \cite[7.4]{Kir} it is proved that $S_0=\{{\bf g}\cdot p:\ {\bf g}\in{\bf G},\  p\in\mu^{-1}(0)\}$ and that there is a natural homeomorphism
\[
M\ssq G\simeq S_0/{\bf G}.
\]
Hence there is a natural correspondence between $\De_{M\ssq G\times M\ssq G}$ and $\De_{S_0/{\bf G}\times S_0/{\bf G}}$.  We have the following result:

\begin{epiprop}\label{DeKah}
In the described Kähler setting, $\De_G$ is the preimage of the diagonal under the projection
\[
S_0\times S_0\rightarrow S_0/{\bf G}\times S_0/{\bf G},
\]
i.e. $\De_G=\{(p,q)\in S_0\times S_0:\ \exists {\bf g}\in{\bf G}\ \mbox{s.t.}\ q={\bf g}\cdot p\}$.
\end{epiprop}

We prove the proposition as consequence of a set of lemmas which are all proved, at least implicitly, in \cite[\S 6,7]{Kir} (see also \cite[8.3]{MFK}). For $\eta\in\mg$ let $X^\eta_p=\frac{d}{dt}_{|t=0}\exp(t\eta)p$ be the vector field on $M$ induced by the infinitesimal action of $\eta$. The Lie algebra of ${\bf G}$ is $\mg\oplus i\mg$ and the vector field induced by the infinitesimal action of $i\eta$ is $JX^\eta$.

\begin{epilemm}\label{flow}
For all $p\in M$, $(\na^Jf)_p=2JX^{\mu(p)}_p$.
\end{epilemm}

\begin{proof} Take $\eta_1,\ldots,\eta_d$ an orthonormal basis of $\mg$ and define the functions $\mu_1,\ldots,\mu_d$ on $M$ by $\mu_k(p)=\la\mu(p),\eta_k\ra$. Then we have
\[
\mu(p)=\sum_{k=1}^d\mu_k(p)\eta_k\ \mbox{and}\ f(p)=\sum_{k=1}^d(\mu_k(p))^2.
\] 
Therefore 
\[
d_pf=2\sum_{k=1}^d\mu_k(p)d_p\mu_k=2\sum_{k=1}^d\mu_k(p)\om_p(X^{\eta_k}_p,\cdot)=\om_p(2X^{\mu(p)}_p,\cdot).
\] 
But on the other hand 
\[
d_pf=(\rho_J)_p(\cdot,(\na^Jf)_p)=\om_p(\cdot,J(\na^Jf)_p).
\]
Hence $-J(\na^Jf)_p=2X^{\mu(p)}_p$ and applying $J$ to both sides the result is obtained.\end{proof}

\begin{epilemm}\label{curve}
For every $p\in M$, the integral curve of $-\na^Jf$ through $p$ is contained in the ${\bf G}$-orbit of $p$. That is $\{\xi^J_t(p):\ t\in\RR\}\subseteq {\bf G}\cdot p$.
\end{epilemm}

\begin{proof} We have that 
\[
\frac{d}{dt}_{|t=0}\xi_t^J(p)=-(\na^Jf)_{\xi_t^J(p)}=-2JX^{\mu(\xi_t^J(p))}_{\xi^J_t(p)}
\]
is an element of $\in T_{\xi^J_t(p)}({\bf G}\cdot p)$ because $JX^{\mu(\xi_t^J(p))}$ is induced by the infinitesimal action of $i\mu(\xi_t^J(p))\in i\mg$.
\end{proof}

\begin{epilemm}\label{orbit} 
For every $p\in\mu^{-1}(0)$, ${\bf G}\cdot p\cap \mu^{-1}(0)=G\cdot p$.
\end{epilemm}

\begin{proof}
To begin with, we will see that if $\eta\in\mg$ and $\exp(i\eta)p\in\mu^{-1}(0)$, then $\exp(i\eta)p=p$. Define $h:\RR\rightarrow\RR$ by
\[
h(t)=\la\mu(\exp(it\eta)p),\eta\ra.
\] 
Then $h(0)=\la\mu(p),\eta\ra=0$ and $h(1)=\la\mu(\exp(i\eta)p),\eta\ra=0$, because $p$ and $\exp(i\eta)p$ are in $\mu^{-1}(0)$. By Rolle's theorem there exists $s\in(0,1)$ such that $h'(s)=0$. The function $h$ is the composition of the curve $\ga:t\mapsto\exp(it\eta)p$ with the function $\la\mu(\cdot),\eta\ra$. Note that $\ga'(t)=JX^{\eta}_{\ga(t)}$ and that $d_q\la\mu(\cdot),\eta\ra=\om_q(X_q^\eta,\cdot)$. Then, by the chain rule
\[
\begin{array}{rcccl}
0 & = & h'(s) & = & d_{\ga(s)}\la\mu(\cdot),\eta\ra(\ga'(s))\\
  &   &       & = & \om_{\ga(s)}(X^\eta_{\ga(s)},JX^\eta_{\ga(s)})\\
  &   &       & = & (\rho_J)_{\ga(s)}(X^\eta_{\ga(s)},X^\eta_{\ga(s)})\\
  &   &       & = & \|X^\eta_{\ga(s)}\|^2
\end{array}
\] 
Hence $JX^\eta_{\exp(is\eta)p}=0$, which means that the $1$-parameter subgroup $\exp i\eta\RR$ fixes $\exp(is\eta)p$ and so it also fixes $p$. In particular $\exp(i\eta)p=p$, as we wanted to show.

Now let ${\bf g}\in{\bf G}$ be such that ${\bf g}p\in\mu^{-1}(0)$. Then, since ${\bf G}=G\exp(i\mg)$, we have that ${\bf g}p=g\exp(i\eta)p$ for certain $g\in G$ and $\eta\in\mg$. We have that $\exp(i\eta)p\in\mu^{-1}(0)$ because $\mu^{-1}(0)$ is $G$-invariant. From what we have seen previously we know that $\exp(i\eta)p=p$ and then ${\bf g}p=gp$.
\end{proof}

\begin{epilemm}\label{limit}
For every $p\in S_0$, $\pi(p)\in{\bf G}\cdot p$.
\end{epilemm}

\begin{proof} 
By Lemma \ref{curve}, $\{\xi^J_t(p):\ t\in\RR\}\subseteq{\bf G}\cdot p$, so that $\pi(p)$ is in the closure $\overline{{\bf G}\cdot p}$. Hence, either $\pi(p)\in{\bf G}\cdot p$ or $\dim{\bf G}\cdot\pi(p)<\dim {\bf G}\cdot p$. Let us see why the second fact is impossible: since $\pi(p)\in\mu^{-1}(0)$, we know that the stabiliser $G_{\pi(p)}$ is finite, and by Lemma \ref{orbit} so is ${\bf G}_{\pi(p)}$. Therefore 
\[
\dim {\bf G}\cdot\pi(p)=\dim {\bf G}-\dim {\bf G}_{\pi(p)}=\dim {\bf G}\geq\dim {\bf G}\cdot p.\qedhere
\]
\end{proof}

Finally we prove Proposition \ref{DeKah}:

\begin{proof}
Let $p,q\in S_0$. By Lemma \ref{limit} there exist ${\bf k},{\bf h}\in{\bf G}$ such that $\pi(p)={\bf k}p$ and $\pi(q)={\bf h}q$. We check both inclusions:

\noindent $\subseteq)$ Suppose that $(p,q)\in\De_G$, so there is $g\in G$ such that $\pi(q)=g\pi(p)$. Then
\[
q={\bf h}^{-1}\pi(q)={\bf h}^{-1}g\pi(p)={\bf h}^{-1}g{\bf k}\cdot p\in{\bf G}\cdot p.
\]

\noindent $\supseteq)$ Suppose that there exists ${\bf g}\in {\bf G}$ such that $q={\bf g}p$. Then
\[
\pi(q)={\bf h}q={\bf h}{\bf g}p={\bf h}{\bf g}{\bf k}^{-1}\pi(p)\in{\bf G}\cdot\pi(p).
\]
We get $\pi(q)\in{\bf G}\cdot\pi(p)$, so by Lemma \ref{orbit} $\pi(q)\in G\cdot\pi(p)$.
\end{proof}

When the Kähler manifold is a nonsingular complex projective variety and ${\bf G}$ acts linearly, the corresponding GIT quotient coincides precisely with $S_0/{\bf G}$ (see \cite[\S 8]{Kir}, \cite[ch. 8]{MFK} for the details). In this algebraic set-up, thanks to Proposition \ref{DeKah}, it turns out that $\De_G$ is very similar to set that Mundet i Riera uses at the very beginning of \cite{Mun} as an inspiration to study the symplectic 
set-up.

At this point one believes that the amount of evidence clearly points to $\De_G$ as the right object to be studied if one wants to construct a biinvariant diagonal class in Kirwan's original set-up. As it has already been explained, the remaining challenge is to find a right way to compactify it.

\subsection*{Biinvariant diagonal classes in other set-ups}

Kirwan surjectivity it has nowadays been generalised to several other frameworks:
\begin{itemize}
\item If $\CA$ is the space of connections of a $G$-principal bundle over a Riemann surface, in \cite{AtBo1} it is proved that $\CA$ may be treated as infinite dimensional symplectic manifold and that the action of the gauge group $\GG$ on $\CA$ is Hamiltonian. The corresponding moment map $\mu$ sends a connection $A$ to its curvature. Therefore $\mu^{-1}(0)$ is the set of flat connections. The symplectic quotient $\mathcal{A}\sq\mathcal{G}$ is a finite dimensional orbifold that we call the {\it moduli space of flat connections}. This work is actually previous to \cite{Kir} and in it Atiyah and Bott already consider the map $\|\mu\|^2$. 

\item Consider a Lie group $G$ and let $LG$ be its loop group. Let $X$ be a symplectic Banach manifold on which $LG$ acts in a Hamiltonian way. Then we have a moment map $\mu:X\rightarrow L\mg^*$ and the symplectic quotient $X\sq LG$, which coincides with $\mu^{-1}(0)/G$, is a finite dimensional orbifold. By means of the inclusion $\mu^{-1}(0)\hookrightarrow X$ and the Cartan isomorphism we get a Kirwan map $H^*_G(X)\rightarrow H^*(X\sq LG)$. In \cite{BTW} it is proved that this infinite dimensional version of the Kirwan map is surjective. 

Studying Symplectic Banach manifolds with Hamiltonian $LG$-actions has also another motivation: they are in correspondence with quasi-hamiltonian $G$-spaces as described in \cite{AMM}. These spaces arise from the idea of defining moment maps taking values on $G$ instead of $\mg^*$.

\item In \cite{HaLa} a $K$-theoretic version of Kirwan surjectivity is given. Similarly in \cite{HaSe} the result for loop groups from \cite{BTW} is also translated to $K$-theory.
\end{itemize}

In all these cases the techniques used are similar to the ones in \cite{Kir}. It is then reasonable to wonder how techniques similar to biinvariant diagonal classes can be adapted to those other situations. One approach to the loop group case that might be promising is the following:

According to the results in \cite{BTW}, if $LG$ acts on a Banach manifold $X$ in a Hamiltonian way there are finite dimensional manifolds $X_n$ approximating $X$ and smooth maps $f_n:X_n\rightarrow\RR$ approximating $f=\|\mu\|^2$ such that the functions $f_n$ are minimally degenerate in the sense of \cite[\S 10]{Kir} and hence they induce surjective maps $H_G^*(X_n)\rightarrow H_G(f_n^{-1}(0))$. It is also shown that all these maps being surjective implies Kirwan surjectivity. 

In the simpler case $LG=LS^1$, one might try to construct approximations $\mu_n:X_n\rightarrow\RR$ of $\mu$ and see whether they induce surjections 
\[
\ka_n:H_{S^1}^*(X_n)\rightarrow H_{S^1}(\mu_n^{-1}(0)).
\]
Although the manifolds $X_n$ are not expected to be symplectic as explained in \cite{BTW}, it seems plausible that the maps $\mu_n$ can be taken to be Morse-Bott anyway. Then the techniques developed in this thesis should be reproducible in order to construct a biinvariant diagonal class $\de_n\in H^{2\dim X_n-2}_{\STS}(X_n\times X_n)$ for each $\ka_n$. An obvious question then would be what kind of object do the $\de_n$ induce on $H^*_{\STS}(X\times X)$ and how to interpret it from the point of view of quasi-hamiltonian $S^1$-spaces.